\documentclass[11pt,           
english                        
]{article}

\usepackage[english,german,strings]{babel}     
\usepackage[utf8]{inputenc}    
\usepackage[T1]{fontenc}       
\usepackage{longtable}         
\usepackage{exscale}           
\usepackage[final]{graphicx}   
\usepackage[sort]{cite}        
\usepackage{array}             
\usepackage{fancyhdr}          
\usepackage[a4paper]{geometry} 
\usepackage[multiuser]{fixme}  
\usepackage{xspace}            
\usepackage{tikz}              
\usepackage{tikzsymbols}       
\usepackage{nchairx}           
\usepackage{wasysym}           
\usepackage[sectionbib         
           ]{chapterbib}       
\usepackage{appendix}
\usepackage[expansion=false    
           ]{microtype}        
\usepackage[nottoc]{tocbibind} 
\usepackage[backref=page,      
           hypertexnames=false,
           final=true,         
           pdfpagelabels       
           ]{hyperref}         
\usepackage{breakurl}
\newcommand{\chairxauthorbibfont}{\textsc}

\newcommand{\chairxtitlebibfont}{\textit}

\newcommand{\chairxseriesbibfont}{\textit}

\usepackage{tikz-cd}       
\usepackage{xcolor} 
\usetikzlibrary{arrows,calc,through,backgrounds,matrix,
	decorations.pathmorphing,positioning,babel}    
	
\newcommand*{\longrightsquigarrow}[1]{\ 
{\tikz \draw [->,
line join=round,
decorate, decoration={
    zigzag,
    segment length=4,
    amplitude=.9,
    post=lineto,
    post length=2pt
}] (0,0) -- (#1,0);}\ }
\newcommand{\tw}{\mathsf{tw}}   
\newcommand{\Mc}{\mathsf{MC}}   
\newcommand{\Def}{\mathsf{Def}} 
\newcommand{\poly}{{\scriptscriptstyle{\mathrm{poly}}}}    
\newcommand{\Tpoly}{T_\poly}  
\newcommand{\Dpoly}{D_\poly}  
 
\newcommand{\sh}{{\scriptscriptstyle{\mathrm{sh}}}}  

\title{An Introduction to $L_\infty$-Algebras and their Homotopy Theory}
\author{
  \textbf{Andreas Kraft}\thanks{\texttt{akraft@impan.pl}},\\[0.3cm]
	 Institute of Mathematics  \\
	  Polish Academy of Sciences \\
		ul. \'Sniadeckich 8 \\ 
		00-656 Warsaw \\
		Poland  \\[0.5cm]
  \textbf{Jonas Schnitzer}\thanks{\texttt{jonas.schnitzer@math.uni-freiburg.de}},\\[0.3cm]
  Department of Mathematics\\
   University of Freiburg\\
   Ernst-Zermelo-Straße 1\\
	D-79104 Freiburg\\
   Germany
}

\begin{document}
\selectlanguage{english}

\maketitle

\abstract{
In this review we give a detailed introduction to the theory of 
(curved) $L_\infty$-algebras and $L_\infty$-morphisms. In particular, we 
recall the notion of (curved) Maurer-Cartan elements, their equivalence classes 
and the twisting procedure. The main focus is then the study of the 
homotopy theory of $L_\infty$-algebras and $L_\infty$-modules. 
In particular, one can interpret $L_\infty$-morphisms and morphisms of 
$L_\infty$-modules as Maurer-Cartan elements in certain $L_\infty$-algebras, 
and we show that twisting the morphisms with equivalent Maurer-Cartan 
elements yields homotopic morphisms.
}

\tableofcontents

%
%
\section{Introduction}

$L_\infty$-algebras (also called strong homotopy Lie algebras or SH Lie 
algebras) are a generalization of differential graded Lie algebras 
(DGLAs), where the Jacobi identity only holds up to compatible higher 
homotopies. They were introduced in 
\cite{schlessinger.stasheff:1985a,stasheff:1992a,lada.stasheff:1993a}, 
and they appeared at first in a supporting role in deformation theory. 

The idea is the following: In a usual DGLA 
$(\liealg{g},\D,[\argument{,}\argument])$ one has a cochain
differential $\D$ on a $\mathbb{Z}$-graded vector space $\liealg{g}$ 
and a compatible graded
Lie bracket $[\argument{,}\argument]$, i.e. $\D$ is satisfies the graded 
Leibniz rule. In an $L_\infty$-algebra $(L,\{l_n\}_{n\in \mathbb{N}})$ 
one has instead a whole collection of antisymmetric maps 
$l_n \colon \Anti^n L \to L$ of degrees $2-n$, where $n\geq 1$, satisfying 
certain compatibilities that generalize those of DGLAs. These 
maps are called Taylor coefficients or structure maps. In particular, 
the binary map $l_2$ that corresponds to the Lie bracket 
$[\argument{,}\argument]$ satisfies the Jacobi identity only up to terms depending on the homotopy $l_3$, and so on. 
Using the Koszul duality between DGLAs and graded cocommutative coalgebras, 
it is shown in \cite{lada.markl:1994a} that $L_\infty$-algebra 
structures $\{l_n\}_{n\in\mathbb{N}}$ on a graded vector space $L$ are in 
one-to-one correspondence with codifferentials $Q$ of degree $1$ on the cofree 
cocommutative conilpotent coalgebra $\Sym(L[1])$ generated by the 
$L[1]$, where the 
degree is shifted by one. This dual coalgebraic formulation goes 
actually back to the BV-BRST formalism \cite{batalin.vilkovisky:1983a,batalin.fradkin:1983a}, see also \cite{stasheff:1997b}. 

As mentioned above, $L_\infty$-algebras play an important role in deformation 
theory. Here the basic philosophy is that, over a field of characteristic 
zero, every deformation problem is governed by a DGLA or more generally 
an $L_\infty$-algebra via solutions of the Maurer-Cartan equation modulo 
equivalences, see e.g. \cite{kontsevich.soibelman:book,manetti:2005a}. 
Staying for the moment for simplicity in the context of DGLAs, a 
Maurer-Cartan element $\pi$ in $(\liealg{g},\D,[\argument{,}\argument])$ is 
an element of degree one satisfying the Maurer-Cartan equation
\begin{equation*}
  0
	=
	\D \pi 
	+ 
	\frac{1}{2}[\pi,\pi].
\end{equation*}
This can be generalized to $L_\infty$-algebras, and there are suitable 
notions of gauge and homotopy equivalences of Maurer-Cartan elements, 
and in deformation theory one is interested in the 
transformation groupoid of the gauge action, the also called 
\emph{Goldman-Millson groupoid} or \emph{Deligne groupoid} 
\cite{goldman.millson:1988a}. Moreover, one can use 
Maurer-Cartan elements to twist the $L_\infty$-structure, i.e. change 
the $L_\infty$-structure in a certain way. For example, for a DGLA 
$(\liealg{g},\D,[\argument{,}\argument])$ with Maurer-Cartan element 
$\pi$, the twisted DGLA takes the following form 
$(\liealg{g}, \D + [\pi,\argument],[\argument{,}\argument])$, where 
the Maurer-Cartan equation implies that $\D + [\pi,\argument]$ 
squares to zero.

One nice feature of $L_\infty$-algebras is that they allow a 
more general notion of morphisms, so-called $L_\infty$-morphisms:  
an $L_\infty$-morphism from an $L_\infty$-algebra $(L,Q)$ to $(L',Q')$ is  
just a coalgebra morphism between the corresponding coalgebras 
$\Sym(L[1])$ and $\Sym(L'[1])$ that commutes with the codifferentials. 
In particular, $L_\infty$-morphisms are generalizations of Lie algebra morphisms 
and they are still compatible with Maurer-Cartan elements. Moreover, 
there is a notion of $L_\infty$-quasi-isomorphism generalizing the notion 
of quasi-isomorphisms between DGLAs. One can prove that 
such $L_\infty$-quasi-isomorphisms admit quasi-inverses, and 
that $L_\infty$-quasi-isomorphisms induce bijections on the 
equivalence classes of Maurer-Cartan elements, which is important for 
deformation theory. One 
can also twist $L_\infty$-morphisms by Maurer-Cartan elements, which gives 
$L_\infty$-morphisms between the twisted $L_\infty$-algebras. 
Note that the notion of $L_\infty$-algebras and 
$L_\infty$-morphisms can be generalized to algebras over 
general Koszul operads, see e.g. \cite{loday.vallette:2012a}. 
In addition, as one would expect, the generalization from DGLAs to $L_\infty$-algebras 
also leads to a generalization of the representation theory, i.e. from 
DG Lie modules to $L_\infty$-modules.

One famous deformation problem solved by $L_\infty$-algebraic 
techniques is the deformation quantization problem of Poisson manifolds, 
which was solved by Kontsevich's celebrated formality theorem 
\cite{kontsevich:2003a}, see also \cite{dolgushev:2005a,dolgushev:2005b} for the globalization of this result and the 
invariant setting of Lie group actions. More explicitly, the formality 
theorem provides an $L_\infty$-quasi-isomorphism between 
the differential graded Lie algebra of polyvector fields 
$\Tpoly(M)$ and the polydifferential operators $\Dpoly(M)$ on a smooth 
manifold $M$. As such, it induces a one-to-one correspondence 
between equivalence classes of Maurer-Cartan elements, i.e. 
between equivalence classes of (formal) Poisson structures 
and equivalence classes of star products. Using 
the language of $L_\infty$-modules, there has also been proven 
a formality theorem for Hochschild chains 
\cite{dolgushev:2006a,shoikhet:2003a}. 
Moreover, in the last years many additional developments have taken place, 
see e.g. \cite{calaque:2005a,calaque:2007a, liao:2019a}. 

Apart from that, $L_\infty$-algebras can be used to 
describe many more geometric deformation problems, e.g. deformations of 
complex manifolds \cite{manetti:2004a},  deformations of foliations 
\cite{vitagliano:2014a}, deformations of Dirac structures 
\cite{gualtieri.matviichuk.scott:2020a}, and many more. In addition, $L_\infty$-algebras are 
also an important tool in homological reduction theory, following the 
BV-BRST spirit, see e.g. \cite{schaetz:2008a,cattaneo.felder:2007a,esposito.kraft.schnitzer:2020a:pre,esposito.kraft.schnitzer:2022a:pre}, 
and in physics, where they occur for example in string theory and
in quantum field theory.

Another important observation from \cite{dolgushev:2007a} is that 
$L_\infty$-morphisms themselves correspond to 
Maurer-Cartan elements in a certain convolution-like $L_\infty$-algebra, 
which gives a way to speak of homotopic $L_\infty$-morphisms in 
the case of equivalent Maurer-Cartan elements. 
Homotopic $L_\infty$-morphisms share many features: for example, 
an $L_\infty$-morphism that is homotopic to an $L_\infty$-quasi-isomorphism 
is automatically an $L_\infty$-quasi-isomorphism itself, and 
homotopic $L_\infty$-morphisms induce the same maps on the equivalence 
classes of Maurer-Cartan elements. The second observation was used in 
\cite{bursztyn.dolgushev.waldmann:2012a} to prove that the globalization 
of the Kontsevich formality by Dolgushev \cite{dolgushev:2005a,dolgushev:2005b} 
is, at the level of equivalence classes, independent of the chosen connection.

Finally, note that there is a generalization of $L_\infty$-algebras 
to curved $L_\infty$-algebras $(L,\{l_n\}_{n\in \mathbb{N}_0})$, 
where one allows an additional zero-th structure map 
$l_0 \colon\mathbb{K} \to L[2]$, where $\mathbb{K}$ denotes the ground field. 
This corresponds to the curvature $l_0(1)\in L^2$. In this way, 
one can generalize the notion of curved Lie algebras $(\liealg{g},R,\D,[\argument{,}\argument])$, where the curvature $R\in \liealg{g}^2$ is closed, and 
where $\D^2=[R,\argument]$. In this curved setting, one can speak 
of curved Maurer-Cartan elements, 
and all the notions from the above (flat) $L_\infty$-algebras generalize 
to this setting. For example, the curved Maurer-Cartan equation in 
a curved Lie algebras reads
\begin{equation*}
  0
	=
	R +\D \pi + \frac{1}{2}[\pi,\pi].
\end{equation*}
In particular, one can now twist with general 
elements of degree one, not just with Maurer-Cartan elements, without 
leaving the setting of curved $L_\infty$-algebras. Moreover,  
one also obtains a more general notion of curved $L_\infty$-morphisms 
\cite{getzler:2018a}, that are no longer coalgebra morphisms as in the 
above (flat) setting, but only coalgebra morphisms up to a twist. 
Curved $L_\infty$-algebras play for example an important role 
in Tsygan's conjecture of an equivariant formality in deformation 
quantization \cite{tsygan:note}, where the fundamental vector 
fields of the Lie group action play the role of the curvature.

The aim of this paper is to give a review over the theory of 
curved $L_\infty$-algebra, Maurer-Cartan elements, $L_\infty$-modules, 
and their homotopy theories. We mainly collect known results from 
the literature, but also include some results that are to our 
knowledge new or at least folklore knowledge, but not yet written 
properly down. For example, following \cite{kraft:2021a} we show that 
$L_\infty$-algebras that are twisted with equivalent Maurer-Cartan 
elements are $L_\infty$-isomorphic, which we could only find in the 
literature for the case of DGLAs. Moreover, we study the homotopy 
theory of curved $L_\infty$-morphisms, where we show in particular 
that curved $L_\infty$-morphisms that 
are twisted with equivalent Maurer-Cartan elements are homotopic. 
This result for DGLAs allowed us in \cite{kraft.schnitzer:2021a:pre} 
to prove that Dolgushev's globalization procedure \cite{dolgushev:2005a,
dolgushev:2005b} of the Kontsevich formality \cite{kontsevich:2003a} 
with respect to different covariant derivatives yields homotopic 
$L_\infty$-quasi-isomorphisms. 

The paper is organized as follows: We start in 
Section~\ref{sec:Coalgebras} with the construction of cocommutative cofree 
conilpotent coalgebras and coderivations on them. This allows a 
compact definition of (curved) $L_\infty$-algebras and 
$L_\infty$-morphisms in 
Section~\ref{sec:DefLinfty} in terms of codifferentials $Q$ on 
$\Sym(L[1])$ and coalgebra morphisms commuting with the codifferentials. 
 In Section~\ref{sec:homotopytransfertheorem} we give explicit formulas for the homotopy transfer theorem, which provides a way to 
transfer $L_\infty$-structures along deformation retracts. There 
are many formulations and proofs for it, and in our special cases we 
use the symmetric tensor trick.
In Section~\ref{sec:MCandEquiv} we introduce the notion of 
Maurer-Cartan elements and compare the gauge equivalence 
of Maurer-Cartan elements in the setting of DGLAs and the more general 
notion of homotopy equivalence in $L_\infty$-algebras. Moreover, 
we recall the twisting procedure of curved $L_\infty$-algebras and 
of $L_\infty$-morphisms. In Section~\ref{sec:HomotopyTheoryLinftyAlgandMorph} 
we introduce the interpretation of $L_\infty$-morphisms as Maurer-Cartan 
elements and the notion of homotopic $L_\infty$-morphisms. 
We recall the homotopy classification of flat $L_\infty$-algebras 
and show that $L_\infty$-morphisms that are twisted with equivalent 
Maurer-Cartan elements are homotopic. In particular, we 
study curved $L_\infty$-morphisms, which are no longer coalgebra 
morphisms, but in some sense coalgebra morphisms up to  twist. 
They are still compatible with curved Maurer-Cartan elements and 
have an analogue homotopy theory as the strict $L_\infty$-morphisms in 
the flat case. Finally, we recall in 
Section~\ref{sec:LinftyModules} 
the notion of $L_\infty$-modules over $L_\infty$-algebras, 
$L_\infty$-module morphisms between them, and the 
corresponding notion of homotopy equivalence.

\textbf{Acknowledgements:}
  The authors are grateful to Severin Barmeier, Chiara Esposito and Luca Vitagliano 
	for many helpful comments and feedback which helped to improve the first version 
	of this review.

%
%
\section{Introduction to Coalgebras}
\label{sec:Coalgebras}

In this first section we want to recall the basic properties related to 
graded coalgebras. In view of $L_\infty$-algebras we are particularly interested 
in cocommutative conilpotent graded coalgebras. 
We mainly follow \cite{esposito:2015a,loday.vallette:2012a,waldmann:2019:note}.

\subsection{Reminder on (Free) Graded Algebras}

Before starting with coalgebras we recall some notions 
connected to graded algebras over some commutative unital ring 
$\mathbb{K}$, where we always assume $\mathbb{Q} \subseteq \mathbb{K}$. 
By grading we mean $\mathbb{Z}$-grading 
and we always apply the Koszul sign rule, 
e.g. for tensor products of homogeneous morphisms $\phi\colon V^\bullet 
\rightarrow \widetilde{V}^{\bullet + \abs{\phi}}$, $\psi \colon W^\bullet 
\rightarrow \widetilde{W}^{\bullet + \abs{\psi}}$ 
between graded $\mathbb{K}$-modules $V^\bullet,\widetilde{V}^\bullet,W^\bullet,
\widetilde{W}^\bullet$:
\begin{equation*}
  (\phi \otimes \psi)(v\otimes w)
	=
	(-1)^{\abs{\psi}\abs{v}} \phi(v)\otimes \psi(w),
\end{equation*}
where $v\in V^{\abs{v}}, w\in W$. Note that graded $\mathbb{K}$-modules with 
degree zero morphisms become a symmetric monoidal category with the 
graded tensor product and the graded switching map $\tau$. Recall that on homogeneous 
elements $v\in V, w\in W$ one has $\tau(v\otimes w) = (-1)^{\abs{v}\abs{w}} w \otimes v$. 
However, this monoidal structure is not compatible with the degree shift functor, 
where we write $V[i]^k = V^{i+k}$ and $V[i]^\bullet = (\mathbb{K}[i]\otimes V)^\bullet$.

By $(A^\bullet,\mu)$ we usually denote a graded algebra with 
associative product $\mu$ and by $1_A= 1 \colon \mathbb{K} \rightarrow A$ we 
denote a unit. Let $\phi \colon A^\bullet \rightarrow B^\bullet$ be a morphism 
of graded algebras, which is necessarily of degree zero. Then a graded derivation 
$D \colon A^\bullet \rightarrow B^{\bullet +k}$ of degree $k\in \mathbb{Z}$ 
along $\phi$ is a $\mathbb{K}$-linear homogeneous map $D$ such that
\begin{equation*}
  D \circ \mu_A
	=
	\mu_B \circ (\phi \otimes D + D \otimes \phi).
\end{equation*}
The tensor product $A\otimes B$ becomes a graded algebra with product
\begin{equation*}
  \mu_{A\otimes B}
	=
	(\mu_A \otimes \mu_B) \circ (\id_A \otimes \tau \otimes \id_B),
\end{equation*}
where $\tau$ is again the graded switching map. 
The notion of free objects will be used frequently in the following, whence 
we recall the free associative algebra generated by $\mathbb{K}$-modules, 
as well as the commutative analogues. We write $\mathrm{T}^k(V) = V^{\otimes k}$ for 
$k > 0 $ and $\mathrm{T}^0(V)=\mathbb{K}$ and obtain:

\begin{proposition}
  Let $V^\bullet$ be a graded $\mathbb{K}$-module. Then the tensor algebra 
  $\mathrm{T}^\bullet(V)= \bigoplus_{k=0}^\infty \mathrm{T}^k(V)$ with induced grading from 
	$V^\bullet$ and canonical inclusion 
  $\iota \colon  V = \mathrm{T}^1(V) \rightarrow \mathrm{T}(V)$	is the free graded associative 
  unital algebra generated by $V^\bullet$. More precisely, for every 
  homogeneous $\mathbb{K}$-linear map $\phi \colon V \rightarrow A$ from $V$ 
	to a unital graded associative algebra $A$, there exists a unique algebra 
  morphism $\Phi \colon \mathrm{T}(V)\rightarrow A$ such that the following diagram
  \begin{equation*}  
	  \begin{tikzcd}
	    \mathrm{T}(V) \arrow[r,"\exists ! \Phi"]  & A \\
	    V \arrow[ru,swap,"\phi"] \arrow[u,"\iota"]& 
	  \end{tikzcd}
  \end{equation*}
  commutes. 
\end{proposition}
Explicitly, the map $\Phi$ is given by 
\begin{equation*}
  \Phi(1) 
	= 
	1_A
	\quad \text{ and } \quad 
  \Phi(v_1\otimes\cdots \otimes v_n)
  =
	\phi(v_1)\cdots \phi(v_n).
\end{equation*}
In order to construct the free commutative algebra generated by $V^\bullet$, i.e. the 
analogue of the above proposition in the commutative setting, we 
can consider the (graded) symmetric algebra 
\begin{equation}
  \Sym(V)
	=
	\mathrm{T}(V) / \Span\{ x \otimes y - (-1)^{\abs{x}\abs{y}}y \otimes x\}
\end{equation}
with product denoted by $\vee$.
For later use we also recall the definition of the (graded) exterior algebra
\begin{equation}
  \Anti(V)
	=
	\mathrm{T}(V)/\Span\{x \otimes y + (-1)^{\abs{x}\abs{y}} y\otimes x\}
\end{equation}
with product denoted by $\wedge$. 
The following sign conventions are also needed:

\begin{definition}[Graded signature, Koszul sign]
\label{def:GradedSigns}
Let $\sigma \in S_n$ be a permutation, $V^\bullet$ a graded $\mathbb{K}$-module, 
and $x_1,\dots,x_n$ homogeneous elements of degree $\abs{x_i}$ for $i=1,\dots,n$. 
The \emph{Koszul sign} $\epsilon(\sigma)$ is defined by the relation
\begin{equation}
  \epsilon(\sigma) x_{\sigma(1)}\vee \cdots \vee x_{\sigma(n)}
	=
	x_1 \vee \cdots \vee x_n.
\end{equation}
By $\chi(\sigma)= \sign(\sigma)\epsilon(\sigma)$ we denote the \emph{antisymmetric Koszul sign}, i.e. 
\begin{equation}
  \chi(\sigma) x_{\sigma(1)}\wedge \cdots \wedge x_{\sigma(n)}
	=
	x_1 \wedge \cdots \wedge x_n.
\end{equation}
\end{definition}

Note that both signs depend on the degrees of the homogeneous elements, i.e. 
we should actually write $\epsilon(\sigma)=\epsilon(\sigma,x_1,\dots,x_n)$, 
analogously for $\chi(\sigma)$.
With these signs one can consider a right action of $S_n$ on 
$V^{\otimes n}$ given by the symmetrization
\begin{equation*}
  \Symmetrizer_n(v)
	=
	\frac{1}{n!} \sum_{\sigma \in S_n} v \racts \sigma,
\end{equation*}
where
\begin{equation*}
  (v_1\otimes \cdots \otimes v_n) \racts \sigma
	=
	\epsilon(\sigma) v_{\sigma(1)}\otimes \cdots \otimes v_{\sigma(n)}.
\end{equation*}
It turns out that $\bigoplus_n \image \Symmetrizer_n \cong \Sym(V)$ is indeed the 
free associative graded commutative unital algebra generated by $V$.

These free algebras also satisfy a universal property with respect to derivations:

\begin{proposition}
  Let $V^\bullet$ be a graded $\mathbb{K}$-module.
	\begin{propositionlist}
	  \item Let $\phi \colon V^\bullet \rightarrow A^\bullet$ be a homogeneous map 
		      of degree zero into a graded associative unital algebra $A$ and let 
					$\D \colon V^\bullet \rightarrow A^{\bullet +k}$ be a linear map of 
					degree $k\in \mathbb{Z}$. Then there exists a unique graded derivation
					\begin{equation*}
					  D \colon 
						\mathrm{T}(V)^\bullet
						\rightarrow 
						A^{\bullet + k}
					\end{equation*}
					along $\Phi \colon \mathrm{T}(V) \rightarrow A$ such that $D\at{V} = \D$.
		\item If $A$ is in addition graded commutative, then there exists a unique graded 
		      derivation
					\begin{equation*}
					  D \colon 
						\Sym(V)^\bullet
						\rightarrow
						A^{\bullet + k}
					\end{equation*}
					along $\Phi \colon \Sym(V) \rightarrow A$ such that $D\at{V} = \D$.
	\end{propositionlist}
\end{proposition}
The proof is straightforward, one defines $D$ on homogenous factors 
by 
\begin{equation*}
  D(v_1\otimes \cdots \otimes v_n)
	=
	\sum_{r=1}^n (-1)^{k(\abs{v_1}+\cdots+\abs{v_{r-1}})}
	\phi(v_1)\cdots \D(v_r)\cdots \phi(v_n)
\end{equation*}
and $D(1)=0$ and notes that in the commutative case 
\begin{equation}  
  \label{eq:SymmIdeal}
  I(V) 
	=
	\Span\{ x \otimes y - (-1)^{\abs{x}\abs{y}}y \otimes x\} 
\end{equation}	
is in the kernel, whence 
$D$ passes to the quotient $\Sym(V)$.

%
%
\subsection{Definition and First Properties of Graded Coalgebras}
\label{sec:gradedcoalgebras}
 
Now we want to recall the analogous constructions for graded coalgebras.
We start with the definition of graded coalgebras and some basic properties.

\begin{definition}[Graded Coalgebra]
  A \emph{graded coassociative coalgebra} over $\mathbb{K}$ is a 
	graded $\mathbb{K}$-module $C^\bullet$ 
	equipped with a binary co-operation, i.e. a linear map
	\begin{equation}
	  \Delta \colon
		C^\bullet
		\longrightarrow 
		C^\bullet \otimes C^\bullet
	\end{equation}
	of degree zero that is \emph{coassociative}, i.e. 
	\begin{equation}
	  (\Delta \otimes \id) \circ \Delta
		=
		(\id \otimes \Delta) \circ \Delta.
	\end{equation}
\end{definition}
The map $\Delta$ is called \emph{coproduct} and we have
\begin{equation}
  \Delta(C^i)
	\subset
  \bigoplus_{j+k=i} C^j \otimes C^k.
\end{equation}
To simplify the notation we use \emph{Sweedler's notation}
\begin{equation*}
  \Delta^n(x)
	=
	\sum x_{(1)} \otimes \cdots \otimes x_{(n+1)} 
	=
	x_{(1)} \otimes \cdots \otimes x_{(n+1)} 
	\in C^{\otimes n+1},
\end{equation*}
where $\Delta^n = (\Delta \otimes \id \otimes \cdots \otimes \id)\circ \Delta^{n-1}$ 
denotes the iterated coproduct $\Delta^n \colon C \rightarrow C^{\otimes n+1}$ with 
$\Delta^1 = \Delta$ and $\Delta^0 = \id$. Note that by the 
coassociativity we have 
\begin{equation}
  \Delta^n
	=
	(\id\otimes \cdots \otimes \id \otimes \Delta \otimes \id \otimes \cdots 
	\otimes \id) \circ \Delta^{n-1},
\end{equation}
whence the notation makes sense.
The coassociative coalgebra $C$ is called \emph{counital} if it is 
equipped with a \emph{counit}, i.e. a linear map of degree zero with
\begin{equation}
  \epsilon \colon
	C
	\longrightarrow
	\mathbb{K},
	\quad  \text{ satisfying }\quad
	(\epsilon \otimes \id) \circ \Delta
	=
	\id
	=
	(\id \otimes \epsilon)\circ \Delta.
\end{equation}
For example, $\mathbb{K}$ is itself a coassociative coalgebra with 
$\Delta(1)=1 \otimes 1$. 
A \emph{morphism} of graded coalgebras $f \colon C \rightarrow C'$ is a linear 
map of degree zero that commutes with the coproducts and in the counital case also 
with the counits, i.e. 
\begin{equation}
  (f\otimes f) \circ \Delta_C
	=
	\Delta_{C'} \circ f, 
	\quad \quad 
	\epsilon_{C'} \circ f 
	=
	\epsilon_C.
\end{equation}
The graded coalgebra $(C,\Delta)$ is called cocommutative if 
$\tau \circ \Delta=\Delta$, where $\tau$ denotes again the graded switching map. 
Denoting the graded dual by $(C^*)^\bullet = \Hom_\mathbb{K}^\bullet 
(C,\mathbb{K}) $, one can 
check that the inclusion $C^* \otimes C^* \rightarrow (C\otimes C)^*$ induces 
always an algebra structure on $(C^*)^\bullet$. The converse is generally not true; 
e.g., if $\mathbb{K}$ is a field, then the converse is only true in the case of 
finite dimensional vector spaces over $\mathbb{K}$.

Another useful interplay of coalgebras and algebras is the convolution algebra:

\begin{definition}[Convolution product]
\label{def:ConvProduct}
  Let $(A^\bullet,\mu,1)$ be a graded associative unital algebra over 
	$\mathbb{K}$ and 
	let $(C^\bullet,\Delta,\epsilon)$ be a graded coassociative counital algebra over 
	$\mathbb{K}$. Then one defines for $\phi,\psi \in \Hom^\bullet_\mathbb{K}(C,A)$ 
	their \emph{convolution} by 
	\begin{equation}
		\label{eq:convolution}
		\phi \star \psi
		=
		\mu \circ (\phi \otimes\psi) \circ \Delta.
	\end{equation}
\end{definition}
One can directly check that $(\Hom^\bullet_\mathbb{K}(C,A), \star,1\epsilon)$ 
is a graded associative unital algebra, see e.g. 
\cite[Proposition~1.6.1]{loday.vallette:2012a}. 
Dually to ideals and subalgebras of algebras one can consider 
coideals and subcoalgebras.

\begin{definition}[Coideal and Subcoalgebra]
  Let $(C^\bullet,\Delta)$ be a graded coalgebra over $\mathbb{K}$.
	\begin{definitionlist}
	  \item A graded subspace $I^\bullet \subseteq C^\bullet$ is called \emph{coideal} 
		      if
					\begin{equation}
						\label{eq:coideal}
						\Delta(I) 
						\subseteq
						I \otimes C + C \otimes I
						\quad \text{ and } \quad
						I \subseteq \ker \epsilon.
					\end{equation}
		\item A graded subspace $U^\bullet \subseteq C^\bullet$ is called 
		      \emph{subcoalgebra} if
					\begin{equation}
						\label{eq:subcoalgebra}
						\Delta(U) \subseteq U \otimes U.
					\end{equation}
	\end{definitionlist}
\end{definition}
Note that if $\mathbb{K}$ is just a ring and not a field, then $I\otimes C$ might 
not be mapped injectively into $C\otimes C$ due to some torsion effects. We 
always assume that our modules have enough flatness properties and ignore these 
subtleties. 
As expected, the image of a coalgebra morphism is a subcoalgebra and the 
kernel is a coideal. We can quotient by coideals and every subcoalgebra $U$ with 
$U \subseteq \ker \epsilon$ is automatically a coideal.

There are special elements in coalgebras:

\begin{definition}[Group-like elements]
  Let $(C^\bullet,\Delta,\epsilon)$ be a graded coalgebra. An element $g\in C$ is 
	called \emph{group-like} if
	\begin{equation}
	  \label{eq:grouplikeElement}
		\Delta(g)
		=
		g\otimes g
		\quad \quad \text{ and } \quad \quad
		\epsilon(g) = 1.
	\end{equation}
\end{definition}
The second condition is just to exclude the trivial case $g=0$, since 
$\Delta(g)=g\otimes g$ implies $g = \epsilon(g)g$ and since we assume torsion-freeness. Note that a morphism 
of coalgebras maps grouplike elements to grouplike elements.
In view of $L_\infty$-algebras, coderivations will play an important role. 

\begin{definition}[Coderivation]
  Let $\Phi \colon C^\bullet \rightarrow E^\bullet$ be a morphism of graded 
	coalgebras. A \emph{graded coderivation} along $\Phi$ is 
  a linear map $D\colon C^\bullet \rightarrow E^{\bullet + k}$ of degree $k$ such that
	\begin{equation}
    \Delta\circ D
		=
		(D \otimes \Phi +\Phi \otimes D) \circ \Delta.
  \end{equation}
\end{definition}
One can check that the set of coderivations of $C$ along the identity is a graded 
Lie subalgebra of $\End^\bullet_\mathbb{K}(C)$. 

\begin{proposition}
  \label{prop:epsilonDzero}
  A coderivation $D \colon C^\bullet \rightarrow E^{\bullet +k}$ along 
	$\Phi$ satisfies $\epsilon \circ D = 0$.
\end{proposition}
\begin{proof}
We compute
\begin{align*}
  D(x)
	& = 
	(\id\otimes \epsilon)\circ \Delta \circ D(x)
	=
	 D(x) \pm \Phi(x_{(1)})  \epsilon(D(x_{(2)}))
\end{align*}
and thus applying $\epsilon$ gives the result.
\end{proof}

There are examples of coalgebras with many grouplike elements, e.g. 
group coalgebras. However, we are mainly interested in coalgebras with one specific 
group-like element $1$, called coaugmented. 

\begin{definition}
  A counital graded coalgebra $(C,\Delta,\epsilon)$ is called \emph{coaugmented} 
	if there exists a coalgebra morphism $u \colon \mathbb{K} \rightarrow C$.
\end{definition}
The element $1=u(1)$ is indeed group-like since we know $\Delta \circ u = 
(u \otimes u)\circ \Delta$ and we obtain a non-full 
subcategory of coaugmented coalgebras, where the morphisms satisfy $\Phi(1)=1$. 
Moreover, the definition implies $\id_\mathbb{K} = \epsilon \circ u$ and we get  
\begin{equation}
  C 
	=
	\cc{C} \oplus \mathbb{K}1
	=
	\ker \epsilon \oplus \mathbb{K}1
\end{equation}
via $c \mapsto (c - \epsilon(c)1) + \epsilon(c)1$.
In this case one can define a \emph{reduced coproduct} 
$\cc{\Delta}\colon \cc{C} \rightarrow \cc{C}\otimes \cc{C}$ by
\begin{equation}
  \cc{\Delta}(x)
	= 
	\Delta(x) - 1 \otimes x - x \otimes 1
\end{equation}
and we directly check for $x \in \cc{C}$
\begin{equation*}
  (\epsilon \otimes \id) (\cc{\Delta}(x))
	=
	x - \epsilon(1)x - \epsilon(x)1 
	=
	0
	=
	(\id \otimes \epsilon )(\cc{\Delta}(x)).
\end{equation*}
An element $x\in C$ is called \emph{primitive} if 
\begin{equation}
  \Delta(x)
	=
	x \otimes 1 + 1 \otimes x,
\end{equation}
or equivalently $x\in \cc{C}$ with $\cc{\Delta}(x)=0$. The set of primitive elements is 
denoted by $\mathrm{Prim}(C)$.

\begin{proposition}
  The map $\cc{\Delta} \colon C \rightarrow C \otimes C$ is coassociative 
	and restricts to $\cc{\Delta} \colon \cc{C}\to \cc{C}\otimes\cc{C}$. 
	Thus $(\cc{C},\cc{\Delta})$ becomes a coalgebra without counit, the so-called 
	\emph{reduced coalgebra}.
\end{proposition}
In particular, a map $\Phi\colon (C,\Delta,\epsilon,1) \rightarrow (E,\Delta,\epsilon,1)$ 
is a morphism of coaugmented coalgebras, i.e. a coalgebra morphism with $\Phi(1)=1$, 
if and only if 
\begin{equation*}
  ( \Phi \otimes \Phi )\circ \cc{\Delta}
  =
  \cc{\Delta} \circ\Phi,
\end{equation*}
i.e. if and only if $\Phi$ induces a coalgebra morphism $\Phi \colon (\cc{C},\cc{\Delta})
\rightarrow (\cc{E},\cc{\Delta})$. The analogue holds for coderivations: a linear map  
$D \colon C^\bullet \rightarrow E^{\bullet +k}$ with $D(1)=0$ is a coderivation 
along $\Phi$ if and only if 
\begin{equation*}
  (D\otimes \Phi + \Phi \otimes D) \circ \cc{\Delta}
	=
	\cc{\Delta} \circ D.
\end{equation*}

A coaugmented graded coalgebra $(C,\Delta,\epsilon,1)$ is said to be 
\emph{conilpotent} if for each $x \in \cc{C}$ there exists 
$n \in \mathbb{N}$ such that $\cc{\Delta}^m(x)=0$ for all $m\geq n$. This is 
equivalent to the \emph{coradical filtration} being exhaustive, 
i.e.
\begin{equation*}
  C^\bullet
	=
	\bigcup_{k\in\mathbb{Z}} F_k C^\bullet.
\end{equation*}
Here $F_0C = \mathbb{K}1$ and 
\begin{equation*}
  F_kC
	=
	\mathbb{K}1 \oplus \{c \in \cc{C} \mid \cc{\Delta}^{k}c=0\}.
\end{equation*}
In particular, $F_1C = \mathbb{K}1 \oplus \mathrm{Prim}(C)$. Note that 
this filtration is canonical and that morphisms of coaugmented graded coalgebras 
are automatically compatible with this filtration. Moreover, one can check that 
in the case of conilpotent coalgebras the group-like element is unique.

\begin{proposition}
  Let $(C^\bullet,\Delta,\epsilon,1)$ be a conilpotent coaugmented graded coalgebra 
	over a field $\mathbb{K}$. Then $1$ is the only group-like element.
\end{proposition}
\begin{proof}
Let $g\in C$ be a group-like element and set $c = g - 1 \in \cc{C}$ since 
$\epsilon(g)=1$. Then
\begin{equation*}
  \Delta(g)
	=
	\Delta(1+c)
	=
	1\otimes 1 + \Delta(c)
	\quad \text{ and }\quad
	\Delta(g)
	=
	(1+c)\otimes (1+c)
\end{equation*}
imply 
\begin{equation*}
  \cc{\Delta}^k(c)
	=
	c\otimes \cdots \otimes c
\end{equation*}
for any $k$. By the conilpotency $\cc{\Delta}^k(c)=0$ for some $k$, and thus $c=0$.
\end{proof}

We want to list some other immediate features of coaugmented coalgebras. 

\begin{lemma}
  Let $(C^\bullet,\Delta,\epsilon,1)$ be a coaugmented coalgebra and let 
	$V^\bullet$ be a graded $\mathbb{K}$-module. Then 
	\begin{equation*}
	  (\phi_1\otimes \cdots \otimes \phi_n)\circ \Delta^{n-1}
		=
		(\phi_1\otimes \cdots \otimes \phi_n)\circ \cc{\Delta}^{n-1}
	\end{equation*}
	for all $\mathbb{K}$-linear maps $\phi_i \colon C \rightarrow V$ with 
	$\phi_i(1)=0$.
\end{lemma}
This implies for the convolution the following simplification:

\begin{lemma}
  Let $(C^\bullet,\Delta,\epsilon,1)$ be a coaugmented coalgebra and let 
	$(A^\bullet,\mu,1)$ be an algebra. Then one has
	\begin{equation}
	 	\label{eq:CoaugmentedConvolution}
		\phi_1\star \cdots\star \phi_n
		=
		\mu^{n-1} \circ(\phi_1\otimes\cdots\otimes \phi_n)\circ \cc{\Delta}^{n-1}
	\end{equation}
	for all $\mathbb{K}$-linear maps $\phi_i \colon C \rightarrow A$ with 
	$\phi_i(1)=0$.
\end{lemma}
\begin{remark}
  This lemma motivates the following convention. Whenever we have a 
	$\mathbb{K}$-linear map $\phi\colon \cc{C}^\bullet \rightarrow V^\bullet$ 
	we understand it to be extended to $C$ by zero.
\end{remark}

These observations allow us to define power series in $\Hom_\mathbb{K}(\cc{C},A)$: 

\begin{lemma}
  Let $(C^\bullet,\Delta,\epsilon,1)$ be a conilpotent coalgebra and let 
	$A^\bullet$ be an algebra. For $\phi\in \Hom_\mathbb{K}(\cc{C},A)$ the 
	map
	\begin{equation}
		\label{eq:FormalSeriesinConvolution}
		\mathbb{K}[[x]] \ni a 
		=
		\sum_{n=0}^\infty a_n x^n
		\longmapsto 
		a_\star(\phi)
		=
		\sum_{n=0}^\infty a_n \phi^{\star n} \in 
		\Hom_\mathbb{K}(C,A)
	\end{equation}
	is a well-defined unital algebra morphism by the conilpotency, where 
	$\phi^{\star 0} = 1\epsilon$.
\end{lemma}
\begin{example}
  This allows us to define for $\phi \in \Hom_\mathbb{K}(\cc{C},A)$ 
	\begin{equation}
	  \label{eq:ConvolutionExpandLog}
		\log_\star(1\epsilon + \phi)
		=
		\sum_{n=1}^\infty \frac{(-1)^{n-1}}{n}\phi^{\star n} 
		\quad \quad \text{ and } \quad \quad
		\exp_\star(\phi)
		=
		\sum_{n=0}^\infty \frac{1}{n!} \phi^{\star n}
	\end{equation}
	and they are inverse to each other by the usual combinatorial formulas.
\end{example}

A nice consequence is that in the setting of conilpotent coaugmented 
coalgebras we can easily construct the cofree objects.

%
%
\subsection{Cofree Coalgebras}

We state at first the universal property that the cofree conilpotent 
coalgebra should satisfy.

\begin{definition}[Cofree conilpotent coalgebra]
  Let $V$ be a graded $\mathbb{K}$-module and let $C$ be a conilpotent 
	graded coassociative coalgebra. 
  The \emph{cofree conilpotent coassociative coalgebra} over $V$ is a 
  conilpotent coassociative coalgebra $\mathcal{F}^c(V)$ equipped with 
  a linear map $p \colon \mathcal{F}^c(V) \rightarrow V$ such that 
  $1 \mapsto 0$ and such that the following universal condition holds 
	for all $C$: Any linear map $\phi \colon C \rightarrow V$ with $\phi(1)=0$ 
	extends uniquely to a coaugmented coalgebra morphism $\Phi\colon C \rightarrow 
	\mathcal{F}^c(V)$ such that the following diagram
  \begin{equation*}  
	  \begin{tikzcd}
	    C \arrow[d,swap,"\exists ! \Phi"] \arrow[rd,"\phi"] & \\
		  \mathcal{F}^c(V) \arrow[r,"p"] & V
	  \end{tikzcd}
  \end{equation*}
  commutes.
\end{definition}
In other words, $\mathcal{F}^c$ is right adjoint to the forgetful functor. 
Moreover, recall that conilpotent coalgebras are by definition 
coaugmented. Now we want to construct the cofree coalgebra, where the 
uniqueness up to isomorphism follows as usual for universal properties.

Let $V^\bullet$ be a graded $\mathbb{K}$-module, then the 
\emph{tensor coalgebra} $(\mathrm{T}^c(V),\Delta)$ 
is given by the tensor module $\mathrm{T}^c(V) = \mathrm{T}(V)
= \mathbb{K}1 \oplus V \oplus V^{\otimes 2} \oplus \cdots$ with 
\emph{deconcatenation coproduct}
\begin{equation}
  \Delta(v_1\cdots v_n)
	=
	\sum_{i=0}^n( v_1 \cdots v_i )\otimes (v_{i+1}\cdots v_n) \in 
	\mathrm{T}(V) \otimes \mathrm{T}(V)
	\quad \text{and} \quad
	\Delta(1) = 1 \otimes 1,
\end{equation}
where $v_1,\dots,v_n \in V$. The counit $\epsilon\colon \mathrm{T}^c(V)\rightarrow \mathbb{K}$ 
is the identity on $\mathbb{K}$ and zero otherwise, and the coalgebra 
turns out to be coassociative and counital. Moreover, it is coaugmented 
via the inclusion $\iota \colon \mathbb{K} \rightarrow \mathrm{T}^c(V)$ and thus
\begin{equation}
  \mathrm{T}^c(V) 
	\cong \cc{T}^c(V) \oplus \mathbb{K}1
\end{equation}
with reduced tensor module $\cc{T}^c(V) = \bigoplus_{i \geq 1} V^{\otimes i}$. 
The reduced coproduct is given by
\begin{equation}
  \cc{\Delta}(v_1\cdots v_n)
	=
	\sum_{i=1}^{n-1} v_1 \cdots v_i \otimes v_{i+1}\cdots v_n, 
	\quad \text{in particular}\quad
	\cc{\Delta}(v) = 0.
\end{equation}
Thus $(\mathrm{T}^c(V),\Delta,\epsilon,1)$ is conilpotent and 
the construction is functorial in the following sense:

\begin{lemma}
  Let $\phi \colon V^\bullet \rightarrow W^\bullet$ be a homogeneous 
	$\mathbb{K}$-linear map of degree zero. Then the map 
	$\Phi \colon \mathrm{T}(V) \rightarrow 
	\mathrm{T}(W)$ is a morphism of coaugmented coalgebras.
\end{lemma} 
We also have the projection $\pr_V \colon \mathrm{T}^c(V) \rightarrow V$ that is the 
identity on $V$ and $0$ elsewhere, giving the following useful property:

\begin{lemma}
  \label{lemma:projectiononTn}
  Let $V^\bullet$ be a graded $\mathbb{K}$-module. Then one has for all 
	$n\in \mathbb{N}$
	\begin{equation}
  \label{eq:identitycofreecoalg}
	\pr_{\mathrm{T}^n(V)}
	=
	\mu^{n-1}\circ(\pr_V \otimes \cdots \otimes \pr_V)\circ \Delta^{n-1}	
	=
	\mu^{n-1}\circ(\pr_V \otimes \cdots \otimes \pr_V)\circ \cc{\Delta}^{n-1},
\end{equation}
where $\mu=\otimes$ denotes the tensor product.
\end{lemma}
Thus we can write the projection $\pr_{\mathrm{T}^n(V)} = \pr_V^{\star n}$ as 
convolution, which also holds for $n=0$ since $\pr_{\mathrm{T}^0(V)} = 1 \epsilon$, and 
we write $x_n = \pr_{\mathrm{T}^n(V)}x$ for $x \in \mathrm{T}(V)$. 
We can show that we constructed indeed the cofree conilpotent coalgebra, 
compare \cite[Proposition~1.2.1]{loday.vallette:2012a}:

\begin{theorem}[Cofree conilpotent coalgebra]
  \label{thm:CofreeConilpotentCoalg}
  Let $V^\bullet$ be a graded module over $\mathbb{K}$ and let 
	$(C^\bullet,\Delta,\epsilon,1)$ be a conilpotent graded coassociative 
	coalgebra over $\mathbb{K}$.
	\begin{theoremlist}
	  \item The tensor coalgebra $(\mathrm{T}^c(V),\Delta,\epsilon,1)$ is cofree and cogenerated 
		      by $V$ in the category of conilpotent coalgebras. Explicitly, for every 
					homogeneous $\mathbb{K}$-linear map $\phi \colon \cc{C} \rightarrow V$ of 
					degree zero extended to $C$ by $\phi(1)=0$ there exists a unique 
					counital coalgebra morphism $\Phi \colon C \rightarrow \mathrm{T}^c(V)$ such that 
					$\pr_V \circ \Phi = \phi$.
		\item If in addition $\D \colon \cc{C}^\bullet \rightarrow V^{\bullet +k}$ is 
		      a homogeneous $\mathbb{K}$-linear map of degree $k$ extended to $C$ by 
					$\D(1)=0$, then there exists a unique coderivation $D \colon 
					C^\bullet \rightarrow \mathrm{T}^c(V)^{\bullet+k}$ along $\Phi$ vanishing on $1$ 
					such that $\pr_V \circ D = \D$.
	\end{theoremlist}
\end{theorem}
\begin{proof}
The morphism $\Phi \colon C \rightarrow \mathrm{T}^c(V)$ has to satisfy for $x\in \cc{C}$
\begin{itemize}
  \item $\Phi(1)  = 1$ since $\Phi$ maps the group-like element to the group-like 
	      element,
	\item $\Phi(x)_0 = 0$ by counitality,
	\item $\Phi(x)_1 = \phi(x)$ by the universal property,
	\item $\Phi(x)_n = \sum \phi(x_{(1)}) \otimes \cdots \otimes \phi(x_{(n)})$ 
	      by the coalgebra morphism property, since
			  \begin{equation*}
				  \pr_{\mathrm{T}^n(V)}\Phi(x)
					=
					\mu^{n-1} \circ(\phi \otimes \cdots \otimes \phi) \circ \cc{\Delta}^{n-1}(x).
				\end{equation*}
\end{itemize}
As $C$ is conilpotent, there is only a finite number of nontrivial 
$\Phi(x)_n$ and $\Phi(x)=\sum_n \Phi(x)_n$ gives a well-defined linear map 
of degree zero. This shows the uniqueness and a direct computation 
shows that $\Phi$ is indeed a well-defined coalgebra morphism. For the 
second part we show the uniqueness by the same arguments: Suppose $D$ is 
such a coderivation, then since $\epsilon \circ D = 0$ by 
Proposition~\ref{prop:epsilonDzero} we know
\begin{equation*}
  D(c)
	=
	\sum_{n=1}^\infty D(c)_n
\end{equation*}
with $D(c)_1 = \D(c)$ for $c\in C$. For $n>1$ we get with the Leibniz rule
\begin{align*}
  D(c)_n
	& = 
	\mu^{n-1}\circ(\pr_V \otimes \cdots \otimes \pr_V) \circ 
	\left(\sum_{r=0}^{n-1} \Phi \otimes \cdots\otimes\Phi \otimes D \otimes \Phi 
	\otimes \Phi\right) \circ \Delta^{n-1} (c) \\
	& =
	\mu^{n-1}\circ 
	\left(\sum_{r=0}^{n-1} \phi \otimes \cdots\otimes\phi \otimes \D \otimes \phi 
	\otimes \phi\right) \circ \cc{\Delta}^{n-1} (c)
\end{align*}
This shows that necessarily $D(1)=0$ and a straightforward computation shows that 
this is indeed a well-defined coderivation.
\end{proof}

\begin{corollary}
  For the coalgebra morphism $\Phi$ and the coderivation $D$ along $\Phi$ 
	from the above theorem one has
	\begin{equation*}
	  \Phi 
		=
		1\epsilon + \phi + \phi \star \phi + \dots
		=
		\frac{1}{1-\phi}\star 
	\end{equation*}
	and
	\begin{equation*}
		D 
		=
		\D + \phi \star \D + \D \star \phi + \phi \star \D \star\phi + \dots
		=
		\Phi \star \D \star \Phi.
	\end{equation*}
\end{corollary}
For $C = \mathrm{T}^c(V)$ itself and $\phi = \pr_V$, i.e. $\Phi = \id$, 
we get the following nice property.

\begin{corollary}
  \label{cor:coderivationcogenerators}
  The coderivations of $\mathrm{T}^c(V)$ vanishing on $1$ form a Lie subalgebra of the 
	endomorphisms. This Lie algebra is in bijection to $\Hom^\bullet_\mathbb{K}(
	\cc{\mathrm{T}^c(V)},V)$ via
	\begin{equation*}
	  \Hom^\bullet_\mathbb{K}(\cc{\mathrm{T}^c(V)},V) \ni 
		\D
		\longmapsto 
		D = \id \star \D \star \id \in 
		\mathrm{CoDer}^\bullet_0(\mathrm{T}^c(V)).
	\end{equation*}
\end{corollary}

Let now $D \in \mathrm{CoDer}^\bullet_0(\mathrm{T}^c(V))$, then the above corollary implies 
that it is completely determined by its projection $\pr_V \circ D = \D$. 
Here one has
\begin{equation*}
  \D 
	=
	\sum_{n=1}^\infty \D_n 
	\quad \quad \text{ with } \quad \quad
	\D_n = \pr_V \circ D \circ \pr_{\mathrm{T}^n(V)}.
\end{equation*}
These maps $\D_n \colon \mathrm{T}^n(V) \rightarrow V$ are called \emph{Taylor coefficients} 
of $D$.

\begin{remark}[$A_\infty$-structures]
  \label{rem:Ainftystructures}
  By the above corollary we have the following isomorphism 
	\begin{equation*}
	  \mathrm{CoDer}^\bullet_0(\mathrm{T}^c(V)) 
		\cong 
	  \Hom_\mathbb{K}^\bullet(\cc{\mathrm{T}^c(V)},V).
	\end{equation*}
	We can use it to induce on $\Hom_\mathbb{K}^\bullet(\cc{\mathrm{T}^c(V)},V)$ 
	a Lie algebra structure. The induced bracket is just the Gerstenhaber bracket 
	$[\argument{,}\argument]_G$. 
	In particular, for $V = A[1]$, where $(A,\mu)$ is an associative algebra, 
	one gets on the right hand side the space of Hochschild cochains, and the 
	product	$\mu$ of $A$ corresponds to a coderivation 
	on $\mathrm{T}^c(A[1])$. The associativity of $\mu$ is equivalent to $[\mu,\mu]_G=0$, i.e. 
	the induced coderivation is even a codifferential. 
	More generally, for a graded module $A$, 
	a codifferential $D$ on $\mathrm{T}^c(A[1])$ of degree $+1$ is equivalent to a 
	Maurer-Cartan element in $(\Hom^\bullet_\mathbb{K}(\cc{\mathrm{T}^c(A[1])},A[1]),0,
	[\argument{,}\argument]_G)$. This is called \emph{$A_\infty$-structure} on $A$, 
	compare e.g. \cite[Chapter~9]{loday.vallette:2012a},
	and it is determined by a sequence of maps $\mu_n \colon 
	\mathrm{T}^n(A[1]) \rightarrow A[2]$, where $n\geq 1$, with 
	\begin{equation*}
	  \sum_{q=1}^n \sum_{j=0}^{p-1}\pm
		\mu_{n-q+1}(a_1,\dots, a_j,\mu_q(a_{j+1},\dots,a_{j+q}),a_{j+q+1},\dots,a_n)
		=
		0
	\end{equation*}
	because of $D^2=0$. In the case where we also allow $D(1)\neq 0$ for a 
	coderivation on $\mathrm{T}^c(A[1])$ 
	we get curved $A_\infty$-algebras with $\mu_0 \neq 0$. 
  In operadic terms, this definition of $A_\infty$-algebras is motivated by the 
	fact that the Koszul dual of the operad encoding associative 
	algebras is the cooperad encoding coassociative coalgebras. 
\end{remark}

Now we want to conclude this introductory section with the construction 
of the cofree cocommutative conilpotent coalgebras. From the point of 
$L_\infty$-algebras they are of great interest since they are the structures 
which one uses to define $L_\infty$-structures. The abstract reason for this 
is that the Koszul dual of the Lie operad is the cooperad encoding  
cocommutative coalgebras, see e.g. \cite[Chapter~10]{loday.vallette:2012a} for 
the general setting.

Consider again a graded  $\mathbb{K}$-module $V^\bullet$. We can use 
the tensor algebra $\mathrm{T}(V)$ of $V$ to define a new coproduct, different from the 
deconcatenation coproduct $\Delta$ from above. Explicitly, we equip 
$\mathrm{T}(V)$ with the structure of a bialgebra, i.e. we construct the new 
coproduct as an algebra morphism $\Delta_\sh  \colon \mathrm{T}(V) \rightarrow \mathrm{T}(V) \otimes \mathrm{T}(V)$. 
Since $\mathrm{T}(V)$ is the free algebra, we specify $\Delta_\sh $ on the generators by 
\begin{equation*}
  \Delta_\sh (v)
	=
	v\otimes 1 + 1 \otimes v
\end{equation*}
for $v\in V$ and $\Delta_\sh (1) = 1 \otimes 1$. One can check with the signs 
from Definition~\ref{def:GradedSigns} that we get
\begin{equation}
  \label{eq:cocommutativegradedcoproduct}
  \Delta_\sh (v_1  \cdots  v_n)
	=
	\sum_{k=0}^{n} \sum_{\sigma \in Sh(k,n-k)}	\epsilon(\sigma)
	(v_{\sigma(1)}  \cdots  v_{\sigma(k)}) \otimes
	(v_{\sigma(k+1)}  \cdots  v_{\sigma(n)}).	
\end{equation}
Here $Sh(k,n-k)\subset S_n$ denotes the set of $(k,n-k)$-\emph{shuffles}, 
i.e. $\sigma(1) < \cdots < \sigma(k)$ and 
$\sigma(k+1) < \cdots < \sigma(n)$. 
We set $Sh(0,n)=Sh(n,0)=\{\id\}$, 
and the above coproduct is well-defined because 
of $S_n \cong Sh(k,n-k) \circ (S_k \times S_{n-k})$. Moreover, 
$\Delta_\sh $ is coassociative, counital and graded cocommutative with respect 
to the usual counit $\epsilon \colon \mathrm{T}(V) \rightarrow \mathbb{K}$ and we call 
it \emph{shuffle coproduct}. 
Since $\Delta_\sh $ is an algebra morphism, it is sufficient to show all these claims 
on generators. 
In particular, we can again consider the reduced coproduct
\begin{equation}
  \label{eq:redcocommutativegradedcoproduct}
  \cc{\Delta_\sh }(v_1  \cdots  v_n)
	=
	\sum_{k=1}^{n-1} \sum_{\sigma \in Sh(k,n-k)}	\epsilon(\sigma)
	(v_{\sigma(1)} \cdots  v_{\sigma(k)}) \otimes
	(v_{\sigma(k+1)} \cdots  v_{\sigma(n)})
\end{equation}
and as it splits the tensor factors without using the unit $1$ we get again
\begin{equation*}
  \cc{\Delta_\sh }^n(v_1\cdots v_n)
	=
	0.
\end{equation*}

\begin{proposition}
  The tensor algebra $(\mathrm{T}(V),\Delta_\sh ,\epsilon,1)$ is a conilpotent 
	cocommutative coalgebra. In fact, it is even a bialgebra with respect to 
	the usual tensor product $\mu = \otimes$ and the usual unit $1$.
\end{proposition}

In analogy to Lemma~\ref{lemma:projectiononTn} we get the following result:

\begin{lemma}
  \label{lemma:projectiononTnCocom}
	Let $V^\bullet$ be a graded $\mathbb{K}$-module. For $n\in \mathbb{N}$ one 
	has
	\begin{equation}
	  \label{eq:projectiononTnCocom1}
		(\pr_V\otimes \cdots \otimes \pr_V) 
		\circ \Delta_\sh ^{n-1} (v_1\cdots v_n)
		=
		\sum_{\sigma\in S_n}
		\epsilon(\sigma) v_{\sigma(1)} \otimes \cdots \otimes v_{\sigma(n)}.
	\end{equation}
\end{lemma}

This is not exactly the property we would like to have since the right 
hand side is not the tensor we started with. Recall that this property was exactly 
the statement of Lemma~\ref{lemma:projectiononTn} that 
we used extensively in Theorem~\ref{thm:CofreeConilpotentCoalg} to show 
that $(\mathrm{T}(V),\Delta)$ is the cofree conilpotent coalgebra. However, 
we see that the right hand side of \eqref{eq:projectiononTnCocom1} is 
$n! \Symmetrizer_n(v_1\otimes \cdots \otimes v_n)$, whence we would get 
the desired equation if we pass from $\mathrm{T}(V)$ to $\Sym(V)$. This is indeed possible.

\begin{lemma}
  Let $V^\bullet$ be a graded $\mathbb{K}$-module. Then the ideal $I(V) \subseteq \mathrm{T}(V)$ 
	from \eqref{eq:SymmIdeal} is a coideal with respect to $\Delta_\sh$ and $\epsilon$.
\end{lemma}
\begin{proof}
The property $I(V) \subseteq \ker\epsilon$ is obvious, and for $v,w\in V^\bullet$ 
we have 
\begin{equation*}
  \Delta_\sh( vw - (-1)^{\abs{v}\abs{w}} wv)
	=
	1 \otimes (vw - (-1)^{\abs{v}\abs{w}}wv) + 
	(vw - (-1)^{\abs{v}\abs{w}} wv) \otimes 1 \in 
	I(V) \otimes \mathrm{T}(V) + \mathrm{T}(V)\otimes I(V)
\end{equation*}
and the algebra morphism property shows the result.
\end{proof}

This gives us a bialgebra structure on the quotient.

\begin{proposition}
  \label{prop:Symwithshuffle}
  Let $V^\bullet$ be a graded $\mathbb{K}$-module. 
	\begin{propositionlist}
	  \item The shuffle coproduct and the counit pass to the quotient by $I(V)$ 
		      and yield a 
		      bialgebra structure $(\Sym(V),\mu_S,1,\Delta_\sh,\epsilon)$, where 
					$\mu_S = \vee$.
		\item The bialgebra $(\Sym(V),\mu_S,1,\Delta_\sh,\epsilon)$ is graded commutative 
		      and graded cocommutative.
		\item The coalgebra $(\Sym(V),\Delta_\sh,\epsilon,1)$ is coaugmented and 
		      conilpotent with coradical filtration
					\begin{equation}
					  F_n(\Sym(V)) = \bigoplus_{k=0}^n \Sym^k(V).
					\end{equation}
	\end{propositionlist}
\end{proposition}

In particular, Lemma~\ref{lemma:projectiononTnCocom} implies now immediately 
the desired identity:

\begin{lemma}
  \label{lemma:projectiononSnCocom}
	Let $V^\bullet$ be a graded $\mathbb{K}$-module. For $n\in \mathbb{N}$ one 
	has
	\begin{equation}
	  \label{eq:projectiononTnCocom}
		\pr_{\Sym^n(V)} 
		=
		\frac{1}{n!} \mu_S^{n-1} \circ (\pr_V\otimes \cdots \otimes \pr_V) 
		\circ \Delta_\sh ^{n-1}
		=
		\frac{1}{n!} \mu_S^{n-1} \circ (\pr_V\otimes \cdots \otimes \pr_V) 
		\circ \cc{\Delta_\sh}^{n-1}.
	\end{equation}
\end{lemma}

This allows us to show that $\Sym(V)$ is the cofree cocommutative conilpotent 
coalgebra:

\begin{theorem}[Cofree cocommutative conilpotent coalgebra]
  \label{thm:CofreeCocomConilpotentCoalg}
  Let $V^\bullet$ be a graded module over $\mathbb{K}$ and let 
	$(C^\bullet,\Delta,\epsilon,1)$ be a conilpotent graded cocommutative 
	coalgebra over $\mathbb{K}$.
	\begin{theoremlist}
	  \item The cocommutative conilpotent coalgebra 
		      $(\Sym(V),\Delta_\sh,\epsilon,1)$ is the cofree cocommutative conilpotent 
					coalgebra cogenerated by $V$, i.e. the free object in the category 
					of conilpotent cocommutative coalgebras. Explicitly, for every 
					homogeneous $\mathbb{K}$-linear map $\phi \colon C \rightarrow V$ of 
					degree zero with $\phi(1)=0$ there exists a unique 
					coalgebra morphism $\Phi \colon C \rightarrow \Sym(V)$ such that
					\begin{equation*}  
	          \begin{tikzcd}
	            C \arrow[d,swap," \Phi"] \arrow[rd,"\phi"] & \\
		          \Sym(V) \arrow[r,"\pr_V"] & V,
	          \end{tikzcd}
          \end{equation*}
	        commutes, i.e. $\pr_V \circ \Phi = \phi$. Explicitly, one has 
					$\Phi = \exp_\star(\phi)$ for the convolution product of 
					$\Hom^\bullet(C,\Sym(V))$ with respect to $\mu_S$. 
		\item Let $\D \colon C^\bullet \rightarrow V^{\bullet +k}$ be 
		      a homogeneous $\mathbb{K}$-linear map of degree $k$, then there exists 
					a unique coderivation $D \colon 
					C^\bullet \rightarrow \Sym(V)^{\bullet+k}$ along $\Phi$ 
					such that
					\begin{equation*}  
	          \begin{tikzcd}
	            C^\bullet \arrow[d,swap,"D"] \arrow[rd,"\D"] & \\
		          \Sym(V)^{\bullet + k} \arrow[r,"\pr_V"] & V^{\bullet +k},
	          \end{tikzcd}
          \end{equation*}
	        commutes. Explicitly, one has
					$  D 	=		\Phi \star \D		=		\D \star \Phi$ and 
					$D(1) = 0$ if and only if $\D(1)=0$.
	\end{theoremlist}
\end{theorem}
\begin{proof} 
The proof is completely analogue to the proof of Theorem~\ref{thm:CofreeConilpotentCoalg}.
\end{proof}

For $C = \Sym(V)$ itself and $\phi = \pr_V$, i.e. $\Phi = \id$, 
we get analoguously to 
Corollary \ref{cor:coderivationcogenerators}.

\begin{corollary}
  \label{cor:coderivationcogeneratorscocom}
  Coderivations of $\Sym(V)$ form a Lie subalgebra of the 
	endomorphisms. This Lie algebra is in bijection to $\Hom^\bullet_\mathbb{K}(
	\Sym(V),V)$ via
	\begin{equation*}
	  \Hom^\bullet_\mathbb{K}(\Sym(V),V) \ni 
		\D
		\longmapsto 
		D = \D  \star \id \in 
		\mathrm{CoDer}^\bullet(\Sym(V)).
	\end{equation*}
\end{corollary}
Let now $D \in \mathrm{CoDer}^\bullet(\Sym(V))$, then we know again that it is completely 
determined by its projection $\pr_V \circ D = \D$. 
Here 
\begin{equation*}
  \D 
	=
	\sum_{n=0}^\infty \D_n 
	\quad \quad \text{ with } \quad \quad
	\D_n = \pr_V \circ D \circ \pr_{\mathrm{T}^n(V)}
\end{equation*}
and the maps $\D_n \colon \mathrm{T}^n(V) \rightarrow V$ are again called \emph{Taylor coefficients} 
of $D$.
By the above corollary we have the isomorphism 
\begin{equation*}
	\mathrm{CoDer}^\bullet(\Sym(V)) 
	\cong 
	\Hom_\mathbb{K}^\bullet(\Sym(V),V)
\end{equation*}
that we can use to induce on $\Hom_\mathbb{K}^\bullet(\Sym(V),V)$ 
a Lie algebra structure. The induced bracket is called 
\emph{Nijenhuis-Richardson bracket}. Similarly to the case of 
associativity in Remark~\ref{rem:Ainftystructures} this 
bracket encodes now the Jacobi identity. For example, for a 
graded Lie algebra $(L,[\argument{,}\argument])$ the 
rescaled bracket $\widetilde{\mu}(x,y) = -(-1)^{\abs{x}}[x,y]$, 
where $\abs{x}$ denotes the degree of $x$ in $L[1]$, is 
a Maurer-Cartan element in $(\Hom_\mathbb{K}^\bullet(\Sym(L[1]),L[1]),0,
[\argument{,}\argument]_{NR})$, inducing a codifferential of degree $+1$ on 
$\Sym(L[1])$. Consequently, in the general case a (curved) $L_\infty$-structure 
on a graded $\mathbb{K}$-module $L$ is a Maurer-Cartan element in the DGLA 
$(\Hom_\mathbb{K}^\bullet(\Sym(L[1]),L[1]),0,[\argument{,}\argument]_{NR})$, 
which in turn is equivalent to 
a codifferential of degree $+1$ on $\Sym(L[1])$. As mentioned above, this is 
due to the fact that the Koszul dual of the operad encoding Lie algebras is the 
cooperad encoding cocommutative coalgebras, and we want to study these 
structures now in more detail.

%
%
\section{Introduction to $L_\infty$-algebras}
\label{sec:DefLinfty}

%
%
\subsection{Definition and First Properties}

After this introduction to coalgebras we are now ready to study $L_\infty$-algebras. 
We start with (flat) $L_\infty$-structures, i.e. those corresponding to coderivations 
$Q$ with $Q(1)=0$. For simplicity, we assume from now on that $\mathbb{K}$ is a field 
of characteristic zero and we follow mainly \cite{canonaco:1999a,dolgushev:2005b}.

\begin{definition}[$L_\infty$-algebra]
  A (flat) $L_\infty$-\emph{algebra} is a graded vector space $L$ over $\mathbb{K}$ 
	endowed with a degree one codifferential $Q$ on the reduced 
	symmetric coalgebra $(\cc{\Sym}(L[1]),\cc{\Delta_\sh})$.
  An $L_\infty$-morphism between two $L_\infty$-algebras $F\colon (L,Q) 
  \rightarrow (L',Q')$ is a morphism of graded coalgebras
  \begin{equation}
    F \colon 
		\cc{\Sym}(L[1])
		\longrightarrow
		\cc{\Sym}(L'[1])
  \end{equation}
	such that $F \circ Q = Q' \circ F$.
\end{definition}
By Corollary~\ref{cor:coderivationcogeneratorscocom} we can characterize 
the coderivation by its Taylor coefficients, also called structure maps.

\begin{proposition}
  \label{prop:coderivationsymmetricsequence}
  An $L_\infty$-algebra $(L,Q)$ is a graded vector space $L$ endowed 
	with a sequence of maps
	\begin{equation}
	  Q_n^1 \colon
		\Sym^n (L[1]) 
		\longrightarrow 
		L[2]
	\end{equation}
	for $n> 0$. The coderivation $Q$ is given by 
	\begin{align}
	  \label{eq:QinSymmetricviaTaylor}
	  \begin{split}
	  Q(x_1 \vee \cdots \vee x_n)
		& =
		\sum_{k=1}^{n} \sum_{\sigma\in Sh(k,n-k)}
		\epsilon(\sigma) Q_{k}^1(x_{\sigma(1)}\vee\cdots\vee x_{\sigma(k)}) 
		\vee x_{\sigma(k+1)} \vee \cdots  \vee	x_{\sigma(n)}
		\end{split}
	\end{align}
	for homogeneous vectors $x_1,\dots,x_n \in L$, and $Q^2=0$ is equivalent to
	\begin{equation}
	\label{eq:QsquaredZero}
	  \sum_{k=1}^n  \sum_{\sigma\in Sh(k,n-k)} 
		\epsilon(\sigma) 
		Q_{n-k+1}^1(Q_k^1(x_{\sigma(1)}\vee\cdots\vee x_{\sigma(k)})\vee 
		x_{\sigma(k+1)}\vee \dots \vee 	x_{\sigma(n)})
		=
		0.
	\end{equation}
\end{proposition}

In particular, \eqref{eq:QsquaredZero} implies $(Q_1^1)^2=0$, i.e. that $Q_1^1$ 
defines a differential on $L$ of degree one. The next order shows that 
$Q_2^1$ satisfies a Leibniz rule with respect to $Q_1^1$, and the third one 
can be interpreted as $Q_2^1$ satisfying a Jacobi identity up to terms depending on 
$Q_3^1$.
We also write $Q_k = Q_k^1$ and following
\cite{canonaco:1999a} we denote by $Q_n^i$ the component 
\begin{equation*}
  Q_n^i 
	=
	\pr_{\Sym^i(L[1])} \circ \, Q\at{\Sym^n(L[1])}
\colon \Sym^n( L[1] )\longrightarrow \Sym^i (L[2])
\end{equation*} 
of $Q$. It is given by
\begin{equation}
  \label{eq:Qniformula}
  Q_n^i(x_1\vee \cdots \vee x_n)
	=
	\sum_{\sigma \in \mathrm{Sh}(n+1-i,i-1)}
	\epsilon(\sigma) Q_{n+1-i}^1(x_{\sigma(1)}\vee \cdots\vee  x_{\sigma(n+1-i)})\vee
	x_{\sigma(n+2-i)} \vee \cdots \vee x_{\sigma(n)},
\end{equation}
where $Q_{n+1-i}^1$ are the usual structure maps. 
$L_\infty$-algebras have also been called strong homotopy Lie algebras 
as they are DGLAs up to higher homotopies.

\begin{example}[DGLA]
  \label{ex:DGLAasLinfty}
  A differential graded Lie algebra (DGLA) 
	$(\liealg{g},\D,[\argument{,}\argument])$ is an $L_\infty$-algebra 
	with $Q_1 = -\D$ and $Q_2(\gamma \vee \mu) = 
	-(-1)^{\abs{\gamma}}[\gamma,\mu]$ and $Q_n=0$ for all 
	$n > 2$, where $\abs{\gamma}$ denotes the degree 
	in $\liealg{g}[1]$.
\end{example}

This means that every DGLA is an $L_\infty$-algebra and for 
later use we consider the following example: 

\begin{example}\label{ex:EndDGLA}
Let $(M^\bullet,\D)$ be a cochain complex over $\mathbb{K}$. Then we define 
	\begin{align*}
	\End^k(M)=\{\phi\colon M^\bullet\to M^{\bullet+k}\ | \ \phi \text{ linear }\}
	\end{align*}
and $\End^\bullet(M)=\bigoplus_{i\in\mathbb{Z}} \End^i(M)$. For elements 
$A_i\in \End^{|A_i|}(M)$, we define
	\begin{align*}
	[A_1,A_2]:=A_1\circ A_2 -(-1)^{|A_1||A_2|}A_2\circ A_1,
	\end{align*}	 
which is a graded Lie bracket. Finally, setting $D=[\D,\argument]$, 
one sees that $(\End^\bullet(M),D,[\argument{,}\argument])$ is a DGLA and hence an $L_\infty$-algebra. 
\end{example}

\begin{remark}
\label{rem:CohomologyLieAlg}
  Note that $(Q_1^1)^2=0$ allows us to study the cohomology 
	$\mathrm{H}(L)$ of the cochain complex $(L,Q^1_1)$. In particular,  
	\eqref{eq:QsquaredZero} implies that $\mathrm{H}(L)$ interits 
	a Lie algebra structure since the bracket induced by $Q^1_2$ satisfies 
	the usual Jacobi identity.
\end{remark}

\begin{remark}[Antisymmetric formulation]
  Using the d\'ecalage-isomorphism 
  \begin{align*}
    dec^n \colon 
	  \Sym^n (L) 
	  & \longrightarrow 
	  \Anti^n(L[-1])[n]
  \end{align*}
	one can show that an $L_\infty$-algebra structure on $L$ is equivalently given 
	by a sequence of 	maps
	\begin{equation}
	  Q^a_n \colon
		\Anti^n L 
		\longrightarrow 
		L[2-n]
	\end{equation}
	for $n> 0$ with 
	\begin{equation}
	  \sum_{k=1}^n (-1)^{n-k} \sum_{\sigma\in Sh(k,n-k)} 
		\chi(\sigma) 
		Q^a_{n-k+1}(Q^a_k(x_{\sigma(1)}\wedge\dots\wedge x_{\sigma(k)})\wedge 
		x_{\sigma(k+1)}\wedge \dots \wedge 	x_{\sigma(n)})
		=
		0
	\end{equation}
	for $n \geq 1$, compare 
	e.g. \cite[Theorem~1.3.2]{manetti:note}.
  Because of the grading shift from $\Sym(L[1])$ to $\Anti L$  
	one sometimes also speaks of $L_\infty[1]$-algebra structures on $L[1]$ 
	in the symmetric setting, and of $L_\infty$-algebra structures on $L$ 
	in the antisymmetric setting. 
	But since they are equivalent we refer to them just as $L_\infty$-algebras.
	Moreover, note that the symmetric interpretation of the structure maps of 
  $L_\infty$-morphisms is more natural in the sense of the definition via 
  the cocommutative cofree coalgebra. However, if one sees $L_\infty$-algebras 
	as a generalization of DGLAs the antisymmetric interpretation is the most natural one.
\end{remark}

Similarly, we know from Theorem~\ref{thm:CofreeCocomConilpotentCoalg} that every 
$L_\infty$-morphism is characterized by its 
Taylor coefficients, i.e. by a sequence of maps.

\begin{proposition}
  \label{prop:linftymorphismsequence}
  An $L_\infty$-morphism $F\colon (L,Q) \rightarrow (L',Q')$ is 
	uniquely determined by the collection of 
	multilinear graded maps
	\begin{equation}
	  F_n^1
		=
		\pr_{L[1]} \circ F \at{\Sym^n(L[1])}		
		\colon
		\Sym^n (L[1])
		\longrightarrow
		L'[1]
	\end{equation}
	for $n \geq 1$. Setting $F_0^1=0$, it follows that $F$ is given by
  \begin{align}
	  \begin{split}
	  F(x_1 \vee \cdots \vee x_n)
		& 
		=
		\exp_\star(F^1)(x_1 \vee \cdots \vee x_n)
		=
		\sum_{p \geq 1} \sum_{k_1 + \cdots + k_p=n, k_i\geq 1} 
		\sum_{\sigma\in Sh(k_1,\dots,k_p)}
		\frac{\epsilon(\sigma)}{p!} \\
		& \quad \quad F_{k_1}^1(x_{\sigma(1)}\vee \dots\vee x_{\sigma(k_1)}) 
		\vee \cdots  \vee F_{k_p}^1(x_{\sigma(n-k_p+1)}\vee \cdots \vee 
		x_{\sigma(n)})
		\end{split}
	\end{align}
	and the compatibility with the coderivations leads to further constraints. 
	In particular, one has in lowest order $F_1^1 \circ Q_1^1 =  (Q')_1^1\circ F_1^1$.
\end{proposition}

We also write $F_k = F_k^1$ and we get coefficients $F_n^j =\pr_{\Sym^j(L[1])} \circ F 
\at{\Sym^n(L[1])}
\colon \Sym^n (L[1] )\rightarrow 
\Sym^j (L'[1])$ of $F$. Note that $F_n^j$ depends only on $F_k^1 = F_k$ for $k\leq n-j+1$. 

\begin{remark}
  Analogously to the case of coderivations, one can interpret 
	an $L_\infty$-morphism as a sequence of multilinear maps
	\begin{equation}
	  F_n \colon
		\Anti^n L
		\longrightarrow
		L'[1-n],
	\end{equation}
	satisfying certain compatibility relations. In the case of DGLAs 
	$(\liealg{g}_1,\D_1,[\argument{,}\argument]_1)$ and 
	$(\liealg{g}_2,\D_2,[\argument{,}\argument]_2)$ one can show 
	that the compatibility with the differentials takes the following form
	\begin{align}
	  \begin{split}
		\label{eq:linftymorphismdiff}
		  \D_2 F_n(x_1,\dots,x_n) 
			& = 
			\sum_{i=1}^n (-1)^{k_1+ \cdots + k_{i-1}+1-n} 
			F_n(x_1,\dots,\D_1 x_i,\dots,x_n)   \\
			& \quad 
			+ \frac{1}{2} \sum_{k=1}^{n-1} \sum_{\sigma\in Sh(k,n-k)} 
			\pm [F_k(x_{\sigma(1)},\dots,x_{\sigma(k)}),F_{n-k}(x_{\sigma(k+1)},\dots,
			x_{\sigma(n)})]_2 \\
			& \quad
			- \sum_{i \neq j} \pm F_{n-1}([x_i,x_j]_1,x_1,\dots,\widehat{x_i},\dots,
			\widehat{x_j},\dots,x_n)
		\end{split}
	\end{align}
	for $x_i \in \liealg{g}_1^{k_i}$, compare \cite{dolgushev:2005a}.
\end{remark}

In order to gain a better understanding, we consider at first 
$L_\infty$-isomorphisms and their inverses. We 
follow \cite[Section~2]{canonaco:1999a}.

\begin{proposition}
  \label{prop:Linftyiso}
  An $L_\infty$-morphism $F$ between $L_\infty$-algebras $(L,Q)$ and 
	$(L',Q')$ is an isomorphism if and only if $F_1^1$ is an isomorphism. 
\end{proposition}
\begin{proof}
We recall the proof from \cite[Proposition~2.2]{canonaco:1999a}. 
We only have to show that $F$ is invertible as coalgebra morphism if 
and only if $F_1^1$ is invertible with inverse $(F^1_1)^{-1}$. The fact that 
$F^{-1}$ is an $L_\infty$-morphism follows then directly:
\begin{equation*}
  F^{-1}Q'
	=
	F^{-1}Q' FF^{-1}
	=
	F^{-1}FQF^{-1}
	=
	QF^{-1}.
\end{equation*}
Let therefore $F$ be an isomorphism with inverse $F^{-1}$. Then
\begin{equation*}
  \id_{L[1]}
	=
	(\id_{\cc{\Sym} (L[1])})^1_1
	=
	(F^{-1}F)^1_1
	=
	(F^{-1})^1_1 F_1^1,
\end{equation*}
analogously $\id_{L'[1]}=F^1_1(F^{-1})_1^1$, whence $F_1^1$ is an isomorphism.

Suppose now that $F_1^1$ is an isomorphism. We construct a left inverse 
$G$ of $F$ by defining recursively its coefficient functions $G_n$. 
Starting with $G_1 = (F_1^1)^{-1}$ we want for $n>1$
\begin{equation*}
  (GF)_n
	=
	\sum_{i=1}^n G^1_i F^i_n
	\stackrel{!}{=}
	(\id_{\cc{\Sym}(L[1])})^1_n
	=
	0,
\end{equation*}
which is fulfilled for 
\begin{equation*}
  G_n
	=
	G^1_n
	=
	- \left( \sum_{i=1}^{n-1} G^1_i F^i_n\right) (F^n_n)^{-1}.
\end{equation*}
Note that $F_n^n$ is invertible since $F_1^1$ is invertible. By the same 
argument there exists a coalgebra morphism $F'$ with $F'G = \id$ and 
thus $F' = F' GF = F$.
\end{proof}

We saw in Proposition~\ref{prop:linftymorphismsequence} that 
the first Taylor coefficient resp. structure map $F_1$ of an $L_\infty$-morphism 
is a morphism of cochain complexes $(L,Q_1)$ and $(L,Q'_1)$, which leads us to the 
following definition.

\begin{definition}[$L_\infty$-quasi-isomorphism]
\label{def:Linftyquis}
  An \emph{$L_\infty$-quasi-isomorphism} between two $L_\infty$-algebras $(L,Q)$ and 
	$(L',Q')$ is a morphism $F$ of $L_\infty$-algebras such that the first 
	structure map $F_1$ is a quasi-isomorphism of cochain complexes. 
\end{definition}

\begin{example}[DGLA II]
  A DGLA quasi-isomorphism $\phi \colon \liealg{g} \rightarrow \liealg{g}'$ 
	is an $L_\infty$-quasi-isomorphism $F$ with only non-vanishing structure 
	map $F_1 = \phi$.
\end{example}
 
The notion of $L_\infty$-quasi-isomorphisms allows us to introduce 
the notion of formal $L_\infty$-algebras.

\begin{definition}
  An $L_\infty$-algebra $(L,Q)$ is called \emph{formal}, if there exists 
	$L_\infty$-quasi-isomorphism $(L,Q)\to (\mathrm{H}(L),Q_\mathrm{H})$, 
	where $Q_\mathrm{H}$ denotes the Lie algebra structure on the cohomolology 
	$\mathrm{H}(L)$ from Remark~\ref{rem:CohomologyLieAlg}. 
\end{definition}

\begin{example}[Kontsevich's Formality Theorem]
  In \cite{kontsevich:2003a} Kontsevich proved that there exists an 
	$L_\infty$-quasi-isomorphism 
	from the DGLA of polyvector fields $T_\poly(\mathbb{R}^d)$ to the DGLA 
	of polydifferential operators $D_\poly(\mathbb{R}^d)$, 
	compare Example~\ref{ex:tpoly} 
	and Example~\ref{ex:dpoly}. This implies that the DGLA of 
	polydifferential operators is formal, explaining the name 
	"formality theorem".
\end{example}

One reason why $L_\infty$-algebras are particularly useful is that all
$L_\infty$-quasi-isomorphisms admit \emph{$L_\infty$-quasi-inverses}, 
i.e. $L_\infty$-quasi-isomorphisms in the other direction that 
induce the inverse isomorphism in cohomology, see 
Theorem~\ref{thm:QuisInverse} below. Moreover, $L_\infty$-quasi-isomophisms 
are important for the homotopy classification of $L_\infty$-algebras, compare 
Section~\ref{sec:HomClassLinftyAlgs}, for which also another construction 
is needed: the direct sum resp. 
direct product of $L_\infty$-algebras. 

\begin{lemma}
\label{lem:DirectSumLinftyAlg}
  Let $(L,Q)$ and $(L',Q')$ be two $L_\infty$-algebras. Then 
	\begin{equation}
	   \label{eq:directsumCoder}
		 \widehat{Q}
		=
		i \circ Q \circ p + i' \circ Q' \circ p'
	\end{equation}
	defines a codifferential on $\Sym(\widehat{L}[1])$, where 
	$\widehat{L} = L \oplus L'$, $i \colon \cc{\Sym}(L[1]) \rightarrow 
	\cc{\Sym}(\widehat{L}[1])$ is 
	the inclusion	and $p \colon \cc{\Sym}( \widehat{L}[1]) \rightarrow 
	\cc{\Sym}(L[1])$ is the projection, analogously for $i',p'$. Moreover,  
	$(\widehat{L},\widehat{Q})$ is the direct product in the category 
	of conilpotent cocommutative differential graded coalgebras without counit.
\end{lemma}
\begin{proof}
Explicitly, one has for $x_1,\dots,x_m \in L$ and $x_{m+1},\dots,x_n 
\in L'$
\begin{equation*}
  \widehat{Q}^1(x_1\vee\cdots\vee x_n)
	=
	\begin{cases}
	  Q^1 (x_1\vee \cdots \vee x_n) \quad & \text{ if } m=n \\
		Q'^1 (x_1\vee \cdots\vee x_n) \quad & \text{ if } m=0 \\
		0 \quad \quad & \text{ if } 0 <m <n,
	\end{cases}
\end{equation*}
and one can directly check $\widehat{Q}\widehat{Q}=0$. In order to show that 
$(\widehat{L},\widehat{Q})$ is the direct product of $(L,Q)$ and $(L',Q')$, 
consider two morphisms $F \colon C \rightarrow 
\cc{\Sym}(L[1])$ and $F' \colon C \rightarrow \cc{\Sym}(L'[1])$, 
where $(C,D,\nabla)$ is a conilpotent cocommutative DG coalgebra without 
counit. Then $\widehat{F} \colon C \rightarrow \cc{\Sym}(\widehat{L}[1])$ 
defined by $\widehat{F}^1 = F^1 \oplus F'^1$ is the only coalgebra morphism 
with $p\widehat{F}=F$ and $p'\widehat{F}=F'$. We only have to check 
$\widehat{Q}\widehat{F} = \widehat{F}D$, where we get with 
Lemma~\ref{lemma:projectiononSnCocom}
\begin{align*}
  \widehat{Q}^1\widehat{F}
	& =
	\sum_{n> 0} \frac{1}{n!}\widehat{Q}^1_n  
	\mu_S^{n-1}(F^1 \oplus F'^1)^{\otimes n}\Delta^{(n-1)}_\sh  \\
	& =
	\sum_{n>0} \frac{1}{n!} \left(  
	Q^1_n \mu_S^{n-1}(F^1)^{\otimes n} \Delta^{(n-1)}_\sh \oplus 
	Q'^1_n \mu_S^{n-1}(F'^1)^{\otimes n} \Delta^{(n-1)}_\sh  \right) \\
	& = 
	Q^1F \oplus Q'^1F'
	=
	F^1 D \oplus F'^1 D
	= 
	\widehat{F}^1D.
\end{align*}
This shows the desired equality.
\end{proof}

\subsection{Curved $L_\infty$-algebras}
\label{subsec:CurvedLinftyAlgs}

As mentioned above, we can use the whole power of 
Theorem~\ref{thm:CofreeCocomConilpotentCoalg} by considering 
coderivations of $\Sym(L[1])$ that do not vanish on the unit, 
yielding the notion of curved $L_\infty$-algebras. 

\begin{definition}[Curved $L_\infty$-algebra]
  A \emph{curved} $L_\infty$-\emph{algebra} is a graded vector space  
	$L$ over $\mathbb{K}$ endowed with a degree one codifferential $Q$ 
	on the cofree conilpotent coalgebra $(\Sym(L[1]),\Delta_\sh)$ 
	cogenerated by $L[1]$.
\end{definition}
This codifferential $Q$ is equivalent to a sequence of maps $Q_n$ with 
$n= 0,1,\dots$, where the sum \eqref{eq:QinSymmetricviaTaylor}
starts now at $k=0$. In particular, $Q^2=0$ implies 
\begin{equation}
  Q_1(Q_0(1))
	=
	0
	\quad \text{ and } \quad
  Q_1(Q_1(x))
	=
	- Q_2(Q_0(1)\vee x),
\end{equation}
i.e. $Q_0(1)$ is always closed with respect to $Q_1$, but 
$Q_1$ is in general no longer a coboundary operator. However, 
if $Q_0(1)$ is \emph{central}, i.e. 
\begin{equation*}
  Q_{n+1}(Q_0(1)\vee x_1 \vee \cdots\vee x_n)
	=
	0
\end{equation*}
for all $n \geq 1$, then we have again $(Q_1)^2=0$. Morphisms of 
curved $L_\infty$-algebras are degree $0$ counital coalgebra morphisms $F$  
such that $F \circ Q = Q' \circ F$. As in the flat setting they are 
characterized by a sequence of maps $F_n$ with $n\geq 1$ satisfying the 
properties of Proposition~\ref{prop:linftymorphismsequence} and the fact that 
$F(1)=1$. Note that this last property is clear since we consider a 
morphism between conilpotent counital coalgebras. These have unique 
grouplike elements $1$ and $F$ has to map one to the other. 
Finally, note that a curved $L_\infty$-algebra with $Q_0=0$ is just a 
flat $L_\infty$-algebra as expected.

\begin{example}[Curved Lie algebra]
  The basic example is a curved Lie algebra 
	$(\liealg{g},R,\D,[\,\cdot\,{,}\,\cdot\,])$, i.e. a graded Lie algebra 
	with derivation $\D$ of degree $+1$ and $\D^2 = \ad(R)= [R,\argument]$ as well as $\D R=0$. 
	The element $R\in \liealg{g}^2$ is 
	also called \emph{curvature}.
	By setting $Q_0(1)=- R$ and with higher orders as in Example~\ref{ex:DGLAasLinfty} 
	we obtain a curved 
	$L_\infty$-algebra and $\D$ is a differential, i.e. $\D^2=0$, if and only if 
	$R$ is central. 
	Morphisms of curved Lie algebras are Lie algebra morphisms $f\colon \liealg{g}\rightarrow 
	\liealg{g}'$ such that $f \circ \D = \D' \circ f$ and $f(R)=R'$. 
\end{example}

\begin{remark}[Curved morphisms of curved Lie algebras]
  \label{rem:curvedMorphisms}
  Note that there exists a more general notion of \emph{curved morphisms} $(f,\alpha) 
	\colon (\liealg{g},R,\D)\rightarrow (\liealg{g}',R',\D')$ of curved 
	Lie algebras where $\alpha \in \liealg{g}'^{1}$ and where for all $x\in\liealg{g}$
	\begin{align*}
	  \D' f(x)
		=
		f(\D x) + [\alpha,f(x)]
		\quad \text{ and } \quad
		R'
		=
		f(R) + \D'\alpha - \frac{1}{2}[\alpha,\alpha],
	\end{align*}
	see \cite[Definition~4.3]{maunder:2017a}. The usual case with $\alpha=0$ is called 
	\emph{strict} and for a curved Maurer-Cartan element $x\in \liealg{g}^1$, see 
	Section~\ref{sec:MCinDGLA} for the definition, one gets a curved 
	Maurer-Cartan element $(f,\alpha)(x)=f(x)-\alpha \in \liealg{g}'^1$. In particular, 
	$(\id,\alpha)$ corresponds to twisting with $\alpha$ since 
	\begin{equation*}
	  R'
		=
		R +  \D \alpha + [\alpha,\alpha]- \frac{1}{2}[\alpha,\alpha]
		=
		R^\alpha,
	\end{equation*}
	and $(f,\alpha)$ can be seen as strict morphism into the twisted 
	curved Lie algebra $(\liealg{g}',R'^{-\alpha},\D'^{-\alpha}) =
	(\liealg{g}', R'- \D'\alpha + \frac{1}{2}[\alpha,\alpha], 
	\D' - [\alpha,\argument])$, 
	see again Section~\ref{sec:MCinDGLA} for more details on the twisting 
	procedure. 
\end{remark}

\begin{remark}[Curved morphisms of curved $L_\infty$-algebras]
  The above curved morphisms of curved Lie algebras can be generalized to 
	curved $L_\infty$-algebras by allowing zero-th Taylor coefficients
	$F_0^1 \colon \mathbb{K} \rightarrow L'[1]$ with $F_0^1(1)=\alpha\in L'[1]^0$. 
	These curved morphisms of $L_\infty$-algebras are no longer coalgebra morphisms, 
	but they still have some nice properties. In order to fully understand them we need 
	to introduce Maurer-Cartan elements and the concept of twisting, which is done in the 
	next sections. Afterwards, we will investigate the curved 
	morphisms of curved $L_\infty$-algebras in Remark~\ref{rem:CurvedMorphLinfty}, 
	see also \cite{positselski:2018a} 
	for the case of curved $A_\infty$-algebras.
\end{remark}

Concerning the compatibility of flat $L_\infty$-morphisms with the curvature 
one has the following relation:

\begin{proposition}
  \label{prop:CompatibilityMorphismCurvatur}
  Let $F$ be an $L_\infty$-morphism of flat $L_\infty$-algebras $(L,Q)$ 
	and $(L',Q')$. In addition, let $Q_0 
	\in L[1]^1, Q'_0 \in L'[1]^1$ be closed and
	central with respect to the $L_\infty$-structures $Q$ and $Q'$, 
	with induced curved $L_\infty$-structures $\widetilde{Q}$ on $L$ and 
	$\widetilde{Q'}$ on $L'$. Then the structure maps of $F$ 
	induces a morphism of curved $L_\infty$-algebras if and only if 
	\begin{equation}
	  F_1^1(Q_0) 
		=
		Q_0' 
		\quad \quad \text{ and } \quad \quad 
		F_k^1(Q_0 \vee \argument)
		=
		0
		\quad \quad \forall \; k>1.
	\end{equation}
\end{proposition}
\begin{proof} 
  One only has to check if $F$ is compatible with the codifferentials 
	$\widetilde{Q}$ and $\widetilde{Q'}$, where one gets for 
	$v_1\vee \cdots \vee v_n$ with $ v_i \in L[1]^{k_i}$ and $n>0$
	\begin{align*}
	  F^1 \circ \widetilde{Q} (v_1 \vee \cdots \vee v_n)
		& =
		F^1 ( Q_0 \vee v_1 \vee \cdots \vee v_n 
		+ Q(v_1\vee \cdots \vee v_n))  \\
		& =
		F_{n+1}^1( Q_0 \vee v_1 \vee \cdots \vee v_n)  
		+  (Q')^1\circ F (v_1\vee \cdots \vee v_n) \\
		& \stackrel{!}{=}
		\widetilde{Q'}^1 \circ F(v_1\vee \cdots \vee v_n).
	\end{align*}
	In addition, one has
	\begin{equation*}
	  Q'_0
		=
	  \widetilde{Q'} \circ F(1)
		\stackrel{!}{=}
		F \circ \widetilde{Q}(1)
		=
		F_1^1(Q_0),
	\end{equation*}
	which directly yields the above identities. 
\end{proof}

In the following, if we speak of $L_\infty$-algebras we allow curved 
$L_\infty$-algebras, in cases where flatness is required we speak of 
flat $L_\infty$-algebras. 

\section{The Homotopy Transfer Theorem and the Minimal Model of a $L_\infty$-algebra}
\label{sec:homotopytransfertheorem}

It is well-known that given a
homotopy retract one can transfer $L_\infty$-structures, see 
e.g. \cite[Section~10.3]{loday.vallette:2012a}. Explicitly, a homotopy retract (also 
called homotopy equivalence data) consists of two cochain complexes $(A,\D_A)$ and 
$(B,\D_B)$ with chain maps $i,p$ and 
homotopy $h$ such that 
\begin{equation}
  \label{eq:homotopyretract}
  \begin{tikzcd} 
	(A,\D_A)
	\arrow[rr, shift left, "i"] 
  &&   
  \left(B, \D_B\right)
  \arrow[ll, shift left, "p"]
	\arrow[loop right, "h"] 
\end{tikzcd} 
\end{equation}
with $h \circ \D_B + \D_B \circ h = \id - i \circ p$, and such that $i$ and $p$ are
quasi-isomorphisms. Then the homotopy transfer theorem
states that if there exists
a flat $L_\infty$-structure on $B$, then one can transfer it to $A$ in
such a way that $i$ extends to an $L_\infty$-quasi-isomorphism. By 
the invertibility of $L_\infty$-quasi-isomorphisms, a statement that we 
prove in Theorem~\ref{thm:QuisInverse} below, there also exists an 
$L_\infty$-quasi-isomorphism into $A$ denoted by $P$, see
e.g. \cite[Proposition~10.3.9]{loday.vallette:2012a}.

\subsection{Homotopy Transfer Theorem via Symmetric Tensor Trick}
\label{sec:HTTSymmetricTensorTrick}

We want to state different versions of this statement. 
For simplicity, we assume that we have a deformation retract, i.e. we 
are in the situation \eqref{eq:homotopyretract} with additionally
\begin{equation*}
  p\circ i
  =
  \id_A.
\end{equation*}
By \cite[Remark~2.1]{huebschmann:2011a} we can assume that we have even a 
special deformation retract, also called \emph{contraction}, where
\begin{equation*}
  h^2 
  =
  0,
  \quad \quad 
  h \circ i
  =
  0
  \quad \text{ and } \quad 
  p\circ h 
  = 
  0.
\end{equation*} 
Assume now that $(B,Q_B)$ is an $L_\infty$-algebra with $(Q_{B})_1^1=-\D_B$. 
In the following we
give a more explicit description of the transferred $L_\infty$-structure $Q_A$ on $A$ and 
of the $L_\infty$-projection $P\colon (B,Q_B)\rightarrow (A,Q_A)$ inspired by the 
symmetric tensor 
trick \cite{berglund:2014a,huebschmann:2010a,huebschmann:2011a,manetti:2010a}, 
see also \cite{merkulov:1999a} for the case of $A_\infty$-algebras.
The map $h$ extends to a homotopy $H_n
\colon \Sym^n (B[1]) \rightarrow \Sym^n(B[1])[-1]$ with respect to
$Q_{B,n}^n \colon \Sym^n (B[1]) \rightarrow \Sym^n(B[1])[1]$, see
e.g. \cite[p.~383]{loday.vallette:2012a} for the construction on the
tensor algebra, which adapted to our setting works as follows:
we define the operator 
	\begin{align*}
	K_n
	\colon 
	\Sym^n(B[1])
	\longrightarrow 
	\Sym^n(B[1])
	\end{align*}
by 
	\begin{align*}
	K_n(x_1\vee\cdots\vee x_n)
	=
	\frac{1}{n!} 
	\sum_{i=0}^{n-1}
	\sum_{\sigma\in S_n}
	\frac{\epsilon(\sigma)}{n-i}ipX_{\sigma(1)}\vee\cdots\vee ipX_{\sigma(i)}
	\vee X_{\sigma(i+1)}\vee X_{\sigma(n)}.
	\end{align*}
Note that here we sum over the whole symmetric group and 
not the shuffles, since in this case the formulas are easier. We extend $-h$ to a 
coderivation to  $\Sym(B[1])$, i.e.
	\begin{align*}
	\tilde{H}_n(x_1\vee\cdots\vee x_n):=
	-\sum_{\sigma\in \mathrm{Sh}(1,n-1)}
	\epsilon(\sigma) \;
	hx_{\sigma(1)}\vee x_{\sigma(2)}\vee\cdots\vee x_{\sigma(n)}.
	\end{align*}
\begin{lemma}
With the definition above, we have
\begin{align}
 \label{eq:extendedHomotopy}
    K_n\circ Q_{B,n}^n
	=
	Q_{B,n}^n\circ K_n
	 \quad 
	\text{ and } 
	\quad
	K_n\circ \tilde{H}_n
	=
	\tilde{H}_n\circ K_n
	,
	\end{align}
	where $Q_{B,n}^n$ is the extension of the differential $Q_{B,1}^1 = - \D_B$ to 
$\Sym^n(B[1])$ as a coderivation.
\end{lemma}
\begin{proof}
Since $i$ and $p$ are chain maps, it is clear that 
$K_n\circ Q_{B,n}^n=Q_{B,n}^n\circ K_n$. The second equation follows from the fact that
 we have $h\circ i=0$ and $p\circ h=0$. 
 \end{proof}
With the definition 
	\begin{align*}
	H_n
	:=
	K_n\circ \tilde{H}_n
	=
	\tilde{H}_n\circ K_n
	\end{align*}
we have, together with equations \eqref{Eq: ExtHomotopy}, that
	\begin{align*}
	Q_{B,n}^n H_n + H_n Q_{B,n}^n
	=
	(n\cdot\id-ip)\circ K_n,
	\end{align*}
where $ip$ is extended as a coderivation to $\Sym(B[1])$.

\begin{proposition}\label{Prop: CoAlgHom}
In the above setting one has
\begin{equation}
	\label{Eq: ExtHomotopy}
  	Q_{B,n}^n H_n + H_n Q_{B,n}^n
	=
	\id - (ip)^{\vee n}.
\end{equation}
\end{proposition}
\begin{proof}
Because of the previous results, it is enough to show that
$\id - (ip)^{\vee n}=(n\cdot\id-ip)\circ K_n$.  
Let $X_1\vee\dots \vee X_n\in \Sym^n(B[1])$, then we have 
	\begin{align*}
	(n\cdot\id-&ip)\circ K_n (X_1\vee\dots \vee X_n)\\&
	=
	\frac{1}{(n-1)!} 
	\sum_{i=0}^{n-1}
	\sum_{\sigma\in S_n}
	\frac{\epsilon(\sigma)}{n-i}ipX_{\sigma(1)}\vee\cdots\vee ipX_{\sigma(i)}
	\vee X_{\sigma(i+1)}\vee X_{\sigma(n)}\\&
	\quad-
	\frac{1}{n!} 
	\sum_{i=0}^{n-1}
	\sum_{\sigma\in S_n}
	\frac{i\,\epsilon(\sigma)}{n-i}ipX_{\sigma(1)}\vee\cdots\vee ipX_{\sigma(i)}
	\vee X_{\sigma(i+1)}\vee X_{\sigma(n)}\\&
	\quad-
	\frac{1}{n!} 
	\sum_{i=0}^{n-1}
	\sum_{\sigma\in S_n}
	\frac{\epsilon(\sigma)}{n-i}ipX_{\sigma(1)}\vee\cdots\vee ipX_{\sigma(i)}
	\vee ip(X_{\sigma(i+1)}\vee X_{\sigma(n)})\\&
	=
	\frac{1}{n!} 
	\sum_{i=0}^{n-1}
	\sum_{\sigma\in S_n}
	\epsilon(\sigma)ipX_{\sigma(1)}\vee\cdots\vee ipX_{\sigma(i)}
	\vee X_{\sigma(i+1)}\vee X_{\sigma(n)}\\&
	\quad -
	\frac{1}{n!} 
	\sum_{i=0}^{n-1}
	\sum_{\substack{\sigma\in S_n,\\\tau \in Sh(1,n-i-1)}}\hspace*{-5pt}
	\frac{\epsilon(\sigma\circ (\id^i\times\tau))}{n-i} 
	ipX_{\sigma(1)}\vee\cdots\vee ipX_{\sigma(i)}
	\vee ipX_{\sigma(\tau(i+1))}\vee X_{\sigma(\tau(n))}\\&
	=
	\frac{1}{n!} 
	\sum_{i=0}^{n-1}
	\sum_{\sigma\in S_n}
	\epsilon(\sigma)ipX_{\sigma(1)}\vee\cdots\vee ipX_{\sigma(i)}
	\vee X_{\sigma(i+1)}\vee X_{\sigma(n)}\\&
	\quad -
	\frac{1}{n!} 
	\sum_{i=0}^{n-1}
	\sum_{\substack{\sigma\in S_n,\\\tau \in Sh(1,n-i-1)}}
	\frac{\epsilon(\sigma)}{n-i}ipX_{\sigma(1)}\vee\cdots\vee ipX_{\sigma(i)}
	\vee ipX_{\sigma(i+1)}\vee X_{\sigma(n)}\\&
	=
	\frac{1}{n!} 
	\sum_{i=0}^{n-1}
	\sum_{\sigma\in S_n}
	\epsilon(\sigma)ipX_{\sigma(1)}\vee\cdots\vee ipX_{\sigma(i)}
	\vee X_{\sigma(i+1)}\vee X_{\sigma(n)}\\&
	\quad -
	\frac{1}{n!} 
	\sum_{i=0}^{n-1}
	\sum_{\sigma\in S_n}
	\epsilon(\sigma)ipX_{\sigma(1)}\vee\cdots\vee ipX_{\sigma(i)}
	\vee ipX_{\sigma(i+1)}\vee X_{\sigma(n)}.
	\end{align*}
The finalization of the proof is just a comparison of the summands. 
\end{proof}
Suppose now that we already have constructed a codifferential $Q_A$ and a 
morphism of coalgebras $P$ with structure maps $P_\ell^1 \colon \Sym^\ell(B[1]) \rightarrow A[1]$ 
such that $P$ is an $L_\infty$-morphism up to order $k$, i.e.
\begin{equation*}
  \sum_{\ell=1}^m P^1_\ell \circ Q_{B,m}^\ell
  =
  \sum_{\ell=1}^m Q_{A,\ell}^1\circ P^\ell_{m}
\end{equation*}
for all $m \leq k$. Then we have the following statement.
\begin{lemma}
\label{lemma:Linftyuptok}
  Let $P \colon \Sym(B[1]) \rightarrow \Sym (A[1])$ be an
  $L_\infty$-morphism up to order $k\geq 1$ . Then
  \begin{equation}
    \label{eq:Linftykplusone}
    L_{\infty,k+1}
    =
    \sum_{\ell = 2}^{k+1} Q_{A,\ell}^1 \circ P^\ell_{k+1} - \sum_{\ell =1}^{k} 
    P_\ell^1 \circ 	Q^\ell_{B,k+1}
  \end{equation}
  satisfies
  \begin{equation}
    \label{eq:linftycommuteswithq}
    L_{\infty,k+1} \circ Q_{B,k+1}^{k+1}
    =
    -Q_{A,1}^1 \circ L_{\infty,k+1}.
  \end{equation}
\end{lemma}
\begin{proof}
  The statement follows from a straightforward computation. For
  convenience we omit the index of the differential:
  \begin{align*}
    L_{\infty,k+1} Q_{k+1}^{k+1}
    & = 
    \sum_{\ell = 2}^{k+1} Q_{\ell}^1 (P\circ Q)^\ell_{k+1}  
    - \sum_{\ell = 2}^{k+1} \sum_{i=1}^k Q_{\ell}^1 P^\ell_i Q^i_{k+1}
    + \sum_{\ell =1}^{k}\sum_{i=1}^k P_\ell^1 Q^\ell_i Q_{k+1}^i  \\
    & =
    \sum_{\ell = 2}^{k+1} Q_{\ell}^1 (Q\circ P)^\ell_{k+1}  
    - \sum_{\ell = 2}^{k+1} \sum_{i=1}^k Q_{\ell}^1 P^\ell_i Q^i_{k+1}
    + \sum_{\ell =1}^{k}\sum_{i=1}^k Q_\ell^1 P^\ell_i Q_{k+1}^i  \\
    & =
    -Q_1^1 (Q \circ P)^1_{k+1} + Q_1^1 \sum_{i=1}^k P^1_{i}Q^i_{k+1} 
    = 
    - Q_1^1 L_{\infty,k+1},
  \end{align*}
  where the last equality follows from $Q_1^1Q_1^1 = 0$. 
\end{proof}

This allows us to prove one version of the homotopy transfer theorem 
as formulated in \cite[Theorem~B.2]{esposito.kraft.schnitzer:2022a:pre}.

\begin{theorem}[Homotopy transfer theorem]
  \label{thm:HTTJonas}
  Let $(B,Q_B)$ be a flat $L_\infty$-algebra with $(Q_B)^1_1=-\D_B$ and contraction 
	\begin{equation}
    \begin{tikzcd} 
	  (A,\D_A)
	  \arrow[rr, shift left, "i"] 
    &&   
    \left(B, \D_B\right)
    \arrow[ll, shift left, "p"]
	  \arrow[loop right, "h"] 
  \end{tikzcd} 
  \end{equation}
	Then 
	\begin{align}
	  \begin{split}
		\label{eq:HTTTransferredStructures}
		  (Q_A)_1^1
			& =
			- \D_A,
			\quad \quad\quad\quad\;
			(Q_A)^1_{k+1}
			=
			\sum_{i=1}^{k} P^1_i \circ (Q_B)^i_{k+1} \circ i^{\vee (k+1)},\\
			P_1^1
			& =
			p, 
			\quad \quad\quad\quad\quad\quad\;\;
			P^1_{k+1}
			=
			L_{\infty,k+1}\circ H_{k+1}
			\quad \text{ for  } k\geq 1
		\end{split}
	\end{align}
	turns $(A,Q_A)$ into an $L_\infty$-algebra with 
	$L_\infty$-quasi-isomorphism $P\colon (B,Q_B)\rightarrow (A,Q_A)$. 
	Moreover, one has $P^1_k \circ i^{\vee k} =0$ for $k\neq 1$.
\end{theorem}
\begin{proof}
 We observe $P_{k+1}^1(ix_1 \vee \cdots \vee i x_{k+1}) = 0$ for all
 $k\geq 1$ and $x_i \in A$, which directly follows from $h\circ i =
 0$ and thus $H_{k+1} \circ i^{\vee (k+1)} = 0$. Suppose that 
$Q_A$ is a codifferential up to order $k\geq 1$, i.e. 
$\sum_{\ell=1}^m (Q_A)^1_\ell(Q_A)^\ell_m=0$ for all $m \leq k$, 
and that $P$ is an $L_\infty$-morphism 
up to order $k\geq 1$. We know that these conditions are satisfied for $k=1$ and we show that they hold for $k+1$. Starting with $Q_A$ we compute
\begin{align*}
  (Q_AQ_A)^1_{k+1}
	& =
	(Q_AQ_A)^1_{k+1} \circ P^{k+1}_{k+1}\circ i^{\vee (k+1)} 
	=
	\sum_{\ell=1}^{k+1} (Q_AQ_A)^1_\ell P^\ell_{k+1} i^{\vee (k+1)} 
	=
	(Q_AQ_AP)^1_{k+1}i^{\vee (k+1)} \\
	& =
	\sum_{\ell=2}^{k+1} (Q_A)^1_\ell (Q_AP)^\ell_{k+1} i^{\vee (k+1)} 
	+(Q_A)^1_1(Q_AP)^1_{k+1} i^{\vee (k+1)}  \\
	& =
	\sum_{\ell=2}^{k+1} (Q_A)^1_\ell (PQ_B)^\ell_{k+1} i^{\vee (k+1)} 
	+(Q_A)^1_1(Q_A)^1_{k+1}  \\
	& =
	(Q_APQ_B)^1_{k+1}i^{\vee (k+1)}  -(Q_A)^1_1(Q_A)^1_{k+1} + (Q_A)^1_1(Q_A)^1_{k+1} \\
	& =
	\sum_{\ell=1}^{k} (Q_AP)^1_\ell (Q_B)^\ell_{k+1}i^{\vee (k+1)} 
	+ (Q_AP)^1_{k+1} (Q_B)^{k+1}_{k+1}i^{\vee (k+1)} \\
	& = 
	\sum_{\ell=1}^{k} (PQ_B)^1_\ell (Q_B)^\ell_{k+1}i^{\vee (k+1)} 
	+ (Q_AP)^1_{k+1}i^{\vee (k+1)}  (Q_A)^{k+1}_{k+1}\\
	& =
	- (PQ_B)^1_{k+1}i^{\vee (k+1)}  (Q_A)^{k+1}_{k+1} + 
	(Q_AP)^1_{k+1}i^{\vee (k+1)}  (Q_A)^{k+1}_{k+1} \\
	& =
	- (Q_A)^1_{k+1}(Q_A)^{k+1}_{k+1} + 
	(Q_A)^1_{k+1}(Q_A)^{k+1}_{k+1}
	=
	0.
\end{align*}
By the same computation as in Lemma~\ref{lemma:Linftyuptok}, where one in fact only needs 
that $Q_A$ is a codifferential up to order $k+1$, it follows that 
\begin{equation*}
  L_{\infty,k+1} \circ Q_{B,k+1}^{k+1}
  =
  -Q_{A,1}^1 \circ L_{\infty,k+1}.
\end{equation*}
It remains to show that $P$ is an $L_\infty$-morphism
 up to order $k + 1$. We have 
  \begin{align*}
    P_{k+1}^1 \circ (Q_B)^{k+1}_{k+1}
    & = 
    L_{\infty,k+1} \circ H_{k+1} \circ (Q_B)_{k+1}^{k+1} \\
    &=
    L_{\infty,k+1} - L_{\infty,k+1} \circ (Q_B)_{k+1}^{k+1}\circ H_{k+1} 
    -L_{\infty,k+1} \circ (i\circ p)^{\vee (k+1)} \\
    & =
    L_{\infty,k+1} + (Q_A)_1^1 \circ P_{k+1}^1 
  \end{align*}
	since
	\begin{align*}
	  L_{\infty,k+1} \circ (i\circ p)^{\vee (k+1)}
		& =
		\left(\sum_{\ell = 2}^{k+1} Q_{A,\ell}^1 \circ P^\ell_{k+1} - \sum_{\ell =1}^{k} 
    P_\ell^1 \circ 	Q^\ell_{B,k+1}\right) \circ (i\circ p)^{\vee (k+1)}  \\
		& =
		(Q_A)^1_{k+1} \circ p^{\vee (k+1)} 
		- (Q_A)^1_{k+1}\circ p^{\vee(k+1)}
		=
		0.
	\end{align*}
  Therefore
  \begin{equation*}
    P_{k+1}^1 \circ (Q_B)^{k+1}_{k+1} -  (Q_A)_1^1 \circ P_{k+1}^1
    =
    L_{\infty,k+1},
  \end{equation*}
  i.e. $P$ is an $L_\infty$-morphism up to order $k+1$, and the
  statement follows inductively.
\end{proof}

\begin{remark}[HTT vs. HPL]
In view of Lemma \ref{Prop: CoAlgHom}, we can define $H\colon \Sym(B[1])\to \Sym(B[1])$ by 
$H\at{\Sym^n(B[1])}=H_n$ to obtain 
\begin{equation*}
  \begin{tikzcd} 
	(S(A[1]),\hat{Q}_A)
	\arrow[rr, shift left, "I"] 
  &&   
  \left(S(B[1]), \hat{Q}_B\right)
  \arrow[ll, shift left, "P"]
	\arrow[loop right, "H"] 
\end{tikzcd} 
\end{equation*}
where $I$ (resp. $P$) is the map $i$ (resp. $p$) extended as a coalgebra morphism
and $\hat{Q}_A$ (resp. $\hat{Q}_B$) is the differential $-\D_A$ (resp. $-\D_B$) extended as a coderivation. 
Note that if $h^2=0$ we also have $H^2=0$.  Now we can see the higher brackets $Q_B-\hat{Q}_B$ as a pertubation of the coderivation $\hat{Q}_B$. 
Since $(Q_B-\hat{Q}_B)\circ H$ decreases the symmetric degree, 
it is locally nilpotent, and we can apply the homological pertubation lemma, see e.g. 
\cite{crainic:2004} and references therein. 
Note, that it is not clear why the resulting deformation retract 
gives maps which are compatible with the coalgebra structure.  
\end{remark}

With the homotopy transfer theorem we can show that every contraction 
induces a splitting of the $L_\infty$-algebra in the following sense, 
compare Lemma~\ref{lem:DirectSumLinftyAlg} for the notion of a
direct sum of $L_\infty$-algebras.

\begin{theorem}
\label{thm:ContractionSplitting}
  In the above setting one has an $L_\infty$-isomorphism 
	\begin{equation}
	  L \colon
		B
		\longrightarrow
		A \oplus \image [\D_B,h],
	\end{equation}
	where the $L_\infty$-structure on $\image [\D_B,h]$ is given by just the differential 
	$Q_1^1= -\D_B$ and the $L_\infty$-structure on $A \oplus \image [\D_B,h]$ is the product 
	$L_\infty$-structure of the transferred one on $A$ and the differential on 
	$ \image [\D_B,h]$.
\end{theorem}
\begin{proof}
By Theorem~\ref{thm:HTTJonas} we already have an $L_\infty$-morphism $P\colon B 
\rightarrow A$ with first structure map $p$. Now we construct an $L_\infty$-morphism 
$F \colon B \rightarrow \image [\D_B,h] = C$ by setting
\begin{equation*}
  F_1^1
	=
	[\D_B,h]
	\quad\quad \quad \text{ and }\quad
	\quad \quad 
	F^1_n
	=
	- h \circ \sum_{i=1}^{n-1} F_i^1 (Q_B)^i_n 
	\quad \text{ for } n>1.
\end{equation*}
It is an $L_\infty$-morphism up to order one. Suppose it is one up to order 
$n\geq 1$, then we get
\begin{align*}
  (Q_C)_1^1 F^1_{n+1}
	& =
	\D_B\at{C} \circ h \circ \sum_{i=1}^{n} F_i^1 (Q_B)^i_{n+1}  
	= 
	(\id - ip - h \circ\D_B\at{C})\sum_{i=1}^{n} F_i^1 (Q_B)^i_{n+1}\\
	& =
	\sum_{i=1}^{n} F_i^1 (Q_B)^i_{n+1} + h \circ \sum_{i=1}^{n} (FQ_B)^1_i (Q_B)^i_{n+1} 
	=
	\sum_{i=1}^{n+1} F_i^1 (Q_B)^i_{n+1},
\end{align*}
thus by induction $F$ is an $L_\infty$-morphism. 
The universal property of the product gives the desired 
$L_\infty$-morphism $L= P\oplus F$ which is even an $L_\infty$-isomorphism since its first 
structure map $p \oplus (\D_B h + h\D_B)$ is an isomorphism 
with inverse $i \oplus \id$, see Theorem~\ref{prop:Linftyiso}. 
\end{proof}

We also want to give an explicit formula for a $L_\infty$-quasi-inverse of 
$P$, where we follow \cite[Proposition~B.3]{esposito.kraft.schnitzer:2022a:pre}.

\begin{proposition}
  \label{prop:Infinityinclusion}
  The coalgebra map $I\colon \Sym^\bullet (A[1])\to \Sym^\bullet
  (B[1])$ recursively defined by the maps $I_1^1=i$ and $I_{k+1}^1=h\circ
  L_{\infty,k+1}$ for $k\geq 1$ is an $L_\infty$-quasi inverse of $P$.
  Since $h^2=0= h\circ i$, one even has $I_{k+1}^1 = h \circ \sum_{\ell = 2}^{k+1} Q_{B,\ell}^1 \circ I^\ell_{k+1}$ and $P \circ I = \id_A$.
\end{proposition}
\begin{proof}
	We proceed by induction: assume that $I$ is an $L_\infty$-morphism up to 
	order $k$, then we have 
	\begin{align*}
		I^1_{k+1}Q_{A,k+1}^{k+1} - Q_{B,1}^1I^1_{k+1}&= 
		-Q_{B,1}^1\circ h\circ L_{\infty,{k+1}}
		+h\circ L_{\infty,{k+1}}\circ Q_{A,k+1}^{k+1}\\&
		=-Q_{B,1}^1\circ h\circ L_{\infty,{k+1}}-h\circ Q_{B,1}^{1}\circ L_{\infty,{k+1}}\\&
		=(\id-i\circ p)L_{\infty,{k+1}}.
	\end{align*}
	We used that $Q_{B,1}^1=-\D_B$ and the homotopy equation of $h$. 
	Moreover, we get with $p\circ h=0$
	\begin{align*}
	  p \circ L_{\infty,{k+1}}
		& = 
		p\circ  \left( \sum_{\ell = 2}^{k+1} Q_{B,\ell}^1 \circ I^\ell_{k+1} 
		- \sum_{\ell =1}^{k}   I_\ell^1 \circ 	Q^\ell_{A,k+1} \right) \\
		& =
		\sum_{\ell = 2}^{k+1} (P\circ Q_B)^1_\ell \circ I^\ell_{k+1} 
		- \sum_{\ell =2}^{k+1}\sum_{i=2}^{\ell} P^1_i\circ Q^i_{B,\ell} \circ 
		I^\ell_{k+1} - 
		Q^1_{A,k+1} \\
		& =
		\sum_{\ell = 2}^{k+1} (Q_A \circ P)^1_\ell \circ I^\ell_{k+1} 
		- \sum_{i =2}^{k+1}\sum_{\ell=i}^{k+1} P^1_i\circ Q^i_{B,\ell} \circ 
		I^\ell_{k+1} - 		Q^1_{A,k+1} \\
		& =
		Q^1_{A,k+1} - \sum_{i =2}^{k+1}\sum_{\ell=i}^{k+1} 
		P^1_i\circ I^i_{\ell} \circ 	Q^\ell_{A,k+1} - 		Q^1_{A,k+1}
		=
		0,
	\end{align*}
	and therefore $I$ is an $L_\infty$-morphism. 
\end{proof}

\begin{remark}
  Note that in the homotopy transfer theorem the property $h^2=0$ is not 
	needed, and that one can also adapt the above construction of $I$ 
	to this more general case.
\end{remark}

%
%
%

Following \cite{esposito.kraft.schnitzer:2020a:pre}, 
let us now consider the special case of contractions of
DGLAs. More explicitly, let now $A,B$ be two DGLAs and assume in addition 
that $i$ is a DGLA morphism. Then the homotopy transfer theorem immediately yields:
\begin{proposition}
  \label{prop:Infinityprojection}
  Defining $P_1^1 = p$ and $P_{k+1}^1 = L_{\infty,k+1} \circ H_{k+1}$
  for $k \geq 1$ yields an $L_\infty$-quasi-isomorphism $P \colon
  (B,Q_B) \rightarrow (A,Q_A)$ that is quasi-inverse to $i$. Here the codifferentials are 
	induced by the respective DGLA structures.
\end{proposition}
\begin{proof}
  The transferred $L_\infty$-structure on $A$ from \eqref{eq:HTTTransferredStructures} 
	is indeed just the DGLA structure since $i$ is a DGLA morphism.
\end{proof}

Let us now assume that $p\colon B\to A$ in the contraction
\eqref{eq:homotopyretract} is a DGLA morphism and that $i$ is just a chain
map. Then we can analogously give a formula for the extension $I$ of
$i$ to an $L_\infty$-quasi-isomorphism.

\begin{proposition}
  \label{prop:InfinityinclusionDGLA}
  The coalgebra map $I\colon \Sym^\bullet (A[1])\to \Sym^\bullet
  (B[1])$ recursively defined by the maps $I_1^1=i$ and $I_k^1=h\circ
  L_{\infty,k}$ for $k\geq 2$ is an $L_\infty$-quasi inverse of $p$.
  Since $h^2=0= h\circ i$, one even has $I_k^1 = h \circ Q^1_{2} \circ
  I^2_{k}$.
\end{proposition}
\begin{proof}
	This is just a special case of Proposition~\ref{prop:Infinityinclusion}.
\end{proof}

\subsection{The Minimal Model and the Existence of $L_\infty$-Quasi-Inverses}

As a first application of the homotopy transfer theorem we want to show
 that $L_\infty$-algebras split into the direct product 
of two special ones, whence we recall some definitions.

\begin{definition}
  An $L_\infty$-algebra $(L,Q)$ is called \emph{minimal} if $Q_1^1 = 0$ 
	and \emph{linear contractible} if $Q_n^1=0$ for $n>1$ with acyclic $Q_1^1$.
\end{definition}
An $L_\infty$-algebra is called contractible if it is isomorphic to 
a linear contractible one, and we can show that every $L_\infty$-algebra 
is isomorphic to the direct sum of these two types 
\cite[Proposition~2.8]{canonaco:1999a}:

\begin{proposition}
  \label{prop:StandardFormLinfty}
	Any $L_\infty$-algebra is isomorphic to the direct sum of a 
	minimal and of a linear contractible one.
\end{proposition}
\begin{proof} Let $(L,Q)$ be a $L_\infty$-algebra and let 
$(L,\D = - Q^1_1)$ be the underlying cochain complex. Since we are 
working over a field $\mathbb{K}$ of characteristic zero, we can find a deformation retract

\begin{equation}
    \begin{tikzcd} 
	  (\mathrm{H}(L),0)
	  \arrow[rr, shift left, "i"] 
    &&   
    \left(L, \D\right)
    \arrow[ll, shift left, "p"]
	  \arrow[loop right, "h"]
  \end{tikzcd} 
  \end{equation}
where $\mathrm{H}(L)$ denotes the cohomology of $(L,\D)$.
Then we apply Theorem \ref{thm:ContractionSplitting} and get the result.
\end{proof}

Denoting the transferred minimal $L_\infty$-structure on 
$\mathrm{H}(L)$ by $Q_\mathrm{H}$, the above proposition gives in particular 
an $L_\infty$-quasi-isomorphism 
$(L,Q) \to (\mathrm{H}(L),Q_\mathrm{H})$. For this reason, 
$(\mathrm{H}(L),Q_\mathrm{H})$ is also called \emph{minimal model} of 
$(L,Q)$. This result allows us to explicitly invert 
$L_\infty$-quasi-isomorphisms.

\begin{theorem}
  \label{thm:QuisInverse}
  If $F$ is an $L_\infty$-quasi-isomorphism from $(L,Q)$ to $(L',Q')$, 
	then there exists an $L_\infty$-morphism $G$ in the other direction, inducing 
	the inverse isomorphism in cohomology.
\end{theorem}
\begin{proof}
We know that both $L$ and $L'$ are isomorphic to direct sums of minimal 
$L_\infty$-algebras $L_{min}$ and $L'_{min}$ and linear 
contractible ones. In particular, the inclusions and projections 
\begin{equation*}
  i \colon
	(L_{min},Q_{min}) \longrightarrow (L,Q),
	\quad \quad \quad
	p \colon 
	(L,Q) \longrightarrow (L_{min},Q_{min})
\end{equation*}
are $L_\infty$-quasi-isomorphisms, analogously for $i'$ and $p'$. 
In particular, 
\begin{equation*}
  F_{min} = p'Fi \colon
	(L_{min},Q_{min}) \longrightarrow (L'_{min},Q'_{min})
\end{equation*}
is an $L_\infty$-quasi-isomorphism. But since $(Q_{min})_1^1 = 0 
=(Q'_{min})_1^1$ we know that $(F_{min})_1^1$ is an isomorphism, thus $F_{min}$ is 
an $L_\infty$-isomorphism by Proposition~\ref{prop:Linftyiso}, and we can 
set $G= i(F_{min})^{-1} p'$ for the $L_\infty$-quasi-isomorphism in the 
other direction.
\end{proof}

%
%
\section{Maurer-Cartan Elements, their Equivalence Classes and Twisting}
\label{sec:MCandEquiv}

In this section we recall the notion of Maurer-Cartan elements and different 
notions of equivalences between them. One important application is the 
twisting of $L_\infty$-algebras and $L_\infty$-morphisms.

%
%
\subsection{Maurer-Cartan Elements in DGLAs}
\label{sec:MCinDGLA}

We start with the case of Maurer-Cartan elements in DGLAs.

\begin{definition}[Maurer-Cartan elements]
  \label{def:MCelementDGLA}
  Let $(\liealg{g},\D ,[\argument{,}\argument])$ be a DGLA. Then 
	$\pi \in \liealg{g}^1$ is called \emph{Maurer-Cartan element} if it 
	satisfies the Maurer-Cartan equation
	\begin{equation}
	  \D \pi + \frac{1}{2}[\pi,\pi]
		=
		0.
	\end{equation}
\end{definition}
The set of Maurer-Cartan elements is denoted by $\Mc(\liealg{g})$ and 
we directly see that for a Maurer-Cartan element $\pi$ the map 
$\D + [\pi,\,\cdot\,]$ is again a differential, the so-called 
\emph{twisted} differential. Taking a general element $x \in \liealg{g}^1$ 
the derivation $\D + [x,\,\cdot\,]$ yields a curved Lie algebra with curvature 
\begin{equation*}
  R^x 
	=
	\D x + \frac{1}{2}[x,x].
\end{equation*}
Starting with a curved Lie algebra $(\liealg{g},R,\D,[\argument{,}\argument])$, 
the twisting yields the new curvature $R^x = R +\D x + \frac{1}{2}[x,x]$ and one calls $\pi \in \liealg{g}^1$ 
\emph{curved Maurer-Cartan element} if one has 
\begin{equation*}
  R^\pi
	=
	R + \D \pi + \frac{1}{2} [\pi,\pi]
	=
	0.
\end{equation*}
In this case the twisted DGLA $(\liealg{g}, R^\pi=0, \D + [\pi,\argument],
[\argument{,}\argument])$ is a flat DGLA. One example for 
curved Maurer-Cartan elements are principal connections on principal bundles, 
see e.g. \cite[Section~11]{kolar.michor.slovak:1993a} for more details.
 
\begin{example}[Principal Connection]
  Let $\group{G}$ be a Lie group and $\pi \colon P \to M$ be a 
	smooth principal $\group{G}$-bundle. Then one way to define a 
	\emph{principal $\group{G}$-connection} on $P$ is as an 
	equivariant $\liealg{g}$-valued differential $1$-form 
	$\omega \in \Omega^1(P,\liealg{g})^\group{G}$, where the 
	equivariance is taken with respect to the product of the action on 
	$P$ and the adjoint action on the Lie algebra $\liealg{g}$ of 
	$\group{G}$, satisfying $\omega(\xi_P) = \xi$ for all $\xi \in \liealg{g}$, 
	where $\xi_P$ denotes the fundamental vector field. The \emph{curvature form} 
	$\Omega \in \Omega^2(P,\liealg{g})^\group{G}$ is given by
	\begin{equation}
	  \Omega
		=
		\D \omega + \frac{1}{2} [\omega,\omega],
	\end{equation}
	where $\D$ denotes the de Rham differential, and where the Lie bracket 
	$[\argument{,}\argument]$ is induced by the $\wedge$ product of ordinary 
	differential forms and the Lie bracket on $\liealg{g}$. 
	In other words, the principal connection $\omega$ satisfies the Maurer-Cartan equation 
	in the curved Lie algebra $(\Omega(P,\liealg{g})^\group{G}, -\Omega, 
	\D,[\argument{,}\argument])$.
\end{example}

Two other examples of DGLAs and Maurer-Cartan elements that 
we want to mention are the DGLAs of polyvector fields 
and the DGLA of polydifferential operators. They play an important role in 
(formal) deformation quantization \cite{bayen.et.al:1978a,kontsevich:2003a,waldmann:2007a}.

\begin{example}[Polyvector fields]
  \label{ex:tpoly}
  Let $M$ be a smooth manifold. The \emph{polyvector fields} are 
  the sections $\Tpoly^\bullet(M)=\Secinfty(\Anti^{\bullet +1} TM)$. 
  Together with the Schouten bracket $[\argument{,}\argument]_S$ they form 
  a graded Lie algebra, and together with the zero differential a DGLA
  \begin{equation}
    T_\poly^\bullet(M)
		=
		(\Secinfty(\Anti^{\bullet +1} TM),0,[\argument{,}\argument]_S).
  \end{equation}
  The Maurer-Cartan elements are given by bivectors $\pi \in 
  \Secinfty(\Anti^2 TM)$ satisfying the Maurer-Cartan equation $[\pi,\pi]_S=0$, 
  i.e. by Poisson structures. We denote 
	by $\{\argument,\argument\}$ the corresponding Poisson brackets on $\Cinfty(M)$. 
	In deformation quantization one is interested in formal deformations, 
	whence one considers formal 
	polyvector fields $\Tpoly^\bullet(M)[[\hbar]]
	= \Secinfty(\Anti^{\bullet +1} TM)[[\hbar]]$ in the real formal parameter $\hbar$. 
	Formal Poisson structures are then formal power series
	\begin{equation}
	  \pi_\hbar
		=
		\pi_0 + \hbar \pi_1 + \cdots \in
		\Secinfty(\Anti^2 TM)[[\hbar]]
	\end{equation}
	with $[\pi_\hbar,\pi_\hbar]_S = 0$. In lowest order this implies in particular 
	$[\pi_0,\pi_0]_S=0$. Two such formal Poisson structures $\pi_\hbar$ and 
	$\widetilde{\pi}_\hbar$ are called \emph{equivalent} if there exists a 
	formal diffeomorphism such that 
	\begin{equation}
	  \pi_\hbar
		=
		\exp(\hbar [X,\,\cdot\,]_S) \widetilde{\pi}_\hbar,
	\end{equation}
	where $X\in \Secinfty(TM)[[\hbar]]$. In particular, $\pi_\hbar$ and 
	$\widetilde{\pi}_\hbar$ deform the same $\pi_0$. In view of later applications, 
	it turns out to be useful to 
	consider formal Poisson structures $\pi_\hbar = \hbar\pi_1 + \cdots$ that 
	start in the first order of $\hbar$, i.e. that deform the zero Poisson structure. 
	Consequently, one considers formal Maurer-Cartan elements $\Mc^\hbar 
	= \Mc \cap \hbar \Tpoly^1(M)[[\hbar]]$. We denote the equivalence classes 
	by 
	\begin{equation}
	  \Def(\Tpoly(M)[[\hbar]])
		=
		\frac{\Mc^\hbar(\Tpoly(M))}{\group{G}^0(\Tpoly(M)[[\hbar]])},
	\end{equation}
	where $\group{G}^0(\Tpoly(M)[[\hbar]])= 
	\{ \exp([X,\,\cdot\,]_S) \mid X \in \hbar \Tpoly^0(M)[[\hbar]]\}$ is called 
	\emph{gauge group}. 
\end{example}

In deformation quantization, the polyvector fields with their formal Poisson structures as Maurer-Cartan elements 
corresponds to the classical side. From the quantum point of view one looks 
for star products on $\Cinfty(M)[[\hbar]]$ that can be interpreted as 
formal power series of bidifferential operators. Before we can introduce 
the corresponding DGLA of polydifferential operators we have to 
recall the Hochschild cochains:

\begin{example}[Hochschild cochains]
  \label{ex:Hochschild}
  For a unital associative algebra $(A,\mu_0,1)$ we recall the Hochschild cochains 
with a shifted grading
\begin{equation}
  C^n(A)
	=
	\Hom(A^{\otimes(n+1)},A)
\end{equation}
for $n \geq 0$ and $C^{-1}(A)=A$. There are different 
operations for cochains $D$ and $E$, the \emph{cup-product}
\begin{equation}
  D \cup E (a_0,\dots,a_{d+e+1})
	=
	D(a_0,\dots, a_d)E(a_{d+1},\dots,a_{d+e+1}),
\end{equation}
where $a_0,\dots,a_{d+e+1}\in A$, the concatenation
\begin{equation}
  D \circ E (a_0,\dots,a_{d+e}) 
  = 
	\sum_{i=0}^{\abs{ D}} (-1)^{i \abs{ E}} 
	D(a_0,\dots, a_{i-1}, E(a_i,\dots,a_{i+e}),a_{i+e+1},\dots,a_{d+e})
\end{equation}
and the \emph{Gerstenhaber bracket}
\begin{equation}
  \label{eq:GerstenhaberBracketClassical}
  [D,E]_G
	=
	(-1)^{\abs{E}\abs{D}} \left(D \circ E - (-1)^{\abs{D} \abs{E}} E \circ D\right).
\end{equation}
Note that we use the sign convention from \cite{bursztyn.dolgushev.waldmann:2012a}, 
not the original one from \cite{gerstenhaber:1963a}.

This yields a graded Lie algebra $(C^{\bullet}(A),[\argument{,}\argument]_G)$ 
and a graded associative algebra $(C^{\bullet-1}(A),\cup)$. 
The product $\mu_0$ is an element of $C^1(A)$ 
and one notices that the associativity of $\mu_0$ is equivalent to 
$[\mu_0,\mu_0]_G=0$, compare Remark~\ref{rem:Ainftystructures}. 
In particular, we get an induced differential 
$\del \colon C^\bullet(A) \rightarrow C^{\bullet +1}(A)$ via 
$\del = [\mu_0,\argument]_G$, the so-called \emph{Hochschild differential}, 
and thus a DLGA structure. One can check that $\pi \in C^1(A)$ is a Maurer-Cartan 
element if and only if $\mu_0 + \pi$ is an associative product. 
The cohomology of $(C^\bullet(A),\del)$ 
inherits even the structure of a Gerstenhaber algebra:
More explicitly, in cohomology 	the $\cup$-product is graded commutative 
and one has the following Leibniz rule:
\begin{equation*}
  [F,D \cup E]_G
	=
	[F,D]_G \cup E + (-1)^{\abs{F} (\abs{D} -1)}  D \cup [F,E]_G,
\end{equation*}
compare \cite{gerstenhaber:1963a} and \cite[Satz~6.2.18]{waldmann:2007a}.
\end{example}

In the context of deformation quantization \cite{bayen.et.al:1978a} we 
are interested in star products. They are Maurer-Cartan elements in the 
DGLA of polydifferential operators, i.e. in the differential Hochschild cochain complex:

\begin{example}[Polydifferential operators]
  \label{ex:dpoly}
  Recall that a star product $\star$ on a Poisson manifold $(M,\pi)$ is an associative 
	product on $\Cinfty(M)[[\hbar]]$ of the form 
	\begin{equation}
	  f \star g 
		=
		\mu_0(f, g) + \sum_{r=1}^\infty \hbar^r C_r(f,g) \in 
		\Cinfty(M)[[\hbar]],
	\end{equation}
	where $C_1(f,g)-C_1(g,f) =  \{f,g\}$, and where all $C_r$ are bidifferential 
	operators vanishing on constants. Two star products $\star$ and $\star'$ 
	are equivalent if there 
	exists an $\hbar$-linear isomorphism $S = \id + \sum_{r=1}^\infty \hbar^r S_r$
	with differential operators $S_r$ and
	\begin{equation}
	  S(f \star g)
		=
		Sf \star' Sg.
	\end{equation}
	To describe the notion of star products in terms of a DGLA we consider 
	the associated Hochschild DGLA $(C^\bullet(\Cinfty(M)),\del,
	[\argument{,}\argument]_G)$. To incorporate the bidifferentiability we restrict 
	ourselves to the \emph{polydifferential operators}
	\begin{equation}
	  \Dpoly^\bullet(M)[[\hbar]]
		=
		\bigoplus_{n=-1}^\infty \Dpoly^{n}(M)[[\hbar]],
	\end{equation}
	were $\Dpoly^{n}(M)[[\hbar]]= \Hom_{\mathrm{diff}}(\Cinfty(M)^{\otimes n+1},
	\Cinfty(M))[[\hbar]]$ are polydifferential operators vanishing on constants. 
	We know that star products can be interpreted as $\star = \mu_0 + \sum_{r=1}^\infty 
	\hbar^r C_r= \mu_0 +\hbar m_\star \in \Dpoly^{1}(M)[[\hbar]]$ and the associativity leads to 
	\begin{equation}
	  \label{eq:starproductmc}
	  0
		=
		[\star,\star]_G
		=
		2[\mu_0,\hbar m_\star] + [\hbar m_\star,\hbar m_\star].
	\end{equation}
	Thus star products correspond to Maurer-Cartan elements 
	$\hbar m_\star \in \hbar \Dpoly^{1}(M)[[\hbar]]$, and the equivalence of $\star$ 
	and $\star'$ is equivalent to 
	\begin{equation}
	  \exp(\hbar [T,\cdot]_G) \star
		=
		\star',
	\end{equation}
	where $\hbar T = \log S \in \hbar 
	\Dpoly^{(0)}(M)[[\hbar]]$, see \cite[Proposition~6.2.20]{waldmann:2007a}. 
	Consequently, the equivalence classes 
	of star products are given by 	
	\begin{equation}
	  \Def(\Dpoly(M)[[\hbar]])
		=
		\frac{\Mc^\hbar(\Dpoly(M)[[\hbar]])}{\group{G}^0(\Dpoly(M)[[\hbar]])},
	\end{equation}
	where the gauge group action of $\group{G}^0(\Dpoly(M)[[\hbar]])= 
	\{ \exp([\hbar T,\,\cdot\,]_G) \mid \hbar T\in \hbar \Dpoly^{0}(M)[[\hbar]] \}$ 
	is given by
	\begin{equation*}
	  \hbar m_{\star'}
		=
		\exp([\hbar T,\argument]_G) \acts \hbar m_\star 
		=
		\exp([\hbar T,\argument]_G)(\mu_0 +\hbar m_\star) - \mu_0,
	\end{equation*}
	and where we consider again 
	only formal Maurer-Cartan elements $\hbar m_\star \in \Mc^\hbar = \Mc \cap 
	\hbar \Dpoly^{1}(M)[[\hbar]]$, i.e. starting in order one of $\hbar$.
\end{example}

In the above examples we have seen that it is useful to identify 
'equivalent' Maurer-Cartan elements by actions of elements of degree $0$. 
In both settings the exponential maps were well-defined because of 
the complete filtration induced by the formal power series. 
Therefore, we restrict ourselves to (curved) 
Lie algebras with complete descending filtrations $\mathcal{F}^\bullet\liealg{g}$ satisfying
\begin{equation}
  \cdots 
	\supseteq 
	\mathcal{F}^{-2}\liealg{g}
	\supseteq 
	\mathcal{F}^{-1}\liealg{g}
	\supseteq 
	\mathcal{F}^{0}\liealg{g}
	\supseteq 
	\mathcal{F}^{1}\liealg{g}
	\supseteq 
	\cdots,
	\quad \quad
	\liealg{g}
	\cong
	\varprojlim \liealg{g}/\mathcal{F}^n\liealg{g},
\end{equation}  
and 
\begin{equation}
  \D(\mathcal{F}^k\liealg{g})
	\subseteq
	\mathcal{F}^k\liealg{g}
	\quad \quad \text{ and } \quad \quad
	[\mathcal{F}^k\liealg{g},\mathcal{F}^\ell\liealg{g}]
	\subseteq 
	\mathcal{F}^{k+\ell}\liealg{g}.
\end{equation}
In most cases the filtration will be bounded below, i.e. 
bounded from the left with $\liealg{g}=\mathcal{F}^k\liealg{g}$ for some 
$k\in \mathbb{Z}$, preferably $k=0$ or $k=1$. If the filtration is unbounded, 
then we assume in addition that it is exhaustive, i.e. that
\begin{equation}
  \liealg{g}
	=
	\bigcup_n \mathcal{F}^n\liealg{g},
\end{equation}
even if we do not mention it explicitly. Note that instead of considering 
filtered DGLAs one can tensorize the DGLAs with nilpotent algebras:

\begin{remark}[Nilpotent DGLA]
\label{rem:nilpotentdgla}
Alternatively to the filtration, one can tensorize the DGLA $(\liealg{g},\D,
[\argument{,}\argument])$ by a graded commutative associative 
$\mathbb{K}$-algebra $\liealg{m}$, compare \cite[Section~2.4]{esposito:2015a}: 
\begin{align*}
  (\liealg{g} \otimes \liealg{m})^n
	& = 
	\bigoplus_i (\liealg{g}^i \otimes \liealg{m}^{n-i})  \\
	\D(x \otimes m)
	& = 
	\D x \otimes m   \\
	[x \otimes m, y \otimes n]
	& = 
	(-1)^{\abs{m}\abs{y}} [x,y] \otimes mn.
\end{align*}
In particular, if $\liealg{m}$ is nilpotent, the DGLA $\liealg{g}\otimes 
\liealg{m}$ is nilpotent, too. Thus under the assumption that 
$\liealg{g}^0 \otimes \liealg{m}$ is nilpotent, the above exponential maps 
are also in the non-filtered setting well-defined. For further details on this approach and more details on the deformation functor 
see \cite{canonaco:1999a,manetti:2005a}. 
\end{remark}

\begin{example}[Formal power series]
  If we consider formal power series $\liealg{G}=\liealg{g}[[\hbar]]$ of a 
  DGLA $\liealg{g}$, where all maps are $\hbar$-linearly extended, 
  then we can choose the filtration 
  $\mathcal{F}^k \liealg{G} = \hbar^k \liealg{G}$ and the completeness 
  is trivially fulfilled.
\end{example}

As in the above examples we set for the curved Maurer-Cartan elements in 
$(\liealg{g},R,\D,[\argument{,}\argument])$ 
\begin{equation}
  \Mc^1(\liealg{g})
	=
	\{\pi \in \mathcal{F}^1\liealg{g}^1 \mid 	R +\D \pi + \frac{1}{2}[\pi,\pi]=0\}
\end{equation}
and we want to define a group action by the gauge group 
\begin{equation}
  \group{G}^0(\liealg{g})
	=
	\{ \Phi = e^{[a,\,\cdot\,]} \colon \liealg{g} 
	\longrightarrow \liealg{g} \mid a \in \mathcal{F}^1\liealg{g}^0\},
\end{equation}	
where we consider again only elements of filtration order $1$. 
In order to define the gauge action, we consider as in \cite{manetti:2005a} 
the DGLA $(\liealg{g}_{\D},[\argument{,}\argument]_{\D},0)$ with 
$\liealg{g}_{\D} = \liealg{g} \oplus \mathbb{K}\D$ and 
\begin{equation}
  \label{eq:gplusd}
  [x+ r\D, y+s\D]_{\D}
	=
	[x,y] + r \D(y) -(-1)^{\abs{x}}s\D(x) + 2rs R.
\end{equation}
Then $\pi$ is a Maurer-Cartan element in $\liealg{g}$ if and only if 
$\phi(\pi)= \pi + \D$ is a Maurer-Cartan element in $\liealg{g}_{\D}$ with zero 
differential. But in $(\liealg{g}_{\D},[\argument{,}\argument]_{\D},0)$ 
we already have an action of $\group{G}^0(\liealg{g})$ on the 
Maurer-Cartan elements by the adjoint representation. 
Pulling this action back on $\liealg{g}$ yields the following:

\begin{proposition}[Gauge action]
  \label{prop:GaugeactionDGLA}
  Let $(\liealg{g},R,\D,[\argument{,}\argument])$ be a curved Lie algebra with complete 
	descending filtration. The \emph{gauge group} $ \group{G}^0(\liealg{g})$
	acts on $\Mc^1(\liealg{g})$ via
	\begin{equation}
	  \label{eq:gaugeactiong0}
	  \exp([g,\,\cdot\,]) \acts \pi 
		=
		\sum_{n=0}^\infty \frac{( [g,\,\cdot\,])^n}{n!}(\pi) 
		-
		\sum_{n=0}^\infty \frac{([g,\,\cdot\,])^n}{(n+1)!} (\D g)
		=
		\pi + \sum_{n=0}^\infty \frac{([g,\,\cdot\,])^n}{(n+1)!} ([g,\pi]-\D g).
	\end{equation}
	The equivalence classes of Maurer-Cartan elements are denoted by 
	\begin{equation}
	  \Def(\liealg{g})
		=
		\frac{\Mc^1(\liealg{g})} {\group{G}^0(\liealg{g})}.
	\end{equation}
\end{proposition}
$\Def(\liealg{g})$ is the orbit space of the 
transformation groupoid $\group{G}^0(\liealg{g})\ltimes\Mc^1(\liealg{g})$ of the gauge action and 
$\group{G}^0(\liealg{g})\ltimes\Mc^1(\liealg{g})$ is also called \emph{Goldman-Millson groupoid} or 
\emph{Deligne groupoid} \cite{goldman.millson:1988a}. 
It plays an important role in deformation theory \cite{manetti:2005a}.

An additional motivation for this gauge action comes from 
the following general consideration: 
\begin{lemma}
  \label{lemma:TwistedDGLAsIsomorphic}
  Let $(\liealg{g},R,\D,[\argument{,}\argument])$ be a curved Lie algebra with complete 
	descending filtration. For all $g\in \mathcal{F}^1\liealg{g}^0$ 
	and all derivations $D$ of degree $+1$ one has 
	\begin{equation}
		\label{eq:ConjunctionofDer}
		\exp([g,\argument]) \circ D \circ \exp([-g,\argument])
		=
		D - \left[ \frac{\exp([g,\argument])-\id}{[g,\argument]} (Dg),
		\argument\right].
	\end{equation}
\end{lemma}
\begin{proof}
The proof is the same as \cite[Lemma~1.3.20]{neumaier:2001a}. It follows directly 
from $\Ad(\exp([g,\argument])) D = e^{\ad([g,\argument])} D$ and 
$\ad([g,\argument])D = - [Dg,\argument]$.
\end{proof}

This immediately implies that twisting with gauge equivalent Maurer-Cartan 
elements leads to isomorphic DGLAs.

\begin{corollary}
  \label{cor:GaugeEquivMCTwistsQuis}
  Let $(\liealg{g},R,\D,[\argument{,}\argument])$ be a curved 
	Lie algebra with complete 
	descending filtration and let $\pi, \widetilde{\pi}$ be two 
	gauge equivalent Maurer-Carten elements 
	via $  \widetilde{\pi}	=	\exp([g,\,\cdot\,]) \acts \pi $. Then one has
	\begin{equation}
		\label{eq:MCEquivalenceIso}
		\D + [\widetilde{\pi},\argument]
		=
		\exp([g,\,\cdot\,]) \circ (\D + [\pi,\argument]) \circ \exp([-g,\,\cdot\,]),
	\end{equation}
	i.e. $\exp([g,\argument]) \colon  (\liealg{g},\D+[\pi,\argument],
	[\argument{,}\argument]) \rightarrow (\liealg{g},\D +[\widetilde{\pi},\argument],
	[\argument{,}\argument])$ is an isomorphism of DGLAs.
	Moreover, one also has in the coalgebra setting 
	\begin{equation}
		\label{eq:MCEquivalenceCoalgIso}
		Q^{\widetilde{\pi}}
		=
		\exp([g,\,\cdot\,]) \circ Q^\pi \circ \exp([-g,\,\cdot\,]),
	\end{equation}
	where $Q^\pi,Q^{\widetilde{\pi}}$ are the codifferentials with 
	$Q^\pi_1 = - \D - [\pi,\argument]$ and $Q^{\widetilde{\pi}}_1 = -\D 
	- [\widetilde{\pi},\argument]$ and second structure map given by 
	the bracket.
\end{corollary} 
\begin{proof}
The first equation is clear by \eqref{eq:ConjunctionofDer}. The second one is 
clear since $\exp([g,\,\cdot\,])$ is a Lie algebra automorphism of degree zero 
intertwining the differentials.
\end{proof}

Finally, note that we recover indeed the equivalence notions for the polyvector 
fields and polydifferential 
operators.
\begin{example}[$\Tpoly(M)$ and $\Dpoly(M)$]
It is easy to see that the action of $\group{G}^0$ on $\Mc^\hbar$ as defined in 
\eqref{eq:gaugeactiong0} coincides with the actions in the above 
examples $\Tpoly(M)[[\hbar]]$ and $\Dpoly(M)[[\hbar]]$. 
For $\Tpoly(M)$ in Example~\ref{ex:tpoly} the differential of the DGLA is 
zero, i.e. the gauge action $\exp(\hbar[X,\,\cdot\,]) \acts \cdot$ 
coincides with the usual action by formal diffeomorphisms. In the case of $\Dpoly(M)$ in 
Example~\ref{ex:dpoly} the differential is given by $\del=[\mu_0,\,\cdot\,]$, 
where $\mu_0$ denotes the pointwise product on functions. Therefore, two formal 
star products $\star=\mu_0 + \hbar m_\star$ and $\star'=\mu_0 + \hbar m_{\star'}$ 
are equivalent via
$ \exp(\hbar [T,\,\cdot\,]) \star=	\star'$ with $\hbar T \in 
\hbar \Dpoly^{0}(M)[[\hbar]]$ if and only if
\begin{align*}
  \star' 
	=
	\mu_0 + \hbar m_{\star'}
  = 
	\exp(\hbar [T,\,\cdot\,]) (\mu_0 + \hbar m_\star)
	=
	\mu_0 + \exp(\hbar[T,\,\cdot\,]) \acts \hbar m_\star.
\end{align*}
\end{example}

%
%
\subsection{Maurer-Cartan Elements in $L_\infty$-algebras}
\label{sec:MCinLinftyAlgs}

The notion of Maurer-Cartan elements can be transferred to general 
$L_\infty$-algebras. There are again different possibilities for conditions 
on the $L_\infty$-algebra. 
One possibility is to require the $L_\infty$-algebra $(L,Q)$ to be 
\emph{nilpotent} \cite{getzler:2009a,manetti:note}: for example, one requires 
$L^{[n]}=0$ for $n \gg 0$, where
\begin{equation}
  L^{[n]}
	=
	\Span\{ Q_k(x_1\vee \cdots \vee x_k) \mid 
	k \geq 2, \; x_i \in L^{[n_i]}, \; 0 < n_i <n ,\; \sum_i n_i > n\}
\end{equation} 
and $L^{[1]}=L$. In particular, this implies $Q_n =0$ for $n \gg 0$. 
However, we consider again $L_\infty$-algebras with complete descending and 
exhaustive filtrations, where we implicitly include exhaustive when we say 
'complete filtration'. Moreover, we 
require from now on that the codifferentials and the $L_\infty$-morphisms 
are compatible with the filtrations. 

\begin{remark}
  There are again different conventions about the filtrations. Sometimes one 
	considers \emph{complete} $L_\infty$-algebras $L$, i.e. 
	$L_\infty$-algebras with complete descending filtrations where 
	$L = \mathcal{F}^1 L$, see e.g. \cite{dotsenko.poncin:2016a}.
\end{remark}

\begin{definition}[Maurer-Cartan elements II]
  Let $(L,Q)$ be a (curved) $L_\infty$-algebra 
	with complete descending filtration. Then $\pi\in \mathcal{F}^1 L[1]^0
	= \mathcal{F}^1 L^1$ is called 
	\emph{(curved) Maurer-Cartan element} if it satisfies the Maurer-Cartan equation
	\begin{equation}
	  \label{eq:mclinfty}
		Q^1(\exp(\pi))
		=
	  \sum_{n\geq 0} \frac{1}{n!} Q_n^1(\pi\vee\cdots\vee \pi)
		=
		0.
	\end{equation}
	The set of (curved) Maurer-Cartan elements is denoted by $\Mc^1(L)$.
\end{definition}
 
Note that the sum in \eqref{eq:mclinfty} is well-defined for $\pi\in 
\mathcal{F}^1L^1$ because of the completeness. From now on we
assume to be in this setting and we collect some useful properties:

\begin{lemma}
  \label{lemma:twistinglinftymorphisms}
  Let $F\colon (L,Q)\rightarrow (L',Q')$ be an 
	$L_\infty$-morphism of (curved) $L_\infty$-algebras and 
	$\pi \in \mathcal{F}^1L^1$. 
	\begin{lemmalist}
		\item $\pi$ is a (curved) Maurer-Cartan element if and only if 
		      $Q(\exp(\pi))=0$. 
		
		\item $F(\exp(\pi)) = \exp(S)$ with 
		      $S = F_\MC(\pi)= F^1(\cc{\exp}(\pi))$, where 
					$\cc{\exp}(\pi)=\sum_{k=1}^\infty \frac{1}{k!} \pi^{\vee k}$. 
					
		\item If $\pi$ is a (curved) Maurer-Cartan element, then so is 
		      $F_\MC(\pi)$.
  \end{lemmalist}
\end{lemma}
\begin{proof}
The proof for the case of flat DGLAs can be found in \cite[Proposition~1]{dolgushev:2005b}. 
Note that in the flat case it suffices to consider 
$\cc{\exp}(\pi)=\sum_{k=1}^\infty \frac{1}{k!} \pi^{\vee k}$ instead of $\exp(\pi)$.
In our general setting we have
\begin{equation*}
  \Delta_\sh(\exp(\pi))
	=
	\exp(\pi) \otimes \exp(\pi),
\end{equation*}
where we have to consider the completed symmetric tensor algebra with respect 
to the symmetric degrees. For $\pi \in \mathcal{F}^1 L[1]^0$ we compute with \eqref{eq:QinSymmetricviaTaylor}
\begin{align*}
  Q(\exp(\pi))
	& = 
	\sum_{n=0}^\infty \frac{1}{n!} \sum_{k=0}^n 
	\binom{n}{k} Q_k^1(\pi\vee \cdots \vee \pi)\vee \pi \vee \cdots \vee \pi  \\
	& = 
	Q^1(\exp(\pi))\vee \exp (\pi)
\end{align*}
and thus the first point is clear. 
The second point follows from the explicit form of $F$ from 
Proposition~\ref{prop:linftymorphismsequence}.
Concerning the last part we have
\begin{equation*}
  Q'(\exp(S))
	=
	Q'F(\exp(\pi))
	=
	F Q(\exp(\pi))
	=
	0
\end{equation*}
whence the statement follows with the first and second point.
\end{proof}

We have shown that $L_\infty$-morphisms map Maurer-Cartan elements to Maurer-Cartan 
elements, but we would like to have even more: they should map equivalence 
classes of Maurer-Cartan elements to equivalence classes and thus yield maps 
on the deformations $\Def(\liealg{g})$ of DGLAs resp. of $L_\infty$-algebras. 
Therefore, we have to generalize at first the gauge action to an equivalence 
relation on the set of Maurer-Cartan elements of (curved) $L_\infty$-algebras. 
This is achieved by introducing the homotopy equivalence relation, see e.g.
\cite[Lemma~5.1]{manetti:2005a} for the case of DGLAs and 
\cite[Definition~6.5]{maunder:2017a} for the case of curved Lie algebras.
For the $L_\infty$-setting we follow \cite[Section~4]{canonaco:1999a} but 
adapt the definitions to the case of $L_\infty$-algebras with complete 
descending filtrations as in \cite{dotsenko.poncin:2016a}. 

Let therefore $(L,Q)$ be an $L_\infty$-algebra with complete descending filtration and 
consider $L[t]=L \otimes \mathbb{K}[t]$ which has again a descending filtration
\begin{equation*}
  \mathcal{F}^k L[t]
	=
	\mathcal{F}^kL \otimes \mathbb{K}[t].
\end{equation*} 
We denote its completion by $\widehat{L[t]}$ and note that since $Q$ is compatible 
with the filtration it extends to $\widehat{L[t]}$. Similarly, 
$L_\infty$-morphisms extend to these completed spaces. 

\begin{remark} 
  \label{rm:completion}
  Note that one can define the completion as space of equivalence classes of 
	Cauchy sequences with respect to the filtration topology.
  Alternatively, the completion can be identified with
  \begin{equation*}
    \varprojlim L[t] / \mathcal{F}^n L[t] 
	  \subset \prod_n L[t]/\mathcal{F}^nL[t]
		\cong
		\prod_n L/\mathcal{F}^nL \otimes\mathbb{K}[t]
  \end{equation*}
  consisting of all coherent tuples $X=(x_n)_n \in 
	\prod_n L[t]/\mathcal{F}^nL[t] $, where
  \begin{equation*}
    L[t]/\mathcal{F}^{n+1}L[t] \ni x_{n+1}
	  \longmapsto
	  x_n \in L[t]/\mathcal{F}^n[t]
  \end{equation*}
	under the obvious surjections. 
	Moreover, $\mathcal{F}^n\widehat{L[t]}$ corresponds to the kernel 
	of $\varprojlim L[t]/\mathcal{F}^nL[t] \rightarrow L[t]/\mathcal{F}^nL[t]$ 
	and thus
	\begin{equation*}
	  \widehat{L[t]} / \mathcal{F}^n\widehat{L[t]} 
		\cong 
    L[t] / \mathcal{F}^n L[t].
	\end{equation*}
	Since $L$ is complete, we can also interpret $\widehat{L[t]}$ as the 
	subspace of $L[[t]]$ such that 
	$X \;\,\mathrm{mod}\;\, \mathcal{F}^nL[[t]]$ is polynomial in $t$. In particular, 
	$\mathcal{F}^n\widehat{L[t]}$ is the subspace of elements in 
	$\mathcal{F}^nL[[t]]$ that are polynomial in $t$ modulo 
	$\mathcal{F}^mL[[t]]$ for all $m>n$.
\end{remark}

By the above construction of $\widehat{L[t]}$ it is clear that 
differentiation $\frac{\D}{\D t}$ and integration with respect to 
$t$ extend to it since they do not change the filtration. Moreover, 
\begin{equation*}
  \delta_s 
	\colon
	\widehat{L[t]} \ni X(t)
	\longmapsto
	X(s) \in L
\end{equation*}
is well-defined for all $s\in \mathbb{K}$ since $L$ is complete.

\begin{example}
  In the case that the filtration of $L$ comes from a grading $L^\bullet$, the 
	completion is given by $\widehat{L[t]}\cong \prod_i L^i[t]$, i.e. by polynomials 
	in each degree. A special case is here the case of formal power series 
	$L= V[[\hbar]]$ with $\widehat{L[t]} \cong (V[t])[[\hbar]]$ as in 
	\cite[Appendix~A]{bursztyn.dolgushev.waldmann:2012a}.
\end{example}

Now we can introduce a general equivalence relation between Maurer-Cartan 
elements of $L_\infty$-algebras. We write $\pi_0 \sim \pi_1$ if 
there exist $\pi(t) \in \mathcal{F}^1\widehat{L^1[t]}$ and 
$\lambda(t) \in \mathcal{F}^1\widehat{L^0[t]}$ such that
	\begin{align}
	  \label{eq:EquivMCElements}
		\begin{split}
		  \frac{\D}{\D t} \pi(t)
			& = 
			Q^1  (\lambda(t) \vee \exp(\pi(t)))
			=
			\sum_{n=0}^\infty \frac{1}{n!} Q^1_{n+1} (\lambda(t) \vee \pi(t) \vee 
			\cdots \vee \pi(t)),   \\
			\pi(0) 
			& = 
			\pi_0 
			\quad \quad \text{ and }\quad \quad 
			\pi(1)
			=
			\pi_1.
		\end{split}
\end{align}
We directly see that $\sim$ is reflexive and symmetric and one can check that 
it is also transitive. We write $[\pi_0]$ for the homotopy class of 
$\pi_0$ and define:

\begin{definition}[Homotopy equivalence]
\label{def:HomEquivalenceofMC}
  Let $(L,Q)$ be a (curved) $L_\infty$-algebra	with complete descending filtration. 
	The \emph{homotopy equivalence relation} on the set $\Mc^1(L)$ is given by 
	the relation $\sim$ from \eqref{eq:EquivMCElements}. The set of equivalence classes of Maurer-Cartan elements is denoted 
	by $\Def(L) = \Mc^1(L) / \sim$.
\end{definition} 

\begin{remark}
This definition can be reformulated: two Maurer-Cartan elements $\pi_0$ and $\pi_1$ in $L$ 
are homotopy equivalent if and only if there exists a Maurer-Cartan element 
$\pi(t)-\lambda(t)\D t$ in $\widehat{L[t,\D t]} $ with $\pi(0)=\pi_0$ 
and $\pi(1)=\pi_1$, see e.g. 
\cite{dotsenko.poncin:2016a} for $L_\infty$-algebras and \cite{manetti:2005a} for DGLAs.
\end{remark}

Note that in the case of 
nilpotent $L_\infty$-algebras it suffices to consider polynomials in $t$ 
as there is no need to complete $L[t]$, compare \cite{getzler:2009a}.
We check now that this is well-defined and even yields a curve $\pi(t)$ 
of Maurer-Cartan elements, see \cite[Proposition~4.8]{canonaco:1999a}.

\begin{proposition}
  \label{prop:PioftUnique}
  For every $\pi_0 \in \mathcal{F}^1L^1$ and $\lambda(t) \in 
	\mathcal{F}^1\widehat{L^0[t]}$ 
	there exists a unique $\pi(t) \in \mathcal{F}^1\widehat{L^1[t]}$ 
	such that 
	$\frac{\D}{\D t} \pi(t) = Q^1  (\lambda(t) \vee \exp(\pi(t)))$ and 
	$\pi(0) = \pi_0$. If $\pi_0 \in \Mc^1(L)$, then $\pi(t) \in \Mc^1(L)$ for 
	all $t\in \mathbb{K}$.
\end{proposition}
\begin{proof}
At first we show that there exists a unique solution $\pi(t) = \sum_{k=0}^\infty 
\pi_k t^k$ in the formal power series $\mathcal{F}^1 L^1 \otimes \mathbb{K}[[t]]$. 
On one hand one has
\begin{equation*}
  \frac{\D}{\D t} \pi(t)
	=
	\sum_{k =0}^\infty (k+1) \pi_{k+1} t^k,
\end{equation*}
on the other hand there exist $\phi_k \in \mathcal{F}^1L^1$ such that
\begin{equation*}
  Q^1  (\lambda(t) \vee \exp(\pi(t)))
	=
	\sum_{k=0}^\infty \phi_k t^k.
\end{equation*}
Here the $\phi_k$ depend only on the $\lambda_j$ of $\lambda(t)= \sum_{j=0}^\infty 
\lambda_j t^j$ and the $\pi_i$ for $i\leq k$, hence they can be defined 
inductively. It remains to check that one has even 
$\pi(t)\in \mathcal{F}^1\widehat{L^1[t]}$, i.e. by 
Remark~\ref{rm:completion} that
$\pi(t)\;\,\mathrm{mod}\;\, \mathcal{F}^nL^1[[t]] \in L^1[t]$ for all $n$. 
Indeed, we have inductively
\begin{equation*}
  \frac{\D}{\D t}\pi(t) \;\,\mathrm{mod}\;\, \mathcal{F}^2L^1[[t]]
	=
	Q^1(\lambda(t)) \;\,\mathrm{mod}\;\,\mathcal{F}^2L^1[[t]]
	\in 
	L^1[t].
\end{equation*}
For the higher orders we get
\begin{equation*}
  \frac{\D}{\D t}\pi(t) 
	\equiv
	\sum_{k=0}^{n-2}\frac{1}{k!} 
	Q^1_{k+1}(\lambda(t)\vee (\pi(t)\;\,\mathrm{mod}\;\,\mathcal{F}^{n-1})\vee \cdots 
	\vee (\pi(t)\;\,\mathrm{mod}\;\,\mathcal{F}^{n-1}))
	\;\,\mathrm{mod}\;\, \mathcal{F}^nL^1[[t]]
\end{equation*}
and thus $\pi(t) \,\mathrm{mod}\, \mathcal{F}^nL^1[[t]] \in L^1[t]$.

Let now $\pi_0$ be a curved Maurer-Cartan element, the flat case follows directly 
from the curved one. We have to show $g(t) = Q^1(\exp(\pi(t))) = 0$ for 
all $t$, so it suffices to show $g^{(n)}(0)= \frac{\D^n}{\D t^n} g (0)=0$ for all $n\geq 0$. 
The case $n=0$ is clear, for $n=1$ we get 
\begin{align*}
  g^{(1)}(t)
	& =
	Q^1 \left(\exp(\pi(t)) \vee Q^1(\lambda(t)\vee \exp(\pi(t)))\right) \\
	& =
	Q^1\left(Q\left(\lambda(t)\vee \exp(\pi(t))\right) + \lambda(t)\vee \exp(\pi(t)) \vee 
	Q^1(\exp(\pi(t)))\right) \\
	& = 
	Q^1\left(\lambda(t)\vee \exp(\pi(t)) \vee g^{(0)}(t)\right).
\end{align*} 
The statement follows by induction.
\end{proof}

In the case of a curved Lie algebra $(\liealg{g},R,\D,[\argument{,}\argument])$ this 
recovers the gauge action from Proposition~\ref{prop:GaugeactionDGLA}. 
Going from gauge equivalence to homotopy equivalence is easy:
Explicitly, let $\pi_1 = \exp([g,\argument])\acts \pi_0$, then setting 
	$\lambda(t) = g$ and $\pi(t) = \exp([tg,\argument])\acts \pi$ satisfies
	\begin{align*}
	  \frac{\D}{\D t} \pi(t)
	  & =
	  \exp([tg,\argument])[g,\pi_0] - \exp([tg,\argument])(\D g) \\
		& =
		-\D g + \exp([tg,\argument])[g,\pi_0] -
		\sum_{n=0}^\infty \frac{([tg,\,\cdot\,])^{n+1}}{(n+1)!} (\D g) \\
		& =
		Q_1^1(\lambda(t)) + [\lambda(t), \exp([tg,\argument])\acts \pi_0].
	\end{align*} 
For the flat setting, the other direction from the homotopy equivalence to 
the gauge equivalence is contained in the following theorem, see e.g. 
\cite[Theorem~5.5]{manetti:2005a}.

\begin{theorem}
  \label{thm:DGALHomvsGaugeEquiv}
  Two Maurer-Cartan elements in $(\liealg{g},\D,[\argument{,}\argument])$ 
	are homotopy equivalent if and only if they are gauge equivalent.
\end{theorem}

This theorem can be rephrased in a more explicit manner in the following proposition, 
see \cite[Proposition~2.13]{kraft.schnitzer:2021a:pre}.

\begin{proposition}
\label{prop:HomEquvsGaugeEqu}
  Let $(\liealg{g},R,\D,[\argument{,}\argument])$ be a curved Lie algebra 
	equipped with a complete descending filtration. 
	Consider $\pi_0 \sim \pi_1$ with homotopy equivalence given by  
  $\pi(t) \in \mathcal{F}^1\widehat{\liealg{g}^1[t]}$ and 
	$\lambda(t) \in \mathcal{F}^1\widehat{ \liealg{g}^0[t]}$. 
	The formal solution of
	\begin{equation}
		\label{eq:ODEforA}
		\lambda(t)
		=
		\frac{\exp([A(t),\argument])-\id}{[A(t),\argument]} \left(\frac{\D}{\D t}A(t)\right),
		\quad\quad
		A(0)
		=
		0
	\end{equation}
	is an element $A(t) \in \mathcal{F}^1\widehat{ \liealg{g}^0[t]}$
	and satisfies
	\begin{equation}
	  \pi(t)
		=
		e^{[A(t),\argument]}\pi_0 
	  - \frac{\exp([A(t),\argument]-\id}{[A(t),\argument]} \D A(t).
	\end{equation}
	In particular, one has for $g =A(1) \in  \mathcal{F}^1 \liealg{g}^0$ 
	\begin{equation}
	  \pi_1
		=
		\exp([g,\argument])\acts \pi_0.
	\end{equation}
\end{proposition}
\begin{proof}
As formal power series in $t$ Equation~\eqref{eq:ODEforA} has a unique solution
$A(t) \in \mathcal{F}^1 \liealg{g}^0 \otimes \mathbb{K}[[t]]$.
But one has even $A(t)\in \mathcal{F}^1\widehat{\liealg{g}^0[t]}$ 
since
\begin{align*}
  \frac{\D A(t)}{\D t}  &
	\equiv
	\lambda(t) - \sum_{k=1}^{n-2} \frac{1}{(k+1)!} [A(t),\argument]^k 
	\frac{\D A(t)}{\D t} 
	\;\,\mathrm{mod}\;\, \mathcal{F}^n\liealg{g}[[t]] \\
	& \equiv
	\lambda(t) - \sum_{k=1}^{n-2} \frac{1}{(k+1)!} 
	[A(t)\;\,\mathrm{mod}\;\, \mathcal{F}^{n-1}\liealg{g}[[t]],\argument]^k 
	\left(\frac{\D A(t)}{\D t} \;\,\mathrm{mod}\;\, \mathcal{F}^{n-1}\liealg{g}[[t]]\right)
	\;\,\mathrm{mod}\;\, \mathcal{F}^n\liealg{g}[[t]] 
\end{align*}
is polynomial in $t$ by induction. Writing $\dot{A}(t)=\frac{\D}{\D t}A(t)$ 
one has
\begin{equation}
  \tag{$*$}
  \label{eq:DiffofExp}
  \frac{\D}{\D t} e^{[A(t),\argument]}
	=
	\left[\frac{\exp([A(t),\argument]-\id}{[A(t),\argument]}\dot{A}(t) , 
	\argument \right] 
	\circ
	\exp([A(t),\argument]).
\end{equation}
Our aim is now to show that 
$\pi'(t) = e^{[A(t),\argument]}\pi_0 
	- \frac{\exp([A(t),\argument]-\id}{[A(t),\argument]} \D A(t)$ 
satisfies
\begin{equation*}
  \frac{\D \pi'(t)}{\D t}
  =
  -\D \lambda(t) + \left[\lambda(t), e^{[A(t),\argument]}\pi_0 
	- \frac{\exp([A(t),\argument])-\id}{[A(t),\argument]} \D A(t)\right].
\end{equation*} 
Then we know $\pi'(t) = \pi(t)\in \mathcal{F}^1\widehat{\liealg{g}^1[t]}$ 
since the solution $\pi(t)$ is unique by 
Proposition~\ref{prop:PioftUnique}, which immediately gives 
$\pi'(1)=\pi_1$. At first compute
\begin{align*}
  \D \lambda(t)
	& =
	\frac{\exp([A(t),\argument])-\id}{[A(t),\argument]}\D \dot{A}(t) 
	+
	\sum_{k=0}^\infty \sum_{j=0}^{k-1} \frac{1}{(k+1)!} \binom{k}{j+1} 
	[ \ad_A^j\D A(t), \ad_A^{k-1-j}  \dot{A}(t)]
\end{align*} 
and with \eqref{eq:DiffofExp} we get
\begin{align*}
  \frac{\D \pi'(t)}{\D t}
	& =
	\left[\frac{\exp([A(t),\argument])-\id}{[A(t),\argument]} \dot{A}(t), 
	\exp([A(t),\argument]) \pi_0\right] - 
	\frac{\exp([A(t),\argument])-\id}{[A(t),\argument]}\D \dot{A}(t)  \\
	& \;
	- \sum_{k=0}^\infty \sum_{j=0}^{k-1} \frac{1}{(k+1)!}\binom{k}{j+1} 
	[\ad_A^j \dot{A}(t), \ad_A^{k-1-j}\D A]  \\
	& =
	-\D \lambda(t) + \left[\lambda(t), e^{[A(t),\argument]}\pi_0 
	- \frac{\exp([A(t),\argument])-\id}{[A(t),\argument]} \D A(t)\right]
\end{align*}
and therefore $\pi'(t)=\pi(t)$ and the proposition is proven.
\end{proof}

\begin{remark}
\label{rem:QuillenandGaugeEquivalence}
  More generally, there are also different notions of homotopy resp. gauge 
	equivalences for Maurer-Cartan elements in $L_\infty$-algebras: 
	for example the above definition, sometimes also called \emph{Quillen homotopy}, 
	the \emph{gauge homotopy} where one requires $\lambda(t) = \lambda$ to be 
	constant, compare \cite{dolgushev:2007a}, and the \emph{cylinder homotopy}. 
	In \cite{dotsenko.poncin:2016a} 
	it is shown that these notions are also equivalent for flat 
	$L_\infty$-algebras with complete 
	descending filtration and compatible higher brackets, extending the 
	result for DGLAs from \cite{manetti:2005a}.
\end{remark}

Now we can finally show that $L_\infty$-morphisms map equivalence classes of 
Maurer-Cartan elements to equivalence classes.

\begin{proposition}
\label{prop:FmapsEquivMCtoEquiv}
  Let $F \colon (L,Q) \rightarrow (L',Q')$ be an $L_\infty$-morphism between 
	(curved) $L_\infty$-algebras, 
	and $\pi_0,\pi_1 \in \Mc^1(L)$ with $[\pi_0]=[\pi_1]$. Then $F$ is compatible 
	with the homotopy equivalence relation, i.e. one has 
	$[F^1(\exp\pi_0)] = [F^1(\exp(\pi_1)]$. In particular, one has an induced 
	map $F_\MC \colon \Def(L) \rightarrow \Def(L')$.
\end{proposition}
\begin{proof} 
Let $\pi(t)$ and $\lambda(t)$ encode the equivalence between $\pi_0$ and 
$\pi_1$. We set $\widetilde{\pi}(t) = F^1(\exp(\pi(t)))$ and 
$\widetilde{\lambda}(t) = F^1(\lambda(t)\vee \exp(\pi(t)))$. We compute 
\begin{align*}
  \frac{\D}{\D t} \widetilde{\pi}(t)
	& = 
	F^1\left( \exp(\pi(t))\vee Q^1\left(\lambda(t) \vee \exp(\pi(t))\right)\right) \\
	& =
	F^1\left(Q(\lambda(t)\vee \exp(\pi(t))) + \lambda(t)\vee \exp(\pi(t)) \vee 
	Q^1(\exp(\pi(t)))\right)   \\
	& =
	Q'^1 \circ F \left(\lambda(t)\vee \exp(\pi(t))\right) \\
	& = 
	Q'^1\left(F^1\left(\lambda(t)\vee \exp(\pi(t))\right) \vee \exp (F^1(\exp(\pi(t))))\right) \\
	& = 
	Q'^1\left(\widetilde{\lambda}(t) \vee \exp(\widetilde{\pi}(t))\right),
	\end{align*}
and thus the desired $[F^1(\exp\pi_0)] = [F^1(\exp\pi_1)]$.
\end{proof}

If one does not want to restrict to $L_\infty$-algebras with complete filtrations, one can tensorize general $L_\infty$-algebras by nilpotent algebras, compare 
Remark~\ref{rem:nilpotentdgla} for the case of DGLAs. In this setting the 
deformations are not a set but a functor: 
For a (curved) $L_\infty$-algebra $L$ the deformation 
functor $\Def_L \colon \mathrm{Art}_\mathbb{K} \rightarrow \mathrm{Set}$ 
maps a local Artinian ring $A$ to the set 
\begin{equation*}
  \Def_L(A) 
	=
	\Def(L \otimes m_A).
\end{equation*}
Here $m_A$ is the maximal ideal of $A$ and thus $L\otimes m_A$ is a 
nilpotent $L_\infty$-algebra and the above is well-defined. 
In this case it can be shown that one 
even has the following statement, see \cite[Theorem~4.6]{kontsevich:2003a} and 
\cite[Theorem~4.12]{canonaco:1999a} and also \cite[Proposition~4]{dolgushev:2005a} for 
the filtered setting.

\begin{theorem}
  \label{thm:lquisbijondef}
  Let $F \colon (L,Q)\rightarrow (L',Q')$ be an $L_\infty$-quasi-isomorphism 
	of flat $L_\infty$-algebras. Then the map
	\begin{equation}
	  \label{eq:mclinftymorphism}
		F_\MC \colon
	  \pi 
		\longmapsto
		F_\MC(\pi) = \sum_{n>0}\frac{1}{n!} F_n(\pi\vee \cdots \vee\pi)
	\end{equation}
	induces an isomorphism between the deformation functors 
	$\Def_L$ and $\Def_{L'}$.
\end{theorem}
\begin{proof}[Sketch]
In Proposition~\ref{prop:StandardFormLinfty} we have shown that 
every $L_\infty$-algebra $L$ is isomorphic to the direct product of a minimal one 
$L_{min}$, i.e. one with 
$Q_1=0$, and a linear contractible one $L_{lc}$, see also 
\cite[Proposition~2.8]{canonaco:1999a}. 
On linear contractible $L_\infty$-algebras
the deformation functor is trivial, i.e. for all $A$ the set $\Def_{L_{lc}}(A)$ 
contains just one element: A Maurer-Cartan element $\pi$ is just a closed 
element. But since the cohomology is trivial, it is exact and there exists 
$\lambda \in L^0_{lc} \otimes m_A$ with $Q_1^1(\lambda) = \pi$, and $\pi(t) 
= t \pi$ shows that $\pi$ is homotopy equivalent to zero.
In addition, one easily sees that the deformation functor is compatible 
with direct products, i.e. 
\begin{equation*}
  \Def_{L\oplus L'} 
	\cong
	\Def_L \times \Def_{L'}.
\end{equation*}
Summarizing, this yields
\begin{equation*}
  \Def_L
	\cong
	\Def_{L_{min}}
	\cong
	\Def_{L'_{min}}
	\cong
	\Def_{L'}
\end{equation*}
since $L_{min}$ and $L_{min}'$ are $L_\infty$-isomorphic, see also 
\cite[Section~4]{canonaco:1999a}. 
\end{proof} 

\begin{remark}
  The above result can be further generalized: In fact, a morphism 
	$L \colon \liealg{g} \rightarrow \liealg{g}'$ of DGLAs induces an isomorphism 
	on the deformation functors if the induced map in cohomology is bijective in 
	degree one, injective in degree two and surjective in degree zero, compare 
	\cite[Theorem~3.1]{manetti:2005a}.
\end{remark}

We are mainly interested in the setting of formal power series, where we directly get 
the following statement.

\begin{corollary}
\label{cor:BijFormalMC}
  Let $F$ be an $L_\infty$-quasi-isomorphism between two flat DGLAs $\liealg{g}$ and 
	$\liealg{g}'$. Then it induces a bijection $F_\MC$ between 
	$\Def(\liealg{g}[[\hbar]])$ and $\Def(\liealg{g}'[[\hbar]])$. 
\end{corollary}
\begin{proof}
The above statement follows from Theorem~\ref{thm:lquisbijondef} 
since $\mathbb{K}[[\hbar]] = 
\lim \mathbb{K}[\hbar] / \hbar^k \mathbb{K}[\hbar]$ is pro-Artinian with the 
pronilpotent $\hbar\mathbb{K}[[\hbar]]$ as maximal ideal. 
\end{proof}

In the curved setting the situation is more complicated  
since the proof of Theorem~\ref{thm:lquisbijondef} does not generalize: 
if the curvature is not central one does 
not even have a differential and there is no obvious notion of 
$L_\infty$-quasi-isomorphisms between curved $L_\infty$-algebras. Therefore, we postpone these considerations 
until we understand the twisting procedure, which is a way to obtain a flat 
$L_\infty$-algebra out of a curved one, see Lemma~\ref{lemma:CorrespCurvedandFlatMC} 
below.

\subsection{Twisting of (Curved) $L_\infty$-Algebras}
\label{subsec:Twisting}

Recall that for a Maurer-Cartan 
element $\pi$ of a DGLA $(\liealg{g},\D,[\argument{,}\argument])$ the map 
$\D + [\pi,\,\cdot\,]$ is a differential on $\liealg{g}$, 
the \emph{twisted} differential by $\pi$. This can be generalized 
to $L_\infty$-algebras, see e.g. \cite{dolgushev:2005a,
dolgushev:2005b,dotsenko.shadrin.vallette:2018,esposito.dekleijn:2021a}.

\begin{lemma} 
\label{lemma:twistCodiff}
Let $(L,Q)$ be a (curved) $L_\infty$-algebra and $\pi \in \mathcal{F}^1 L[1]^0$. 
Then the map $Q^\pi$ given by
\begin{equation}
  Q^\pi(X)
	=
	\exp((-\pi)\vee) Q (\exp(\pi \vee)X),
	\quad \quad
	X \in \Sym(L[1]),
\end{equation}
defines a codifferential on $\Sym (L[1])$. If $\pi$ is in addition a (curved) 
Maurer-Cartan element, then $(L,Q^\pi)$ is a flat $L_\infty$-algebra.
\end{lemma}
\begin{proof}
At first we have to show that $Q^\pi$ is a well-defined map into $\Sym(L[1])$ 
and not into its completion with respect to the symmetric degree. This is clear 
since its structure maps $(Q^\pi)^1_n$ are well-defined maps into $L[1]$ by the completeness of the filtration and since $Q^\pi$ defines a coderivation on 
$\Sym (L[1])$ by
\begin{align*}
  \Delta_\sh Q^\pi (X)
	& = 
	\Delta_\sh \exp(-\pi \vee)Q (\exp(\pi\vee)X) \\
	& =	
	\exp(-\pi \vee) \otimes \exp(-\pi \vee)
	(Q \otimes \id + \id \otimes Q) 
	(\exp(\pi\vee) \otimes \exp(\pi\vee))(\Delta_\sh X) \\
	& = 
	(Q^\pi \otimes \id + \id \otimes Q^\pi)(\Delta_\sh X).
\end{align*}
The property $(Q^\pi)^2=0$ is clear. 
If $\pi$ is in addition a (curved) Maurer-Cartan element, then one obtains 
$Q^\pi(1) =\exp(-\pi \vee) Q(\exp(\pi)) = 0$ and thus after twisting a flat 
$L_\infty$-algebra.
\end{proof}
The $L_\infty$-algebra $(L,Q^\pi)$ is again called \emph{twisted} 
and it turns out that one can also twist the $L_\infty$-morphism, 
see \cite[Proposition~1]{dolgushev:2005b} for the flat setting and 
\cite[Lemma~2.7]{esposito.dekleijn:2021a} for the curved setting.

\begin{proposition}
  \label{prop:twistinglinftymorphisms}
  Let $F\colon (L,Q)\rightarrow (L',Q')$ be an 
	$L_\infty$-morphism between (curved) $L_\infty$-algebras, 
	$\pi \in \mathcal{F}^1L^1$ and $S = F^1(\cc{\exp}\pi) 
	\in \mathcal{F}^1 (L')^1$. 
	\begin{propositionlist}					
		\item The map 
		      \begin{equation*}
					  F^\pi
						=
						\exp(-S\vee) F \exp(\pi\vee) \colon 
						\Sym(L[1])
						\longrightarrow
						\Sym(L'[1])
					\end{equation*}
					defines an $L_\infty$-morphism between the (curved) 
					$L_\infty$-algebras $(L,Q^\pi)$ and $(L',(Q')^S)$.
					
		\item The structure maps of $F^\pi$ are given by 
		      \begin{equation}
					  \label{eq:twisteslinftymorphism}
					  F_n^\pi(x_1,\dots, x_n)
						=
						\sum_{k=0}^\infty \frac{1}{k!} 
						F_{n+k}(\pi, \dots, \pi,x_1 ,	\dots, x_n)
					\end{equation}
					and $F^\pi$ is called \emph{twisted by $\pi$}.
		\item If $\pi$ is a (curved) Maurer-Cartan element, then $F^\pi$ is an 
		      $L_\infty$-morphism between flat $L_\infty$-algebras.
		\item Let $F$ be an $L_\infty$-quasi-isomorphism between flat $L_\infty$-algebras 
		      such that $F_1^1$ is not only a quasi-isomorphism of filtered complexes 
					$L\rightarrow L'$ but even induces a quasi-isomorphism
          \begin{equation*}
            F_1^1 \colon
	          \mathcal{F}^k L
	          \longrightarrow
	          \mathcal{F}^kL'
          \end{equation*}
          for each $k$. If $\pi$ is a flat Maurer-Cartan element, then $F^\pi$ is 
					also an $L_\infty$-quasi-isomorphism.	 
	\end{propositionlist}
\end{proposition}
\begin{proof}
$F^\pi$ is well-defined since its structure maps are well-defined maps into 
$L'[1]$ by the completeness of the filtration and since 
$F^\pi$ is a coalgebra morphism 
\begin{align*}
	\Delta_\sh \exp(-S\vee)F\exp(\pi \vee)X 
	& = 
	\exp(-S\vee)\otimes \exp(-S\vee)(F\otimes F)\Delta_\sh\exp(\pi \vee)X \\
	& = 
	F^\pi \otimes F^\pi \Delta_\sh (X).
\end{align*}
The compatibility of $F^\pi$ with the coderivations 
$Q^\pi$ and $(Q')^S$ follows directly with the definitions.
The third point follows directly from Lemma~\ref{lemma:twistCodiff}.

The last claim follows by a standard argument of spectral sequences. 
Since $F_1$ is a quasi-isomorphism w.r.t. 
$Q^1_1$ and $(Q')^1_1$ that is compatible with the filtrations, the map 
$F_1^\pi$ induces a quasi-isomorphism on the zeroth level of 
the corresponding spectral sequence, and therefore also on the 
terminal $E_\infty$-level, compare \cite[Proposition~1]{dolgushev:2005b}.
\end{proof}

It follows directly that the twisting 
procedure is functorial in the sense that 
\begin{equation}
  (G \circ F)^\pi 
	=
	G^S \circ F^\pi
\end{equation}
for an $L_\infty$-morphism $G\colon L' \rightarrow L''$, 
see \cite[Proposition~4]{dolgushev:2006a}, as well as
\begin{equation*}
  (Q^\pi)^B
	=
	Q^{\pi + B},
	\quad \quad 
	(F^\pi)^B
	=
	F^{\pi + B}.
\end{equation*}

Now we can come back to the correspondences of curved Maurer-Cartan elements 
under $L_\infty$-morphisms. Let $(L,Q)$ be a curved $L_\infty$-algebra with 
curved Maurer-Cartan element $m \in \mathcal{F}^1L^1$. Then we know 
from Lemma~\ref{lemma:twistCodiff} that the twisted codifferential 
$Q^{m}$ is flat and we get the following, compare \cite[Proposition~4.6]{maunder:2017a} 
for the case of curved Lie algebras.

\begin{lemma}
\label{lemma:CorrespCurvedandFlatMC}
  Let $(L,Q)$ be a curved $L_\infty$-algebra with 
  curved Maurer-Cartan element $m \in \mathcal{F}^1L^1$. 
	Then the curved Maurer-Cartan elements in $(L,Q)$ are in one-to-one 
	correspondence with flat Maurer-Cartan elements in $(L,Q^m)$ via 
	$\pi \mapsto \pi-m$. The 
	correspondence is compatible with equivalences.
\end{lemma}
\begin{proof}
We have for $\pi \in \mathcal{F}^1L[1]^0$
\begin{equation*}
  Q^1(\exp \pi)
	=
	Q^1(\exp(\pi) \vee \exp (m )\vee \exp(-m))
	=
	(Q^m)^1(\exp(\pi-m)),
\end{equation*}
so the first part follows.
Suppose that $\pi_0$ and $\pi_1$ are equivalent, i.e. there 
exists $\pi(t)$ with $\pi(0)=\pi_0$, $\pi(1)=\pi_1$ and 
\begin{equation*}
  \frac{\D}{\D t} \pi(t)
	= 
	Q^{1}  (\lambda(t) \vee \exp(\pi(t))).
\end{equation*}
Then $\pi'(t) = \pi(t) - m$ induces the 
equivalence between $\pi_0 - m$ and $\pi_1 - m$ in the 
flat setting since 
\begin{align*}
  \frac{\D}{\D t} \pi'(t)
	=
	\frac{\D}{\D t} \pi(t)
	= 
	Q^1  (\exp(m)\vee\lambda(t) \vee \exp(\pi(t)-m)) 
	=
	Q^{m,1}  (\lambda(t) \vee \exp(\pi'(t))),
\end{align*}
and the statement is shown.
\end{proof}

This directly implies for the equivalence classes of curved 
Maurer-Cartan elements:

\begin{corollary}
\label{cor:CurvdMCEquivBijection}
Let $F \colon (L,Q) \rightarrow (L',Q')$ be an 
$L_\infty$-morphism between two curved $L_\infty$-algebras 
with complete filtrations. 
If $L$ has a curved Maurer-Cartan 
element $m \in \mathcal{F}^1L^1$ such that $F^{m}$ induces 
bijection on the equivalence classes of flat Maurer-Cartan elements, then 
$F$ induces a bijection $F_\MC$ on the equivalence classes of curved 
Maurer-Cartan elements. 
\end{corollary}
\begin{proof}
We know that $F$ maps $m$ to a curved Maurer-Cartan element $m' =
F_\MC(m)= F^1(\cc{\exp}m) \in \mathcal{F}^1 L'^1$ and by Lemma~\ref{lemma:CorrespCurvedandFlatMC} we know 
that 
\begin{equation*}
  \Def(L,Q)
	\cong 
	\Def(L,Q^m),
	\quad \quad
	\Def(L',Q')
	\cong 
	\Def(L',(Q')^{m'}).
\end{equation*}
But by assumption we have 
\begin{equation*}
  \Def(L,Q^m)
	\cong
	\Def(L',(Q')^{m'})
\end{equation*}
and the statement is clear.
\end{proof}

%
%
\section{Homotopic $L_\infty$-morphisms}
\label{sec:HomotopyTheoryLinftyAlgandMorph}

One of our aims is to investigate the relation between twisted $L_\infty$-morphisms. 
More explicitly, let 
$F \colon (\liealg{g},Q) \rightarrow (\liealg{g}',Q')$ be an $L_\infty$-morphism 
between DGLAs and let $\pi \in \mathcal{F}^1\liealg{g}^1$ be a 
Maurer-Cartan element equivalent to zero. Then we want to take a look at  
the relation between $F^{\pi}$ and $F$. 

To this end we need to recall the definition 
of homotopic $L_\infty$-morphisms, which also has consequences 
for the homotopy classification of $L_\infty$-algebras. 
At first, let us recall from 
Corollary~\ref{cor:coderivationcogeneratorscocom} that an $L_\infty$-structure 
$Q$ on the vector space $L$ is equivalent to a Maurer-Cartan element in the DGLA
$(\Hom_\mathbb{K}^\bullet(\Sym(L[1]),L[1]),0,[\argument{,}\argument]_{NR})$. 
Analogously, we show now that we 
can also interpret $L_\infty$-morphisms as Maurer-Cartan elements in 
a convolution $L_\infty$-algebra.

%
%
\subsection{$L_\infty$-morphisms as Maurer-Cartan Elements}
\label{sec:LinftyMorphasMC}

Let $(L,Q),(L',Q')$ be two flat 
$L_\infty$-algebras and consider the space $\Hom(\cc{\Sym}(L[1]),L')$ 
of graded linear maps. If $L$ and $L'$ are equipped with complete descending 
filtrations, then we require the maps to be compatible with the filtrations. 
We can interpret elements $F^1,G^1\in \Hom(\cc{\Sym}(L[1]),L')[1]$ as maps in 
$\Hom(\cc{\Sym}(L[1]),\cc{\Sym}(L'[1]))$, where we have a convolution product $\star$, 
compare Definition~\ref{def:ConvProduct}. For example, one has
\begin{equation*}
  F^1 \star G^1
	=
	\vee \circ (F^1 \otimes G^1) \circ \cc{\Delta_\sh} 
	\colon 
	\cc{\Sym}(L[1])
	\longrightarrow
	\Sym^2(L[1]),
\end{equation*}
and since $\cc{\Delta_\sh}$ and $\vee$ are (co-)commutative, one directly sees that 
$\star$ is graded commutative. We can use this and the 
$L_\infty$-structures on $L$ and $L'$ to define an $L_\infty$-structure on this 
vector space of maps, see \cite[Proposition~1 and Proposition~2]{dolgushev:2007a} 
and also \cite{bursztyn.dolgushev.waldmann:2012a} for the case of DGLAs.

\begin{proposition}
\label{prop:ConvLinftyStructure}
	The coalgebra $\cc{\Sym}(\Hom(\cc{\Sym}(L[1]),L')[1])$ can be equipped with 
	a codifferential $\widehat{Q}$ with structure maps
  \begin{equation}
	  \label{eq:DiffonConvLinfty}
	  \widehat{Q}^1_1 F
		=
		Q'^1_1 \circ F - (-1)^{\abs{F}} F \circ Q
	\end{equation}
	and 
	\begin{equation}
		\label{eq:BracketonConvLinfty}
		\widehat{Q}^1_n(F_1\vee \cdots \vee F_n)
		=
		(Q')^1_n \circ 
		(F_1\star F_2\star \cdots \star F_n).
	\end{equation}
	It is called \emph{convolution $L_\infty$-algebra} and its 
	Maurer-Cartan elements can be identified with $L_\infty$-morphisms. 
	Here $\abs{F}$ denotes the degree in 
	$\Hom(\cc{\Sym}(L[1]),L')[1]$. 
\end{proposition}
\begin{proof}
The fact that this yields a well-defined $L_\infty$-structure follows 
directly from the fact that $L$ and $L'$ are $L_\infty$-algebras, and 
in particular from the cocommutativity and coassociativity of $\cc{\Delta_\sh}$. 

Now we want to show that the Maurer-Cartan elements are indeed in one-to-one 
correspondence with the $L_\infty$-morphisms.
At first we recall that a coalgebra morphism $F$ from 
$\cc{\Sym}(L[1])$ into $\cc{\Sym}(L'[1])$ is 
uniquely determined by its projection $F^1$ to $L'$ via $F=\exp_\star(F^1)$, compare 
Theorem~\ref{thm:CofreeCocomConilpotentCoalg}. Consequently, we can 
identify it with a degree one element in  
$\Hom(\cc{\Sym}(L[1]),L')$. It remains to show that the 
Maurer-Cartan equation is equivalent to the fact that $F$ commutes with the 
codifferentials. But again by Theorem~\ref{thm:CofreeCocomConilpotentCoalg} one 
sees that $Q' F = F Q$ is equivalent to 
$\pr_{L'[1]} (Q'F - FQ) = 0$ which is just the Maurer-Cartan equation for $F^1$.
\end{proof}

\begin{example}[Convolution DGLA]
  Let $\liealg{g},\liealg{g}'$ be two DGLAs. Then 
	$\Hom(\cc{\Sym}(\liealg{g}[1]),	\liealg{g}')$ is in fact a DGLA, the 
	so-called \emph{convolution DGLA} with differential
	\begin{equation}
	  \label{eq:DiffonConvDGLA}
	  \del F
		=
		\D' \circ F + (-1)^{\abs{F}} F \circ Q
	\end{equation}
	and bracket
	\begin{equation}
		\label{eq:BracketonConvDGLA}
		[F,G]
		=		
		- (-1)^{\abs{F}} (Q')^1_2 \circ (F\star G) .
	\end{equation}
	Here $\abs{F}$ denotes again the degree in 
	$\Hom(\cc{\Sym}(\liealg{g}[1]),\liealg{g}')[1]$ and the induced codifferential 
	is also denoted by $\widehat{Q}$. 
\end{example}

In order to obtain a notion of equivalent Maurer-Cartan elements we need a 
complete filtration. Note that the convolution $L_\infty$-algebra 
$\mathcal{H} = \Hom(\cc{\Sym}(L[1]),L')$ is indeed 
equipped with the following complete descending filtration:
\begin{align}
  \label{eq:FiltrationConvLieAlg}
  \begin{split}
  \mathcal{H}
	& =
	 \mathcal{F}^1\mathcal{H}
	\supset
	\mathcal{F}^2\mathcal{H}
	\supset \cdots \supset
	\mathcal{F}^k\mathcal{H}
	\supset \cdots \\
	\mathcal{F}^k\mathcal{H}
	& =
	\left\{ f \in \Hom(\cc{\Sym}(L[1]),L') \mid 
	f \at{\Sym^{<k}(L[1])}=0\right\}.
	\end{split}
\end{align}
Thus all twisting procedures are well-defined and 
one can define a notion of homotopic $L_\infty$-morphisms.

\begin{definition}
  \label{def:homotopicMorph}
  Two $L_\infty$-morphisms $F,F'$ between flat $L_\infty$-algebras 
	$(L,Q)$ and $(L',Q')$ 
	are called \emph{homotopic} if they are homotopy equivalent Maurer-Cartan 
	elements in the convolution $L_\infty$-algebra $(\Hom(\cc{\Sym}(L[1]),	L'),\widehat{Q})$. 
\end{definition}

However, we are mainly interested in $L_\infty$-morphisms between 
$L_\infty$-algebras resp. DGLAs with complete filtrations, whence we 
introduce a new filtration on the convolution $L_\infty$-algebra 
$\mathcal{H} = \Hom(\cc{\Sym}(L[1]),L')$ that takes into 
account the filtrations on $\cc{\Sym}(L[1])$ and $L'$:
\begin{align}
  \label{eq:FiltrationConvLieAlg2}
  \begin{split}
  \mathcal{H}
	& =
	 \mathfrak{F}^1\mathcal{H}
	\supset
	\mathfrak{F}^2\mathcal{H}
	\supset \cdots \supset
	\mathfrak{F}^k\mathcal{H}
	\supset \cdots \\
	\mathfrak{F}^k\mathcal{H}
	& =
	\sum_{n+m=k}
	\left\{ f \in \Hom(\cc{\Sym}(L[1]),L') \mid 
	f \at{\Sym^{<n}(L[1])}=0\quad \text{ and }\quad 
	f\colon\mathcal{F}^\bullet \rightarrow 
	\mathcal{F}^{\bullet +m}\right\}.
	\end{split}
\end{align}
Here the filtration on $\cc{\Sym}(L[1])$ is the product filtration 
induced by
\begin{equation*}
  \mathcal{F}^k(L[1] \otimes L[1])
	=
	\sum_{n+m=k} \image\left(  \mathcal{F}^nL[1] \otimes 
	\mathcal{F}^mL[1] \rightarrow L[1]\otimes L[1]\right),
\end{equation*}
see e.g. \cite[Section~1]{dotsenko.shadrin.vallette:2018}.

\begin{proposition}
\label{prop:CompleteFiltrationConvLinfty}
  The above filtration \eqref{eq:FiltrationConvLieAlg2} is a complete descending 
	filtration on the convolution $L_\infty$-algebra 
	$\Hom(\cc{\Sym}(L[1]),L')$.
\end{proposition}
\begin{proof}
The filtration is obviously descending and 
$\mathcal{H}=\mathfrak{F}^1\mathcal{H}$ since we consider in the convolution 
$L_\infty$-algebra only maps that are compatible with the filtration. 
It is compatible with the convolution $L_\infty$-algebra 
structure and complete since $L'$ is complete.
\end{proof}

Recall that we introduced in Remark~\ref{rem:curvedMorphisms} the definition of 
curved morphisms between curved Lie algebras from \cite[Definition~4.3]{maunder:2017a}. 
There exists 
a similar generalizations of curved morphisms between curved $A_\infty$-algebras 
\cite{positselski:2018a}. Considering now the convolution $L_\infty$-algebra between two 
curved $L_\infty$-algebras we get directly the analogue generalization of 
curved morphisms for the $L_\infty$-setting:

\begin{remark}[Curved convolution $L_\infty$-algebra]
  \label{rem:CurvedMorphLinfty}
  Let us consider now two curved $L_\infty$-algebras $(L,Q)$ and $(L',Q')$. 
	Here we use the counital coaugmented coalgebra 
	$\Sym(\Hom(\Sym(L[1]),L')[1])$ with coproduct $\Delta_\sh$ and codifferential 
	$\widehat{Q}$ with Taylor components \eqref{eq:DiffonConvLinfty} and 
	\eqref{eq:BracketonConvLinfty} as in the flat case plus curvature component
	\begin{equation*}
	  \widehat{Q}_0^1 
		=
		(1 
		\longmapsto 
		Q'_0(1) )
		\in 
		(\Hom(\mathbb{K},L')[1])^1.
	\end{equation*}
	This gives indeed a \emph{curved convolution $L_\infty$-algebra} and we want 
	to interpret its Maurer-Cartan elements. We restrict ourselves to the filtered setting 
	where we set $\mathcal{F}^0\mathbb{K}=\mathbb{K}$, and where we 
	assume $Q'_0\in \mathcal{F}^1L'$. Then Mauer-Cartan elements 
	$F \in \mathfrak{F}^1(\Hom(\Sym(L[1]),L')[1])^0$ with the 
	filtration from \eqref{eq:FiltrationConvLieAlg2} are given by 
	Taylor components $F_n^1 \colon \Sym^n (L[1]) 
	\rightarrow L'[1]$ for $n\geq 0$. The only difference to the flat setting is 
	the zero component $F_0^1(1) = \alpha \in \mathcal{F}^1L'^1$. 
	Then the Maurer-Cartan equation implies
	\begin{align}
	  \label{eq:curvedmorphLinfty}
		\tag{$*$}
	  0
		=
		\widehat{Q}_0^1 + Q'^1_1 \circ F^1 - F^1\circ Q + 
		\sum_{n=2}^\infty \frac{1}{n!} Q'^1_n \circ (F^1)^{\star n}.
	\end{align}
	If $F^1_0(1)=0$ the evaluation at $1$ yields $Q'_0(1) = F^1_1(Q_0(1))$ and thus 
	$F$ is a $L_\infty$-morphism of curved $L_\infty$-algebras 
	in the usual sense, i.e.\ a coalgebra morphism commuting with the codifferentials. But for $F_0^1(1)=\alpha\neq 0$ we no longer get an induced coalgebra morphism since we no longer have $F(1)=1$. 
	In analogy to \cite[Definition~4.3]{maunder:2017a} we call such a general $F$ 
	\emph{curved morphism of $L_\infty$-algebras} and for $\alpha=0$ 
	the morphism $F$ is called \emph{strict}. Note that in 
	\cite{getzler:2018a} those curved $L_\infty$-morphisms are just called 
	$L_\infty$-morphism. 
	We will study the curved setting in more details in 
	Section~\ref{sec:HomTheoryofCurvedLinfty}, where we show that 
	curved $L_\infty$-morphisms still satisfy some nice properties; in 
	particular they are compatible with Maurer-Cartan elements, 
	see e.g. Proposition~\ref{prop:ProvertiesofCurvedLinftyMorph}.
\end{remark}


For now we restrict us to the simpler 
flat case. Therefore, from now on, unless stated otherwise, all DGLAs and 
$L_\infty$-algebras are assumed to be flat. We collect a few immediate consequences about 
homotopic $L_\infty$-morphisms, see e.g. \cite[Proposition~1.4.6]{kraft:2021a}:

\begin{proposition}
\label{prop:PropertiesofHomotopicMorphisms}
  Let $F,F'$ be two homotopic $L_\infty$-morphisms 
	between the flat $L_\infty$-algebras $(L,Q)$ and $(L',Q')$.
	\begin{propositionlist}
		\item $F_1^1$ and $(F')_1^1$ are chain homotopic.
		\item If $F$ is an $L_\infty$-quasi-isomorphism, then so is $F'$.
		\item $F$ and $F'$ induce the same maps from $\Def(L)$ to 
		      $\Def(L')$, i.e. $F_\MC = F'_\MC $.
		\item In the case of DGLAs $\liealg{g},\liealg{g}'$, compositions 
		      of homotopic $L_\infty$-morphisms with a 
		      DGLA morphism of degree zero are again homotopic.
	\end{propositionlist}
\end{proposition}
\begin{proof}
Concerning the first two points, let $F^1(t)$ and $\lambda^1(t)$ be the 
paths encoding the homotopy equivalence, i.e. 
\begin{equation}
\label{eq:HomEquFFprime}
\tag{$*$}
  \frac{\D}{\D t} F^1(t)
	=
	\widehat{Q}^1(\lambda^1(t)\vee \exp(F^1(t)))
\end{equation}
with $F^1(0)=F^1$ and $F^1(1)=(F')^1$.
In particular, this implies $\frac{\D}{\D t} F_1^1(t)
= (Q')^1_1 \circ \lambda_1^1(t) + \lambda_1^1(t) \circ Q^1_1$ which gives 
the statement with $F_1^1(0)=F_1^1$. 

For DGLAs the third point is proven in \cite[Lemma~B.5]{bursztyn.dolgushev.waldmann:2012a}. 
In our general setting we consider a
Maurer-Cartan element $\pi \in \mathcal{F}^1L^1$ and recall that
$\cc{\exp}(\pi)=\sum_{k=1}^\infty \frac{1}{k!} \pi^{\vee k}$ 
satisfies $\cc{\Delta_\sh}\cc{\exp}(\pi) 
= \cc{\exp}(\pi) \otimes \cc{\exp}(\pi)$ and 
$ Q \cc{\exp}(\pi) = 0$ by Lemma~\ref{lemma:twistinglinftymorphisms}.
Applying now \eqref{eq:HomEquFFprime} on $\cc{\exp}\pi$ gives
\begin{align*}
  \frac{\D}{\D t} F^1(t)(\cc{\exp}\pi)
	& =
	\widehat{Q}^1(\lambda^1(t)\vee \exp(F^1(t))(\cc{\exp}\pi) \\
	& =
  (Q')^1(\lambda^1(t)(\cc{\exp}\pi) \vee \exp(F^1(t)(\cc{\exp}\pi) ),
\end{align*}
i.e. $\pi(t) = F^1(t)(\cc{\exp}\pi) $ and $\lambda(t) = \lambda^1(t)(\cc{\exp}\pi) $ 
encode the homotopy equivalence between $F_\MC(\pi)= F^1(\cc{\exp}\pi) $ 
and $F'_\MC(\pi)=(F')^1(\cc{\exp}\pi) $.

%
The last point follows directly
since DGLA morphisms commute with brackets and differentials. 
\end{proof}

We want to generalize the last point to compositions with 
$L_\infty$-morphisms. Since we could not find a reference we prove the 
statements in detail. We start with the post-composition as in 
\cite[Proposition~3.5]{kraft.schnitzer:2021a:pre}.

\begin{proposition}
  \label{prop:CompofHomotopicHomotopic}
  Let $F_0,F_1$ be two homotopic $L_\infty$-morphisms 
	from $(L,Q)$ to $(L',Q')$. Let 
	$H$ be an $L_\infty$-morphism from $(L',Q')$ to 
	$(L'',Q'')$, then $HF_0 \sim HF_1$. 
\end{proposition}
\begin{proof}
For $F^1\in \Hom(\cc{\Sym}(L[1]),L')$ we write $ H^1F	= H^1 \circ \exp_\star(F^1)$, 
where $\star$ denotes again the convolution product with respect to $\vee$ and 
$\Delta_\sh$ resp. $\cc{\Delta_\sh}$. 
Let us denote by $F^1(t) \in \widehat{(\Hom(\cc{\Sym}
(L[1]),L')[1])^0[t]}$ and $\lambda^1(t) \in 
\widehat{(\Hom(\cc{\Sym}(L[1]),L')[1])^{-1}[t]}$ 
the paths encoding the homotopy equivalence between $F_0$ and $F_1$. Then 
$H^1F(t) \in 
\widehat{(\Hom(\cc{\Sym}(L[1]),L'')[1])^{-1}[t]}$ 
satisfies
\begin{align*}
  \frac{\D}{\D t} H^1 F(t)
	& =
	H^1 \circ 
	\left( \widehat{Q}^1(\lambda^1(t)\vee \exp(F^1(t))) \star \exp_\star(F^1)\right) .
\end{align*}
Since $F^1(t)$ is a path of Maurer-Cartan elements, we get 
\begin{align*}
  \frac{\D}{\D t} H^1F(t)
	& =
	H^1 \circ Q' \circ
	(\lambda^1(t) \star \exp_\star(F^1(t)))
	+
	H^1\circ(\lambda^1(t) \star \exp_\star(F^1(t)))  \circ Q  \\
	& =
	(Q'')^1 \circ H \circ
	(\lambda^1(t) \star \exp_\star(F^1(t)))  
	+
	H^1\circ(\lambda^1(t) \star \exp_\star(F^1(t)))   \circ Q  \\
	& =
	(\widehat{Q}')^1_1 \left(  H^1 \circ
	(\lambda^1(t) \star \exp_\star(F^1(t)))\right) 
	+ \sum_{\ell=2}^\infty (Q'')^1_\ell \circ H^\ell \circ
	(\lambda^1(t) \star \exp_\star(F^1(t))).
\end{align*}
Finally, we know from Theorem~\ref{thm:CofreeCocomConilpotentCoalg} that 
$\lambda^1(t) \star \exp_\star(F^1(t))$ is a coderivation along $F(t)$ and we get 
for the second term
\begin{align*}
  H^\ell \circ (\lambda^1(t) \star \exp_\star(F^1(t)))
	& =
	\left((H^1\circ(\lambda^1(t)\star \exp_\star F^1)) \star 
	\frac{1}{(\ell-1)!}(H^1F)^{\star (\ell-1)}\right).
\end{align*}
Summarizing, we have
\begin{align*}
  \frac{\D}{\D t} H^1F(t)
	=
	(\widehat{Q}')^1  
	\left((H^1\circ(\lambda^1(t)\star \exp_\star F^1)) \vee
	\exp(H^1F)\right)
\end{align*}
and the statement is shown.
\end{proof}

Analogously, we have for the pre-composition \cite[Proposition~3.6]{kraft.schnitzer:2021a:pre}:

\begin{proposition}
  \label{prop:preCompofHomotopicHomotopic}
  Let $F_0,F_1$ be two homotopic $L_\infty$-morphisms 
	from $(L,Q)$ to $(L',Q')$. Let 
	$H$ be an $L_\infty$-morphism from $(L'',Q'')$ to 
	$(L,Q)$, then $F_0 H\sim F_1H$. 
\end{proposition}
\begin{proof}
Let $F^1(t) \in \widehat{(\Hom(\cc{\Sym}
(L[1]),L')[1])^0[t]}$ and $\lambda^1(t) \in 
\widehat{(\Hom(\cc{\Sym}(L[1]),L')[1])^{-1}[t]}$ 
describe the homotopy equivalence between $F_0$ and $F_1$. Then 
we consider 
\begin{equation*}
  F^1(t)H
	=
	F^1(t) \circ \exp_\star(H^1)
	\in
	\widehat{(\Hom(\cc{\Sym}
  (L''[1]),L')[1])^0[t]}
\end{equation*}
in the notation of the above proposition. We compute
\begin{align*}
  \frac{\D}{\D t} (F^1(t)H)
	& =
	\widehat{Q}^1(\lambda^1(t)\vee \exp(F^1(t))) \circ H \\
	& =
	(Q')^1_1 \circ \lambda^1 \circ H 
	+ \lambda^1 \circ Q \circ H
	+ \sum_{\ell=2}^\infty\frac{1}{(\ell-1)!} 
	(Q')^1_\ell \circ (\lambda^1\star F^1\star \cdots \star F^1)
	\circ H \\
	& =
	(Q')^1_1 \circ \lambda^1 \circ H 
	+ \lambda^1 \circ H\circ Q'' 
	+ \sum_{\ell=2}^\infty\frac{1}{(\ell-1)!} 
	(Q')^1_\ell \circ (\lambda^1 H\star F^1 H\star \cdots \star F^1 H) \\
	& =
	\widehat{Q}^1 (\lambda^1(t)H\vee \exp(F^1(t)H))
\end{align*}
since $H$ is a coalgebra morphism intertwining $Q''$ and $Q$ and of 
degree zero. Finally, since $\lambda^1(t)H \in 
\widehat{(\Hom(\cc{\Sym}(L''[1]),L')[1])^{-1}[t]}$ the 
statement follows.
\end{proof}

\begin{corollary}
  Let $F_0,F_1$ be two homotopic $L_\infty$-morphisms 
	from $(L,Q)$ to $(L',Q')$, and let 
	$H_0,H_1$ be two homotopic $L_\infty$-morphisms from $(L',Q')$ to 
	$(L'',Q'')$, then $H_0 F_0\sim H_1F_1$. 
\end{corollary}

We want to end this section with an example of homotopic $L_\infty$-morphisms: 
Let $(\liealg{g},\D,[\argument{,}\argument])$ be a DGLA with complete descending 
filtration, and assume that $h\in \mathcal{F}^1\liealg{g}^0$. 

\begin{proposition}
  If $g$ is closed, then 
	\begin{equation}
		\label{eq:EtoAdh}
		e^{[g,\argument]} \colon 
		(\liealg{g},\D,[\argument{,}\argument])
		\longrightarrow
		(\liealg{g},\D,[\argument{,}\argument])
	\end{equation}
	is an DGLA automorphism that maps equivalent Maurer-Cartan elements 
	to equivalent ones. If $g$ is exact, then it is even homotopic to 
	the identity $\id_\liealg{g}$.
\end{proposition}
\begin{proof}
If $g$ is closed, then $\D \circ e^{[g,\argument]} =
 e^{[g,\argument]}\circ \D$, see e.g. Lemma~\ref{lemma:TwistedDGLAsIsomorphic}. 
Let now $\pi \in \mathcal{F}^1\liealg{g}^1$ be a Maurer-Cartan element. 
Then we know from the formula for the gauge action from 
Proposition~\ref{prop:GaugeactionDGLA} that we have
\begin{equation*}
  \exp([g,\argument])\acts \pi
	=
	e^{[g,\argument]}\pi,
\end{equation*}
i.e. $e^{[g,\argument]}$ maps a Maurer-Cartan element to an equivalent one. 
Suppose now that $g = - \D \alpha$ with $\alpha \in 
\mathcal{F}^1\liealg{g}^{-1}$. Then we set
\begin{equation*}
  \pi(t)
	=
	e^{[tg,\argument]}, 
	\quad \quad
	\lambda(t)
	=
	e^{[tg,\argument]} \circ [\alpha,\argument].
\end{equation*}
We have $\pi(0)=\id_\liealg{g}$ and $\pi(1)= e^{[g,\argument]}$, 
and we want to show that $\pi(t)$ and $\lambda(t)$ encode the homotopy equivalence relation as in Definition~\ref{def:HomEquivalenceofMC}. 
Using the formulas for the convolution $L_\infty$-structure, we have to show
\begin{align*}
  e^{[tg,\argument]} \circ [g,\argument]
	=
	\frac{\D}{\D t} \pi(t)
	\stackrel{!}{=}
	\widehat{Q}^1(\lambda(t)\vee \exp(\pi(t))).
\end{align*}
For the right hand side we compute
\begin{align*}
  \widehat{Q}^1(\lambda(t)\vee \exp(\pi(t)))
	& =
	- \D \circ \lambda(t) - \lambda(t) \circ \D 
	+ \lambda(t) \circ Q^1_2
	+ Q^1_2 \circ (\lambda(t) \star \pi(t)),
\end{align*}
where $Q^1_2(x \vee y) = -(-1)^{\abs{x}} [x,y]$ for $x,y \in \liealg{g}$ 
with $x \in (\liealg{g}[1])^{\abs{x}}$. For the first two terms we get
\begin{align*}
  - \D \circ \lambda(t) - \lambda(t) \circ \D
	= 
	- \D \circ e^{[tg,\argument]} \circ [\alpha,\argument] 
	- e^{[tg,\argument]} \circ [\alpha,\argument] \circ \D 
	=
	-   e^{[tg,\argument]} \circ [\D \alpha,\argument]
	=
	 e^{[tg,\argument]} \circ [g,\argument].
\end{align*}
Thus we only have to show
\begin{equation*}
   \lambda(t) \circ Q^1_2
	+ Q^1_2 \circ (\lambda(t) \star \pi(t))
	=
	0.
\end{equation*}
For homogeneous $x \in (\liealg{g}[1])^{\abs{x}}, 
y\in (\liealg{g}[1])^{\abs{y}}$ we compute
\begin{align*}
  \lambda(t) \circ Q^1_2(x\vee y)
	& + Q^1_2 \circ (\lambda(t) \star \pi(t))(x \vee y)
	= 
	-(-1)^{\abs{x}} e^{[tg,\argument]} \circ [\alpha,\argument] ([x,y]) \\
	& -(-1)^{\abs{x} -1} [ e^{[tg,\argument]}  [\alpha,x], 
	e^{[tg,\argument]}y] 
	- (-1)^{\abs{x}\abs{y}} (-1)^{\abs{y} -1} 
	[ e^{[tg,\argument]}  [\alpha,y], 
	e^{[tg,\argument]}x] 
	=
	0,	
\end{align*}
and the proposition is shown.
\end{proof}

\subsection{Homotopy Classification of Flat $L_\infty$-algebras}
\label{sec:HomClassLinftyAlgs}
The above considerations allow us to understand the homotopy classification 
of flat $L_\infty$-algebras \cite{canonaco:1999a,kontsevich:2003a} in a better way.

\begin{definition}
\label{def:HomEquLinftyAlgs}
Two  flat $L_\infty$-algebras $(L,Q)$ and $(L',Q')$ are 
said to be 
\emph{homotopy equivalent} if there are $L_\infty$-morphisms 
$F\colon (L,Q)\to (L',Q')$ and 
$G\colon (L',Q')\to (L,Q)$ 
such that $F\circ G\sim \id_{L'}$ and 
$G\circ F\sim \id_{L}$. In that case $F$ and $G$ are said
to be quasi-inverse to each other.
\end{definition}

As in \cite[Lemma~3.8]{kraft.schnitzer:2021a:pre} we can immediately show that this definition coincides indeed with the definition of homotopy equivalence 
via $L_\infty$-quasi-isomorphisms from \cite{canonaco:1999a}.

\begin{lemma}
\label{lemma:HomEquvsQuasiIso}
Two flat $L_\infty$-algebras $(L,Q)$ and $(L',Q')$ are 
homotopy equivalent if and only if there exists an $L_\infty$-quasi-isomorphism 
between them.
\end{lemma}

\begin{proof}
Due to Proposition~\ref{prop:StandardFormLinfty} every $L_\infty$-algebra $L$ is 
isomorphic to the product of a linear contractible one and a minimal 
one $L[1]\cong V\oplus W$. 
This means $L[1]\cong V\oplus W$ as vector spaces, such that 
$V$ is an acyclic cochain complex with differential $\D_V$ and $W$ 
is an $L_\infty$-algebra with codifferential $Q_W$ with 
$Q_{W,1}^1=0$. The codifferential $Q$ on $\cc{\Sym}(V\oplus W)$ is given on 
$v_1\vee \dots \vee v_m$ with $v_1,\dots, v_k\in V$ and 
$v_{k+1},\dots, v_m\in W$ by 
	\begin{align*}
	Q^1(v_1\vee\dots \vee v_m)=
	\begin{cases}
	-\D_V(v_1), & \text{ for } k=m=1\\
	Q_W^1(v_1\vee\dots\vee v_m), & \text{ for } k=0\\
	0, & \text{ else. } 
	\end{cases}
	\end{align*}
This implies in particular that the canonical maps
	\begin{align*}
	i\colon W\longrightarrow V \oplus W \ \text{ and } \ 
	p\colon V\oplus W\longrightarrow W 
	\end{align*}
are $L_\infty$-quasi-isomorphisms. We want to show now that $i\circ p\sim \id$ 
and therefore choose a contracting homotopy 
$h_V\colon V\to V$ with $h_V\D_V+\D_V h_V=\id_V$ and define the maps
\begin{align*}
	P(t)\colon 
	V\oplus W\ni (v,w)
	\longmapsto (tv,w)
	\in V\oplus W
\end{align*}	 
and 
\begin{align*}
	H(t)
	=
	H
  \colon V\oplus W\ni (v,w)\longmapsto (-h_V(v),0)\in V\oplus W.
\end{align*}
Note that $P(t)$ is a path of $L_\infty$ morphisms because of the 
explicit form of the codifferential. We clearly have
\begin{align*}
	\frac{\D}{\D t} P^1_1(t) 
	= 
	\pr_V
	=
	Q^1_1 \circ H(t) + H(t)\circ Q^1_1
	=
	\widehat{Q}^1_1(H(t)) 
\end{align*}
since $h_V$ is a contracting homotopy. This implies
\begin{equation*}
  \frac{\D}{\D t} P^1(t)
	=
	\widehat{Q}^1(H(t)\vee \exp(P(t))
\end{equation*} 
as $\image (H(t))\subseteq V$ and as the higher brackets $Q$ vanish 
on $V$. From  
$P(0)=i\circ p$ and $P(1)=\id$ we conclude that 
$i\circ p\sim \id$.
We choose a similar splitting for $L'[1]=V'\oplus W'$
 with the same
properties and consider an $L_\infty$-quasi-isomorphism 
$F\colon L\to L'$. In Theorem~\ref{thm:QuisInverse} we constructed 
an $L_\infty$-quasi-inverse $G= i\circ(F_{min})^{-1} \circ p'$.
Since by Proposition~\ref{prop:CompofHomotopicHomotopic} and 
Proposition~\ref{prop:preCompofHomotopicHomotopic} compositions of homotopic 
$L_\infty$-morphisms with an $L_\infty$-morphism are again homotopic, we get
	\begin{align*}
	F\circ G&
	= 
	F\circ i\circ (F_{min})^{-1}\circ p'
	\sim 
	i'\circ p'\circ  F\circ i\circ (F_{min})^{-1}\circ p'\\&
	=i'\circ F_{min}\circ (F_{min})^{-1}\circ p'=i'\circ p'\sim \id
	\end{align*}
and similarly $G\circ F \sim \id$.   

The other direction follows from Proposition~\ref{prop:PropertiesofHomotopicMorphisms}. Suppose $F\circ G \sim \id$ and 
$G\circ F \sim \id$, then we know that $F_1^1\circ G_1^1$ and 
$G_1^1\circ F_1^1$ are both chain homotopic to the identity. Therefore, 
$F$ and $G$ are $L_\infty$-quasi-isomorphisms.  
 \end{proof}
 
\begin{corollary}
Let $F\colon (L,Q)\to(L',Q')$ be a an 
$L_\infty$-quasi-isomorphism with two given quasi-inverses
$G,G'\colon (L',Q')\to (L,Q)$ in the sense of 
Definition~\ref{def:HomEquLinftyAlgs}. Then one has $G\sim G'$. 
\end{corollary} 
 
\begin{proof}
One has 
	\begin{align*}
	G\sim G\circ (F\circ G')=(G\circ F)\circ G'\sim G'
	\end{align*}
and the statement is shown.
\end{proof} 

As a first application, we want to show that the construction of the 
morphisms for the homotopy transfer theorem \ref{thm:HTTJonas} are natural 
with respect to homotopy equivalences   

\begin{corollary}
\label{cor:IPsimId}
  In the setting of Theorem \ref{thm:HTTJonas} one has $ P \circ I = \id_A$ and 
	$I \circ P \sim \id_B$.
\end{corollary}
\begin{proof}
By Lemma~\ref{lemma:HomEquvsQuasiIso} $P$ admits a quasi-inverse $I'$ such that $P \circ I' \sim \id_A$ and $I'\circ P\sim\id_B$, 
which implies
\begin{equation*}
  I \circ P 
	=
	\id_B \circ I \circ P
	\sim 
	I' \circ P \circ I \circ P 
	=
	I' \circ P
	\sim 
	\id_B,
\end{equation*}
and the statement is shown.
\end{proof}
%
%
\subsection{Homotopy Equivalence between Twisted Morphisms}
\label{sec:HomEquivTwistedMorphisms}

Let now $F \colon (\liealg{g},\D,[\argument{,}\argument]) \rightarrow (\liealg{g}',
\D',[\argument{,}\argument])$ be an 
$L_\infty$-morphism between (flat) DGLAs with complete descending and 
exhaustive filtrations.  
Instead of comparing the twisted morphisms $F^\pi$ and 
$F^{\pi'}$ with respect to two equivalent Maurer-Cartan elements 
$\pi$ and $\pi'$, we consider 
for simplicity just a Maurer-Cartan element $\pi \in \mathcal{F}^1\liealg{g}^1$ 
 equivalent to zero via 
$\pi = \exp([g,\argument])\acts 0$, i.e. $\lambda(t)=g = \dot{A}(t)\in 
\mathcal{F}^1\widehat{\liealg{g}^0[t]}$. 
Then we know that $0$ and $S =F_\MC(\pi)= F^1(\cc{\exp}(\pi))\in \mathcal{F}^1(\liealg{g}')^1$
are equivalent 
Maurer-Cartan elements in $(\liealg{g}',\D')$. Let the equivalence 
be implemented by an $A'(t)\in\mathcal{F}^1\widehat{(\liealg{g}')^0[t]}$ as in Proposition~\ref{prop:HomEquvsGaugeEqu}.  
Then we have the diagram of $L_\infty$-morphisms between (flat) DGLAs
\begin{equation}
  \label{eq:TwistingofMorph}
	\begin{tikzcd}
	& (\liealg{g}',\D') \arrow[rd, "e^{[A'(1),\argument]}", bend left=12] \arrow[dd, Rightarrow,  shorten >=12pt,  shorten <=12pt] & \\
    (\liealg{g},\D) 	\arrow[ur,"F",bend left=12] \arrow[dr,swap, "e^{[A(1),\argument]}",bend right=12] & & 
    (\liealg{g}',\D' + [S,\argument])\\
	&(\liealg{g},\D + [\pi,\argument]) \arrow[ur, swap,"F^\pi", bend right=12] &
	\end{tikzcd}
\end{equation}
where $e^{[A(1),\argument]}$ and $e^{[A'(1),\argument]}$ are well-defined by 
the completeness of the filtrations. Following \cite[Proposition~3.10]{kraft.schnitzer:2021a:pre},
we show that it commutes 
up to homotopy, which is indicated by the vertical arrow. 

\begin{proposition}
  \label{prop:TwistMorphHomEqu}
  The $L_\infty$-morphisms $F$ and $e^{[-A'(1),\argument]}\circ F^{\pi} \circ 
	e^{[A(1),\argument]}$ are homotopic, i.e. 
	homotopy equivalent Maurer-Cartan elements 
	in $(\Hom(\cc{\Sym}(\liealg{g}[1])),\liealg{g}'),\widehat{Q})$.
\end{proposition}
The candidate for the path between $F$ and 
$e^{[-A'(1),\argument]}\circ F^{\pi} \circ 
	e^{[A(1),\argument]}$ is 
\begin{equation*}
  F(t) 
	= 
	e^{[-A'(t),\argument]}\circ F^{\pi(t)} \circ 
	e^{[A(t),\argument]}.
\end{equation*}
However, $F(t)$ is not necessarily in the completion 
$\widehat{\Hom(\cc{\Sym}(\liealg{g}[1]),\liealg{g}')^1[t]}$ 
with respect to the 
filtration from \eqref{eq:FiltrationConvLieAlg} since for example
\begin{align*}
  F(t) \;\,\mathrm{mod}\;\, \mathcal{F}^2\Hom(\cc{\Sym}(\liealg{g}[1]),\liealg{g}')[[t]]
	=
	 e^{[-A'(t),\argument]}\circ F^{\pi(t)}_1 \circ 
	e^{[A(t),\argument]}
\end{align*}
is in general not polynomial in $t$. But using the filtration from 
\eqref{eq:FiltrationConvLieAlg2} we can prove 
Proposition~\ref{prop:TwistMorphHomEqu}.

\begin{proof}[of Proposition~\ref{prop:TwistMorphHomEqu}]
  The path $F(t) = e^{[-A'(t),\argument]}\circ F^{\pi(t)} \circ 
	e^{[A(t),\argument]}$ is an element in the completion
	$\widehat{(\Hom(\cc{\Sym}(\liealg{g}[1]),\liealg{g}')[1])^{0}[t]}$ 
	with respect to the filtration from \eqref{eq:FiltrationConvLieAlg2}.
	This is clear since $A(t)\in \mathcal{F}^1 \widehat{\liealg{g}^0[t]}$, 
	$A'(t)\in \mathcal{F}^1 \widehat{(\liealg{g}')^0[t]}$ and 
	$\pi(t)\in \mathcal{F}^1 \widehat{\liealg{g}^1[t]}$ imply that
	\begin{align*}
	  \sum_{i=1}^{n-1}
	  e^{[-A'(t),\argument]}\circ F^{\pi(t)}_i \circ 
	  e^{[A(t),\argument]}
	  \;\,\mathrm{mod}\;\, \mathfrak{F}^n(\Hom(\cc{\Sym}(\liealg{g}[1]),\liealg{g}')[1])[[t]]
	\end{align*}
	is polynomial in $t$. Moreover, $F(t)$ satisfies by \eqref{eq:ODEforA}
	\begin{align*}
	  \frac{\D F(t)}{\D t}
		& = 
		-e^{[-A'(t),\argument]} \circ 
		\left[\lambda'(t) , 
	  \argument \right]
	  \circ F^{\pi(t)} \circ 
		e^{[A(t),\argument]}  
		+ 
		e^{[-A'(t),\argument]}\circ F^{\pi(t)} \circ 
		\left[\lambda(t), 
	  \argument \right] \circ 
	  e^{[A(t),\argument]} \\
		& \quad 
		+ e^{[-A'(t),\argument]}\circ \frac{\D F^{\pi(t)}}{\D t} \circ 
	  e^{[A(t),\argument]} .
	\end{align*}
	But we have
	\begin{align*}
	  \frac{\D F^{\pi(t)}_k}{\D t}&(X_1 \vee \cdots \vee X_k)
		=
		F_{k+1}^{\pi(t)}(Q^{\pi(t),1}_1(\lambda(t)) \vee X_1 \vee \cdots 
		\vee X_k) \\
		& =
		F_{k+1}^{\pi(t)}(Q^{\pi(t),k+1}_{k+1}(\lambda(t)\vee X_1\vee \cdots\vee X_k))
		+
		F_{k+1}^{\pi(t)}(\lambda(t) \vee Q^{\pi(t),k}_k(X_1 \vee \cdots \vee X_k)))\\
		& =
		Q^{S(t),1}_1 F_{k+1}^{\pi(t),1}(\lambda(t)\vee X_1 \vee \cdots \vee X_k) +
		Q^{S(t),1}_2 F_{k+1}^{\pi(t),2}(\lambda(t)\vee X_1 \vee \cdots \vee X_k) \\
		& \quad 
		-  F_{k}^{\pi(t),1}\circ Q^{\pi(t),k}_{k+1}
		(\lambda(t)\vee X_1 \vee \cdots \vee X_k) 
		+
		F_{k+1}^{\pi(t)}(\lambda(t) \vee Q^{\pi(t),k}_k(X_1 \vee \cdots \vee X_k)).
	\end{align*}
	Setting now $\lambda_k^F(t)(\cdots) 
	= F_{k+1}^{\pi(t)}( \lambda(t)\vee \cdots)$ we get
	\begin{align*}
	  \frac{\D F^{\pi(t)}_k}{\D t}
		=
		\widehat{Q}^{t,1}_1 (\lambda^F(t)) + 
		\widehat{Q}^{t,1}_2(\lambda^F(t)\vee F^{\pi(t)}) - 
		F^{\pi(t)}_k\circ[\lambda(t),\argument] +[\lambda'(t),\argument] \circ 
		F^{\pi(t)}_k.
	\end{align*}
	Thus we get
	\begin{align*}
	  \frac{\D F(t)}{\D t}
		& =
		e^{[-A'(t),\argument]}\circ \left(\widehat{Q}^{t,1}_1 (\lambda^F(t)) + 
		\widehat{Q}^{t,1}_2(\lambda^F(t)\vee F^{\pi(t)})\right) \circ 
	  e^{[A(t),\argument]} \\
		& =
		\widehat{Q}^{1}_1 (e^{[-A'(t),\argument]}\lambda^F(t)e^{[A(t),\argument]}) + 
		\widehat{Q}^{1}_2(e^{[-A'(t),\argument]}\lambda^F(t)e^{[A(t),\argument]}\vee 
		F(t))		
	\end{align*}
	since the $\exp([A(t),\argument])$ and 
	$\exp([A'(t),\argument])$ commute with the brackets and intertwine the 
	differentials. Thus $F(0) = F$ and $F(1)$ are homotopy equivalent.
\end{proof}

\begin{remark}[Application to Deformation Quantization]
  This result allowed us in \cite{kraft.schnitzer:2021a:pre} to prove that 
	Dolgushev's globalizations \cite{dolgushev:2005a,
dolgushev:2005b} of the Kontsevich formality \cite{kontsevich:2003a} with 
respect to different covariant derivatives are homotopic.
\end{remark}
Now we want to generalize the results from the above section to twisted morphisms 
between general $L_\infty$-algebras. As a first step, we have to generalize 
Lemma~\ref{lemma:TwistedDGLAsIsomorphic} and Corollary~\ref{cor:GaugeEquivMCTwistsQuis}, 
i.e. we have to show that $L_\infty$-algebras that are twisted with equivalent 
Maurer-Cartan elements are $L_\infty$-isomorphic.

\begin{lemma}
\label{lem:MorphPhitTwistedLinftyAlgs}
  Let $(L,Q)$ be a flat $L_\infty$-algebra with complete descending filtration, and let 
	$\pi(t)$ and $\lambda(t)$ encode a homotopy equivalence between two Maurer-Cartan 
	elements as in Definition~\ref{def:HomEquivalenceofMC}. For $a \in \Sym^i(L[1])$ with 
	$i \geq 0$ the recursively defined system of differential equations
	\begin{align}
	  \label{eq:ODEforPhit}
	  \begin{split}
		  \frac{\D}{\D t} (\Phi_t)^1_i(a)
			& = 
		  \sum_{k=1}^i\left(
			  [Q^{\pi(t)}, \lambda(t)\vee \argument] - Q^{\pi(t)}(\lambda(t))\vee \argument
			\right)^1_k (\Phi_t)^k_i(a)  \\
			& = 
			\sum_{k=1}^i  (Q^{\pi(t)})^1_{k+1}(\lambda(t)\vee (\Phi_t)^k_i(a)) ,
			\quad \quad 
			(\Phi_0)^1_i(a)
			= 
			\pr_{L[1]}(a)  
		\end{split}
	\end{align}
	has unique solutions $(\Phi_t)^1_i \colon \Sym^i(L[1]) \rightarrow \widehat{L[1][t]}$, 
	where $(\Phi_t)^k_i(a)$ depends indeed only on 
	$(\Phi_t)^1_j$ for $j\leq i-k+1$ as for $L_\infty$-morphisms. In fact, one has 
	$\Phi_t^1 \in \widehat{\Hom (\cc{\Sym}(L[1]),L)^1[t]}$.
\end{lemma}
\begin{proof}
The right hand side of the differential equation \eqref{eq:ODEforPhit} depends only on 
$(\Phi_t)^1_j$ with $j\leq i$. Thus it has a unique solution 
$(\Phi_t)^1_i(a) \in \widehat{L[t]}$ since 
$[Q^{\pi(t)}, \lambda(t)\vee \argument]- Q^{\pi(t)}(\lambda(t))\vee \argument$ 
increases the filtration and since $\pi(t),\lambda(t)\in \mathcal{F}^1\widehat{L[t]}$. 
Similarly, we see that has $\Phi_t^1 \in \widehat{\Hom (\cc{\Sym}(L[1]),L)^1[t]}$: 
With the filtration $\mathfrak{F}^\bullet$ of the convolution algebra 
from \eqref{eq:FiltrationConvLieAlg2} we have 
\begin{align*}
  \frac{\D}{\D t} (\Phi_t)^1_i 
	&\hspace{-0.1cm} \equiv
	\sum_{\ell=0}^\infty \frac{1}{\ell !}\sum_{k=1}^i Q^1_{k+\ell +1}(
	\pi(t)^{\vee \ell}\vee\lambda(t) \vee (\Phi_t)^k_i(\argument))
	\;\,\mathrm{mod}\;\, \mathfrak{F}^n   \\
	& \equiv
	\sum_{\ell=0}^{n-1} \frac{1}{\ell !}\sum_{k=1}^i Q^1_{k+\ell +1}(
	(\pi(t)\;\,\mathrm{mod}\;\, \mathcal{F}^{n-1})^{\vee \ell}\vee
	(	\lambda(t)\;\,\mathrm{mod}\;\, \mathcal{F}^{n}) \vee (\Phi_t)^k_i \;\,\mathrm{mod}\;\, 
	\mathfrak{F}^{n-1} )	\;\,\mathrm{mod}\;\, \mathfrak{F}^n
\end{align*}
and thus $(\Phi_t)^1_i \;\,\mathrm{mod}\;\, \mathfrak{F}^n \in L[t]$ by induction on $i$ and $n$.
\end{proof}

Thus $\Phi_t^1$ induces for all evaluations of $t$ a coalgebra morphism $\Phi_t$ and 
we have
\begin{equation}
  \frac{\D}{\D t} \Phi_t(a)
	= 
	\left([Q^{\pi(t)}, \lambda(t)\vee \argument] - Q^{\pi(t)}(\lambda(t))\vee \argument
	\right) (\Phi_t)(a),
	\quad \quad
	\Phi_0(a)
	=
	a
\end{equation}
since $\frac{\D}{\D t}$ and $([Q^{\pi(t)}, \lambda(t)\vee \argument] - 
Q^{\pi(t)}(\lambda(t))\vee \argument)$ are coderivations with respect to 
$\Delta_\sh$ vanishing on $1$, i.e. also with respect to $\cc{\Delta_\sh}$. 
But one can even show that $\Phi_t$ is an $L_\infty$-morphism, i.e. compatible with the 
codifferentials: 

\begin{lemma}
  \label{lemma:PhitLinftyMorph}
  One has
	\begin{equation}
	  \Phi_t \circ Q^{\pi_0}
		=
		Q^{\pi(t)} \circ \Phi_t.
	\end{equation}
	In particular, $\Phi_1$ induces an $L_\infty$-isomorphism from $(L,Q^{\pi_0})$ 
	to $(L,Q^{\pi_1})$.
\end{lemma}
\begin{proof}
  We compute for $a \in \Sym(L[1])$
	\begin{align*}
	  \frac{\D}{\D t} ( Q^{\pi(t)} \circ \Phi_t(a))
		& = 
		\left[Q^{\pi(t)}, \frac{\D}{\D t}\pi(t) \vee \argument\right] \Phi_t(a) 
		 +
		Q^{\pi(t)} \circ 	\left([Q^{\pi(t)}, \lambda(t)\vee \argument] - 
		Q^{\pi(t)}(\lambda(t))\vee \argument\right) \circ \Phi_t(a) \\
		& =
		\left([Q^{\pi(t)}, \lambda(t)\vee \argument] - 
		Q^{\pi(t)}(\lambda(t))\vee \argument\right)\circ Q^{\pi(t)} \circ \Phi_t(a).
	\end{align*}
	Thus $Q^{\pi(t)} \circ \Phi_t(a)$ is at $t=0$ just $Q^{\pi_0}(a)$ and satisfies 
	the differential equation \eqref{eq:ODEforPhit}. Since the solution is unique, 
	it follows $\Phi_t \circ Q^{\pi_0}(a)	=	Q^{\pi(t)} \circ \Phi_t(a)$.

	In order to show that $\Phi_1$ is an $L_\infty$-isomorphism it suffices by 
	Proposition~\ref{prop:Linftyiso} to show that $(\Phi_1)^1_1$ is an isomorphism.
	But this is clear since $(\Phi_t)^1_1 - \id \equiv 0 \;\,\mathrm{mod}\;\, \mathfrak{F}^2$ and 
	the completeness of the filtration. Moreover, by the construction of the 
	$L_\infty$-inverse we even have $((\Phi_t)^{-1})^1 \in 
	\widehat{\Hom(\cc{\Sym}(L[1]),L)^1[t]}$.
\end{proof}

\begin{example}
  If $(L,Q)$ is just a DGLA $(\liealg{g},\D,[\argument{,}\argument])$, 
	then \eqref{eq:ODEforPhit} simplifies to 
	\begin{equation*}
	  \frac{\D}{\D t} (\Phi_t)^1_1(x)
		=
		[\lambda(t),(\Phi_t)^1_1(x)],
		\quad \quad 
		(\Phi_0)^1_1(x)
		=
		x
		\quad \quad \forall \, x \in \liealg{g}.
	\end{equation*}
	For $\lambda(t)=g$ this has the solution $(\Phi_t)^1_1(x) = e^{[tg,\argument]}x$ and 
	$(\Phi_t)^1_n=0$ for $n\neq 1$, i.e. we recover the setting of 
	Lemma~\ref{lemma:TwistedDGLAsIsomorphic} and 
	Corollary~\ref{cor:GaugeEquivMCTwistsQuis}.
\end{example}

Finally, we can use this $\Phi_t$ in order to generalize 
Proposition~\ref{prop:TwistMorphHomEqu} to $L_\infty$-algebras.

\begin{proposition}
\label{prop:TwistedLinftyIsom}
  Let $(L,Q)$ be a flat $L_\infty$-algebra with complete descending filtration and let 
	$\pi\in \mathcal{F}^1L^1$ be a Maurer-Cartan element that is homotopy equivalent 
	to $0$ via 
	$\pi(t),\lambda(t)$. Moreover, let $F\colon (L,Q)\rightarrow (L',Q')$ be an 
	$L_\infty$-morphism. Then the $L_\infty$-morphisms $F$ and $(\Phi_1')^{-1}\circ F^{\pi} 
	\circ \Phi_1$ are homotopic.
\end{proposition}
\begin{proof}
The candidate for the homotopy is $F(t)=(\Phi'_t)^{-1} \circ F^{\pi(t)} \circ \Phi_t$. 
In the proof of Lemma~\ref{lemma:PhitLinftyMorph} we saw that 
$((\Phi_t')^{-1})^1 \in \widehat{\Hom(\cc{\Sym}(L'[1]),L')^1[t]}$. 
Thus it directly follows that $F^1(t)$ is indeed in the completion 
$\widehat{\Hom(\cc{\Sym}(L[1]),L')^1[t]}$. We compute
\begin{align*}
  \frac{\D}{\D t} F(t)
	& =
	- (\Phi'_t)^{-1} \circ \left(([Q^{\pi'(t)}, \lambda'(t)\vee \argument]- 
	Q^{\pi'(t)}(\lambda'(t))\vee \argument \right) \circ F^{\pi(t)} \circ \Phi_t \\
	& \quad + 
	(\Phi'_t)^{-1}  \circ F^{\pi(t)} \circ \left(([Q^{\pi(t)}, \lambda(t)\vee \argument]- 
	Q^{\pi(t)}(\lambda(t))\vee \argument \right)\circ \Phi_t 
	+
	(\Phi'_t)^{-1} \circ \frac{\D}{\D t} F^{\pi(t)} \circ \Phi_t.
\end{align*}
With $\frac{\D}{\D t} F^{\pi(t)}  = F^{\pi(t)} \circ (\frac{\D}{\D t}\pi(t) \vee \argument)
- (\frac{\D}{\D t}\pi'(t) \vee \argument) \circ F^{\pi(t)}$ we get
\begin{align*}
  \frac{\D}{\D t} F(t)
	& =
	Q' \circ \left((\Phi'_t)^{-1} \circ (- \lambda'(t)\vee\argument) \circ F^{\pi(t)}
	\circ \Phi_t + (\Phi'_t)^{-1} \circ F^{\pi(t)}\circ ( \lambda(t)\vee\argument)
	\circ \Phi_t \right)  \\
	& \quad + 
	\left((\Phi'_t)^{-1} \circ (- \lambda'(t)\vee\argument) \circ F^{\pi(t)}
	\circ \Phi_t + (\Phi'_t)^{-1} \circ F^{\pi(t)}\circ ( \lambda(t)\vee\argument)
	\circ \Phi_t \right) \circ Q.
\end{align*}
Projecting to $L[1]$ yields indeed
\begin{equation*}
  \frac{\D}{\D t} F^1(t)
	=
	\widehat{Q}^1(\lambda_F(t) \vee\exp F^1(t))
\end{equation*}
for $\lambda_F(t) = \pr_{L[1]}\circ ((\Phi'_t)^{-1} \circ (- \lambda'(t)\vee\argument) 
\circ F^{\pi(t)}\circ \Phi_t + (\Phi'_t)^{-1} \circ F^{\pi(t)}\circ 
( \lambda(t)\vee\argument)	\circ \Phi_t )$. 
Note that this is true since the term in the bracket is a coderivation with respect to 
$\Delta_\sh$ along $F(t)$ that vanishes on $1$, therefore also a coderivation along 
$\cc{\Delta_\sh}$. Moreover, $\lambda_F$ is indeed in the completion 
$\widehat{\Hom(\cc{\Sym}(L[1]),L')^0[t]}$ since $F^1(t)$ is.
\end{proof}

\subsection{Homotopy Theory of Curved $L_\infty$-Algebras}
\label{sec:HomTheoryofCurvedLinfty}

As mentioned above, we want to generalize now 
the above homotopy theory of flat $L_\infty$-algebras to the curved setting. 
Therefore, we want to interpret $L_\infty$-morphisms again as Maurer-Cartan 
elements, as we did in the flat case from \eqref{prop:ConvLinftyStructure}.
From Remark~\ref{rem:CurvedMorphLinfty} we recall the 
following result: 

\begin{proposition}
\label{prop:CurvedConvLinftyStructure}
  Let $(L,Q)$ and $(L',Q')$ be (curved) $L_\infty$-algebras with 
	complete descending filtrations such that one has $Q_0^1 \in 
	\mathcal{F}^1 L$ and $Q_0'\in \mathcal{F}^1 L'$ for the curvatures. 
	Then the coalgebra $\Sym(\Hom(\Sym(L[1]),L')[1])$ can be equipped with 
	a codifferential $\widehat{Q}$ with structure maps
	\begin{equation}
	\label{eq:CurvonConvLinftyCurved}
	  \widehat{Q}_0^1 
		=
		(1 
		\longmapsto 
		Q'_0(1) )
		\in 
		(\Hom(\mathbb{K},L')[1])^1,
	\end{equation}
  \begin{equation}
	  \label{eq:DiffonConvLinftyCurved}
	  \widehat{Q}^1_1 F
		=
		Q'^1_1 \circ F - (-1)^{\abs{F}} F \circ Q
	\end{equation}
	and 
	\begin{equation}
		\label{eq:BracketonConvLinftyCurved}
		\widehat{Q}^1_n(F_1\vee \cdots \vee F_n)
		=
		(Q')^1_n \circ 
		(F_1\star F_2\star \cdots \star F_n),
	\end{equation}
	where $\abs{F}$ denotes the degree in 
	$\Hom(\Sym(L[1]),L')[1]$. Moreover, \eqref{eq:FiltrationConvLieAlg2} 
	generalizes to a complete descending filtration 
	$\mathfrak{F} \Hom(\Sym(L[1]),L')[1]$. 
  The curved $L_\infty$-algebra $(\Hom(\Sym(L[1]),L'), \widehat{Q})$ 
	is called \emph{convolution $L_\infty$-algebra}. 
\end{proposition}
\begin{proof}
The fact that this defines an $L_\infty$-structure follows as in 
Proposition~\ref{prop:ConvLinftyStructure}, the fact that we 
get a complete descending filtration follows as in 
Proposition~\ref{prop:CompleteFiltrationConvLinfty}.
\end{proof}

From now on we always assume that our curved $L_\infty$-algebras $(L,Q)$ 
have a complete descending filtration and that $Q_0^1 \in \mathcal{F}^1 L$.
In this case, the above proposition immediately leads us to the following definition of curved 
$L_\infty$-morphisms and their homotopy equivalence relation, 
generalizing the observations for curved Lie algebras from 
Remark~\ref{rem:curvedMorphisms}:

\begin{definition}[Curved $L_\infty$-morphism]
  Let $(L,Q)$ and $(L',Q')$ be (curved) $L_\infty$-algebras with 
	complete descending filtrations such that one has $Q_0^1 \in 
	\mathcal{F}^1 L$ and $Q_0'\in \mathcal{F}^1 L'$ for the curvatures. 
  The Maurer-Cartan elements in $\mathfrak{F}^1\Hom(\Sym(L[1]),L')$ are 
	called \emph{curved morphisms between $(L,Q)$ and $(L',Q')$} or 
	\emph{curved $L_\infty$-morphisms}. Two curved $L_\infty$-morphisms 
	are called \emph{homotopic} if they are homotopy equivalent Maurer-Cartan 
	elements.
\end{definition}
We write for a curved $L_\infty$-morphism 
$F^1\colon (L,Q) \rightsquigarrow (L',Q')$ and $F^1\sim F'^1$ if $F^1$ and 
$F'^1$ are homotopic. 

\begin{remark}[Difficulties]
\label{rem:Difficulties}
Note that this generalization to the curved setting yields two main difficulties compared to the flat case:
\begin{itemize}
 	\item As mentioned above, in the case of curved $L_\infty$-algebras 
	      $(L,Q)$ the first structure map $Q_1^1$ does in general not square 
				to zero. Thus we do not have the notion of an 
				$L_\infty$-quasi-isomorphism as in Definition~\ref{def:Linftyquis} 
				and we can not use Proposition~\ref{prop:StandardFormLinfty}, 
				i.e. that every flat $L_\infty$-algebra is isomorphic to the 
				direct sum of a minimal one and a linear contractible one.
	      Using the homotopy classification of flat $L_\infty$-algebras 
				from Lemma~\ref{lemma:HomEquvsQuasiIso} we propose below a notion 
				of curved quasi-isomorphisms.			 
	\item Curved morphisms of $L_\infty$-algebras 
				$F^1 \in \mathfrak{F}^1(\Hom(\Sym(L[1]),L')[1])^0$ are given by 
				Taylor components $F_n^1 \colon \Sym^n (L[1]) 
				\rightarrow L'[1]$ for $n\geq 0$, where one has in particular a 
				zero component $F_0^1= F_0^1(1) = \alpha \in \mathcal{F}^1L'^1$. 
				By definition, they satisfy the Maurer-Cartan equation  
				\begin{align}
	        \label{eq:curvedmorphLinftyII}
      	  0
		      =
		      \widehat{Q}_0^1 + Q'^1_1 \circ F^1 - F^1\circ Q + 
		      \sum_{n=2}^\infty \frac{1}{n!} Q'^1_n \circ (F^1)^{\star n}.
	      \end{align}
				If $F^1_0(1)=0$ the evaluation at $1$ yields 
				$Q'_0(1) = F^1_1(Q_0(1))$ and thus $F$ is a $L_\infty$-morphism 
				of curved $L_\infty$-algebras	in the usual sense, from now on 
				called \emph{strict} $L_\infty$-morphism. 
				However, for $F_0^1(1)=\alpha\neq 0$ we no longer get an 
				induced coalgebra morphism on the symmetric coalgebras, compare 
				Theorem~\ref{thm:CofreeCocomConilpotentCoalg}. 
\end{itemize}
\end{remark}

The second point is in fact no big problem as we explain now: 
At first, note that we can still extend $F^1$ to all symmetric orders 
via $F^i = \frac{1}{i!} (F^1)^{\star i}$ and $F^0 = \pr_\mathbb{K}$. 
Writing $F = \exp_\star F^1$ we get a coalgebra morphism
\begin{equation}
\label{eq:relationFFtilde}
  F \colon
  \Sym (L[1])
  \longrightarrow 
  \widehat{\Sym}(L'[1]),
  \quad 
   X 
   \longmapsto 
   F(X)
   =
   \exp \alpha \vee \widetilde{F}(X)
\end{equation}
into the completed symmetric coalgebra, where we complete 
with respect to the induced filtration. Here 
$\widetilde{F}\colon \Sym(L[1]) \rightarrow \Sym(L'[1])$ is 
the extension of $F^1_n$ with $n\geq 1$ to a coalgebra morphism, 
i.e. in particular $\widetilde{F}(1)=1$. We sometimes identify $F^1$ 
with its extension $F$. Note that since $Q'$ is compatible with 
the filtration, it extends to the completion $\widehat{\Sym}(L'[1])$ 
and \eqref{eq:curvedmorphLinftyII} implies 
$F^1 \circ Q = Q'^1 \circ F$ as expected.

Thus we see that curved $L_\infty$-morphisms $F^1$ are still 
well-behaved and we collect some useful properties:

\begin{proposition}
\label{prop:ProvertiesofCurvedLinftyMorph}
  Let $(L,Q)$ and $(L',Q')$ be (curved) $L_\infty$-algebras with 
	complete descending filtrations such that one has $Q_0^1 \in 
	\mathcal{F}^1 L$ and $Q_0'\in \mathcal{F}^1 L'$ for the curvatures. 
	Moreover, let $F^1 \in \mathfrak{F}^1(\Hom(\Sym(L[1]),L')[1])^0$ be a 
	curved $L_\infty$-morphism. 
	\begin{propositionlist}
	  \item Extending $F^1_n$ with $n\geq 1$ to a coalgebra morphism 
		      $\widetilde{F}$, one gets for all $X\in \Sym(L[1])$
		      \begin{equation}
					  \widetilde{F}^1 \circ Q (X)
            =
            Q'^1 \circ \left(\exp (\alpha)\vee \widetilde{F}(X)\right),
          \end{equation}
					i.e. $\widetilde{F}$ is a strict morphism into the 
					twisted $L_\infty$-algebra $(L,Q'^{\alpha})$.
		\item Conversely, every strict morphism $\widetilde{F} \colon 
		      (L,Q)\rightarrow (L',Q'^\alpha)$ corresponds to a curved 
					$L_\infty$-morphism $F^1$ with $F^1_0(1) = \alpha$ and 
					$F^1_i = \widetilde{F}^1_i$ for $i>0$.
	  \item $F^1$ induces a map at the level of Maurer-Cartan elements.
		      If $\pi \in \mathcal{F}^1L^1$ is a curved Maurer-Cartan element 
					in $(L,Q)$, then
					\begin{align}
					\label{eq:CurvdMorphMC}
					  F_\MC(\pi)
						=
						F^1(\exp \pi)
						=
            \alpha + 
						\sum_{k=1}^\infty \frac{1}{k!} F^1_k(\pi \vee \cdots \vee \pi)
          \end{align}
					is a curved Maurer-Cartan element in $(L',Q')$.
		\item The map from \eqref{eq:CurvdMorphMC} is compatible with homotopy 
		      equivalences, i.e. $F^1$ induces a map $F_\MC \colon 
					\Def(L)\rightarrow 	\Def(L')$. 
		\item The case $F^1_1=\id$, $F_0^1=\alpha$ and $F_n^1=0$ for 
		      all $n \geq 2$ corresponds again to twisting by $-\alpha$. We denote 
					it by 
					\begin{equation}
						\label{eq:TwCurvedMorph}
						\tw_{\alpha}\colon 
						(L,Q) 
						\rightsquigarrow 
						(L,Q^{-\alpha}).
					\end{equation}
		\item The curved $L_\infty$-morphisms from the curved $L_\infty$-algebra 
		      $0$ to $L$ are just the set of curved Maurer-Cartan elements of $L$.
	\end{propositionlist}
\end{proposition}
\begin{proof}
The first and second point follow directly from the Maurer-Cartan equation 
\eqref{eq:curvedmorphLinftyII} and \eqref{eq:relationFFtilde}.
Then the third and fourth point follow from Lemma~\ref{lemma:CorrespCurvedandFlatMC} since $\widetilde{F}$ 
is compatible with equivalences by Proposition~\ref{prop:FmapsEquivMCtoEquiv}.
The other statements are clear. 
\end{proof}

\begin{corollary}
  Curved $L_\infty$-morphisms $F^1$ from $(L,Q)$ to $(L',Q')$ are in 
	one-to-one correspondence with strict $L_\infty$-morphisms $\widetilde{F}$ 
	from $(L,Q)$ into $(L',Q'^\alpha)$ with $F^1_0(1)=\alpha \in \mathcal{F}^1
	L'^1$ and $F^1_i =\widetilde{F}^1_i$ for $i>0$.
\end{corollary}

As expected, we can compose curved $L_\infty$-morphisms between curved 
$L_\infty$-algebras:

\begin{proposition}
\label{prop:CompofCurved}
  Let $F^1 \colon (L,Q) \rightsquigarrow (L',Q')$ and $G^1\colon 
	(L',Q') \rightsquigarrow (L'',Q'')$ be curved $L_\infty$-morphisms 
	with $F_0^1 = \alpha \in \mathcal{F}^1L'^1$ and 
	$G_0^1 = \beta \in \mathcal{F}^1L''^1$. Then there exists a 
	curved $L_\infty$-morphism 
	\begin{equation}
		\label{eq:CompCurvedLinftyMorph}
		(G\circ F)^1 
		\coloneqq
		G^1 \circ F
	  \colon
		(L,Q) 
		\longrightsquigarrow{1} 
		(L'',Q'')
	\end{equation}
	with $(G\circ F)_0^1 = G^1(\exp \alpha) = 
	\beta + \widetilde{G}^1(\cc{\exp}\alpha)$ 
	and $\widetilde{G\circ F} = \widetilde{G}^\alpha \circ \widetilde{F}$, 
	and the composition is associative. Moreover, one has
	\begin{equation}
		\label{eq:ComponMC}
		(G\circ F)_\MC 
		=
		G_\MC \circ F_\MC
		\colon 
		\Def(L)
		\longrightarrow
		\Def(L'').
	\end{equation}
	at the level of equivalence classes of Maurer-Cartan elements.
\end{proposition}
\begin{proof}
We saw in Proposition~\ref{prop:ProvertiesofCurvedLinftyMorph} that 
$F$ and $G$ correspond to strict $L_\infty$-morphisms
\begin{equation*}
  \widetilde{F}
	\colon
	(L,Q)
	\longrightarrow
	(L',Q'^\alpha)
	\quad \text{ and }  \quad 
	\widetilde{G}
	\colon
	(L',Q')
	\longrightarrow
	(L'',(Q'')^\beta).
\end{equation*}
In particular, we can twist $\widetilde{G}$ with 
$\alpha$ and obtain
\begin{equation*}
  (L,Q)
	\stackrel{\widetilde{F}}{\longrightarrow}
	(L',Q'^\alpha)
	\stackrel{\widetilde{G}^\alpha}{\longrightarrow}
	(L'',(Q'')^{\beta + \widetilde{G}^1(\cc{\exp}\alpha)}).
\end{equation*}
Thus $(G\circ F)^1_i = (\widetilde{G}^\alpha \circ \widetilde{F})^1_i$ 
for $i>0$ and $(G\circ F)^1_0 = \beta + \widetilde{G}^1(\cc{\exp}\alpha)$ 
defines indeed a curved $L_\infty$-morphism. Moreover, we have 
$G^1 \circ F = G^1 \circ \exp_\star F^1 = G^1 \circ \exp(\alpha) \vee \widetilde{F} = (G\circ F)^1$. It is easy to check that the composition 
is associative. Therefore, let us now look at the induced map at the level 
of Maurer-Cartan elements: by \eqref{eq:CurvdMorphMC} $(G\circ F)_\MC$ maps 
$\pi \in \mathcal{F}^1L^1$ to 
\begin{align*}
  (G\circ F)_\MC (\pi)
	& =
  \beta + \widetilde{G}^1(\cc{\exp}\alpha) 
	+ \widetilde{G\circ F}^1(\cc{\exp}\pi)
	=
	\beta + \widetilde{G}^1(\cc{\exp}\alpha) + 
	\widetilde{G}^1 ( \exp \alpha \vee \widetilde{F}(\cc{\exp}\pi)) \\
	& = 
	G^1	(\exp \alpha \vee \exp \widetilde{F}^1(\cc{\exp}\pi)) 
	=
	G_\MC \circ F_\MC(\pi)
\end{align*}
as desired.
\end{proof}
We see that if both $F,G$ are strict morphisms, then $F=\widetilde{F}$ and $G=\widetilde{G}$ and the above composition is just the usual one. 
Moreover, we have the following observation: 

\begin{corollary}
\label{cor:CompwithTwist}
  Let $F^1 \colon (L,Q) \rightsquigarrow (L',Q')$ be a curved 
	$L_\infty$-morphism with $F_0^1=\alpha \in \mathcal{F}^1L'^1$. 
	Then one has for $\beta \in \mathcal{F}^1L'^1$
	\begin{equation}
	  \label{eq:curvLinftyFlatTwisted}
		\tw_\alpha \circ \tw_\beta
		=
		\tw_{\alpha+\beta},
		\quad \quad 
		F^1 
		=
		\tw_\alpha \circ \widetilde{F},
		\quad \text{ and }\quad
		\tw_{-\alpha} \circ  F 
		=
		\widetilde{F}^1
	\end{equation}
\end{corollary}
\begin{proof}
By the fifth point of Proposition~\ref{prop:ProvertiesofCurvedLinftyMorph} 
we know that $\tw_\alpha$ corresponds to twisting with $-\alpha$, i.e.
\begin{equation*}
  \tw_\alpha \colon 
	(L',Q'^\alpha)
	\longrightsquigarrow{0.8}
	( L',Q'^{\alpha-\alpha})
	=
	(L',Q'), 
	\quad \quad
	(\tw_\alpha)^1_0 = \alpha,
	\quad 
	(\tw_\alpha)^1_1 = \id.
\end{equation*}
We compute $(\tw_\alpha \circ \tw_\beta)^1_0 = \alpha + \beta$ 
and $(\tw_\alpha \circ \tw_\beta)^1_i = \delta^1_i \id$ for $i\geq 0$, 
thus $\tw_\alpha \circ \tw_\beta= \tw_{\alpha + \beta}$ is show. 
Then Proposition~\ref{prop:CompofCurved} gives the desired 
$F^1 = \tw_\alpha \circ \widetilde{F}$, and the last identity follows 
with the first.
\end{proof}

In a next step, we want to investigate if the (curved) homotopy 
equivalence is compatible with the interpretation of curved morphisms 
as strict $L_\infty$-morphisms into a twisted codomain. At first, let 
$F^1$ and $F'^1$ be homotopic via $F^1(t)$ and $\lambda^1(t)$, i.e. 
\begin{align}
\label{eq:HomEquCurvMorph}
   \frac{\D}{\D t} F^1(t)
	& =
	\widehat{Q}^1(\lambda^1(t)\vee \exp(F^1(t)))
  =
	\widehat{Q}^1(\lambda^1(t)\vee \exp(F^1_0(t)(1))\vee 
	\exp(\widetilde{F}^1(t)))
\end{align}
with $F^1(0) = F^1$ and $F^1(1)= F'^1$.
This implies for the component $F^1_0(t)=F^1_0(t)(1)$:
\begin{equation*}
  \frac{\D}{\D t} F^1_0(t)
	=
	\lambda^1 (t) \circ Q_0^1
	+
	Q'^1 \circ (\lambda^1_0(t) \vee \exp F^1_0(t)).
\end{equation*}
Thus we see that $\alpha_0 \neq \alpha_1$ is possible, and 
we can in general not directly compare $\widetilde{F}_0$ and 
$\widetilde{F}_1$ since they can have different codomains. However, we 
directly see:

\begin{proposition}
\label{prop:RelationHomEquiv}
  Let $F^1,F'^1 \colon (L,Q)\rightsquigarrow (L',Q')$ be two curved 
	$L_\infty$-morphisms with $\alpha = F_0^1=F'^1_0$. 
	Then $F^1\sim F'^1$ if and only if $\widetilde{F}^1\sim \widetilde{F'}^1$.
\end{proposition}
\begin{proof}
By Corollary~\ref{cor:CompwithTwist} we only have to show that 
$F^1 \sim F'^1$ implies $(\tw_\beta \circ F) \sim (\tw_\beta \circ F')$ 
for all $\beta \in \mathcal{F}^1L'^1$ and all curved $L_\infty$-mophisms 
with $F_0^1=F'^1_0$, where 
\begin{equation*}
  (\tw_\beta \circ F), (\tw_\beta \circ F') \colon 
	(L,Q)
	\longrightsquigarrow{1}
	(L',Q'^{-\beta}).
\end{equation*}
Thus assume that $F^1(t),\lambda^1(t)$ 
encode the equivalence between $F^1$ and $F'^1$, then we have by  
\eqref{eq:HomEquCurvMorph}
\begin{align*}
   \frac{\D}{\D t} F^1(t)
	& =
	\widehat{Q}^1(\lambda^1(t)\vee \exp(F^1(t)))
	\quad \text{ with }
	F^1(0) = F^1
	\text{ and }
	F^1(1)= F'^1.
\end{align*}
With $G^1(t) = \tw_\beta \circ F^1(t)$ we get 
\begin{align*}
  \frac{\D}{\D t} G^1(t)
	& =
  \frac{\D}{\D t} F^1(t)
	=
	\widehat{Q}^1(\lambda^1(t)\vee \exp(F^1(t))) \\
	& = 
	\widehat{Q}^1(\lambda^1(t)\vee \exp(-\beta)\vee \exp(\beta)
	\vee \exp(F^1(t))) \\
	&	=
	\widehat{Q^{-\beta}}^1(\lambda^1(t) \vee\exp(G^1(t))) ,
\end{align*}
where $\widehat{Q^{-\beta}}$ is the codifferential on the curved 
convolution $L_\infty$-algebra induced by $(L,Q)$ and $(L',Q'^{-\beta})$. 
Thus we get the homotopy equivalence between $\tw_\beta \circ F$ 
and $\tw_\beta \circ F'$.
\end{proof}

Moreover, \eqref{eq:HomEquCurvMorph} implies:
\begin{corollary}
  Let $F^1$ and $F'^1$ be homotopic curved $L_\infty$-morphisms from 
	$(L,Q)$ to $(L',Q')$ with $\alpha_0 = F^1_0$ and 
	$\alpha' =F'^1_0$. If $Q_0^1=0$ and if $\alpha$ is a 
	Maurer-Cartan element, then so is $\alpha'$ and both are equivalent.	
\end{corollary}

Analogously to the flat case in Proposition~\ref{prop:PropertiesofHomotopicMorphisms} we can show that homotopic 
curved $L_\infty$-morphisms induce the same maps on the equivalence classes 
of Maurer-Cartan elements.

\begin{proposition}
\label{prop:HomCurvedMorphonDef}
  Let $F^1,F'^1 \colon (L,Q)\rightsquigarrow (L',Q')$ be two curved 
	$L_\infty$-morphisms. If $F^1$ and $F'^1$ are homotopic, then they 
	induce the same maps from $\Def(L)$ to $\Def(L')$, i.e. 
	$F_\MC = F'_\MC$.
\end{proposition}
\begin{proof}
Let us assume that $F^1(t),\lambda^1(t)$ 
encode the equivalence between $F^1$ and $F'^1$, then we have by  
\eqref{eq:HomEquCurvMorph}
\begin{align*}
   \frac{\D}{\D t} F^1(t)
	& =
	\widehat{Q}^1(\lambda^1(t)\vee \exp(F^1(t)))
	\quad \text{ with }
	F^1(0) = F^1
	\text{ and }
	F^1(1)= F'^1.
\end{align*}
Applying this to $\exp( \pi)$ for $\pi \in \Mc^1(L)$ gives
\begin{align*}
   \frac{\D}{\D t} F^1(t)\exp( \pi)
	& =
	\widehat{Q}^1(\lambda^1(t)\exp( \pi)\vee \exp(F^1(t)\exp( \pi))),
\end{align*}
i.e. $\pi(t) = F^1(t)(\exp\pi) $ and $\lambda(t) = \lambda^1(t)(\exp\pi) $ 
encode the homotopy equivalence between 
$F_\MC(\pi)=F^1(\exp\pi) $ and $F'_\MC(\pi')=(F')^1(\exp\pi)$. Thus $F^1$ and $F'^1$ map indeed 
Maurer-Cartan elements to equivalent ones.
\end{proof}

In the proof of Proposition~\ref{prop:RelationHomEquiv} we saw that 
$\tw_\alpha \circ \widetilde{F}$ is homotopic to 
$\tw_\alpha \circ \widetilde{F'}$ if $\widetilde{F}$ and 
$\widetilde{F'}$ are homotopic. We want to generalize this in 
the spirit of Proposition~\ref{prop:CompofHomotopicHomotopic} and 
Proposition~\ref{prop:preCompofHomotopicHomotopic} and show that general
pre- and post-compositions of homotopic curved $L_\infty$-morphisms with a 
curved $L_\infty$-morphism are again homotopic.

\begin{proposition}
\label{prop:CompofCurvedHomotopicHomotopic}
   Let $F^1,F'^1 \colon (L,Q)\rightsquigarrow (L',Q')$ be two curved 
	$L_\infty$-morphisms and assume $F^1\sim F'^1$.
	\begin{propositionlist}
	  \item If $H^1 \colon (L',Q') \rightsquigarrow (L'',Q'')$ 
		      is a curved $L_\infty$-morphism, then $(H^1\circ F) \sim (H^1\circ F')$.
		\item  If $H^1 \colon (L'',Q'') \rightsquigarrow (L,Q)$ 
		      is a curved $L_\infty$-morphism, then $ (F^1 \circ H)\sim (F'^1\circ H)$.
	\end{propositionlist}
\end{proposition}
\begin{proof}
The statements follow as in the flat case in Proposition~\ref{prop:CompofHomotopicHomotopic} and 
Proposition~\ref{prop:preCompofHomotopicHomotopic} since we used there 
only bialgebraic properties that still hold in our setting by the completeness of the filtration.
\end{proof}

\begin{corollary}
  Let $F^1,F'^1$ be two homotopic curved $L_\infty$-morphisms 
	from $(L,Q)$ to $(L',Q')$, and let 
	$H^1,H'^1$ be two homotopic curved $L_\infty$-morphisms from $(L',Q')$ to 
	$(L'',Q'')$, then $(H^1 \circ F)\sim (H'^1 \circ F')$. 
\end{corollary}

Let us now address the first difficulty from Remark~\ref{rem:Difficulties}, 
namely the fact that we do not have an obvious notion of 
curved $L_\infty$-quasi-isomorphism since curved $L_\infty$-algebras 
$(L,Q)$ do not induce a cochain complex $(L,Q^1_1)$. 
Recall that we showed for the flat case in 
Lemma~\ref{lemma:HomEquvsQuasiIso} that there exists an 
$L_\infty$-quasi-isomorphism between two flat $L_\infty$-algebras 
$(L,Q)$ and $(L',Q')$ if and only if there are $L_\infty$-morphisms 
$F\colon (L,Q)\to (L',Q')$ and $G\colon (L',Q')\to (L,Q)$ 
such that $F\circ G\sim \id_{L'}$ and $G\circ F\sim \id_{L}$. 
Thus we define:

\begin{definition}[Curved $L_\infty$-quasi-isomorphism]
\label{def:CurvedQuis}
Let $F^1 \colon (L,Q) \rightsquigarrow (L',Q')$ be a curved 
$L_\infty$-morphism between curved $L_\infty$-algebras. One 
calls $F$ \emph{curved $L_\infty$-quasi-isomorphism} if and only if 
there exists a curved $L_\infty$-morphism $G^1\colon (L',Q')
\rightsquigarrow (L,Q)$ 
such that $F^1\circ G\sim \id_{L'}$ and $G^1\circ F\sim \id_{L}$. In 
this case, $F^1$ and $G^1$ are said to be quasi-inverse to each other.
\end{definition}

Concerning the behaviour on Maurer-Cartan elements we directly get 
analogue of Theorem~\ref{thm:lquisbijondef}.

\begin{corollary}
  Let $F^1 \colon (L,Q) \rightsquigarrow (L',Q')$ be a 
	curved $L_\infty$-quasi-isomorphism. Then the induced map 
	$F_\MC \colon\Def(L)\rightarrow \Def(L')$ is a bijection.
\end{corollary}
\begin{proof}
Let $G^1$ be the quasi-inverse of $F^1$, then we know by 
Proposition~\ref{prop:HomCurvedMorphonDef} that 
$(F\circ G)_\MC$ and $(G \circ F)_\MC$ are the identity maps 
on $\Def(L')$ resp. $\Def(L)$. But 
Proposition~\ref{prop:CompofCurved} implies that 
$(F\circ G)_\MC = F_\MC \circ G_\MC$ and 
$(G\circ F)_\MC = G_\MC \circ F_\MC$, which implies 
that $F_\MC$ and $G_\MC$ are bijections.
\end{proof}

We can also generalize the twisting procedure for 
$L_\infty$-morphisms from the strict case in 
Proposition~\ref{prop:twistinglinftymorphisms} to our curved 
$L_\infty$-morphisms, which is straightforward:

\begin{proposition}
\label{prop:TwistingofCurvedMorph}
  Let $F^1 \colon (L,Q)\rightsquigarrow (L',Q')$ be a curved 
	$L_\infty$-morphism with $F_0^1=\alpha$ and let 
	$\pi \in \mathcal{F}^1L^1$. Then the curved $L_\infty$-morphism
	\begin{equation}
	  (F^\pi)^1 
		=
		\tw_{- \widetilde{F}^1(\exp \pi)} \circ 
		F \circ \tw_{\pi} \colon 
		(L,Q^\pi)
		\longrightsquigarrow{1}
		(L',Q'^{\widetilde{F}^1(\exp \pi)})
	\end{equation}
	has structure maps $(F^\pi)^1_i = \sum_{k=0}^\infty \frac{1}{k!} F^1_{k+i}(
	\pi^{\vee k}\vee \argument )$ for $i>0$ and $(F^\pi)^1_0=\alpha$.
\end{proposition}
\begin{proof}
At first, it is clear that the composition $(F^\pi)^1 
=	\tw_{- F^1(\exp \pi)} \circ 	F \circ \tw_{\pi}$ is a curved 
$L_\infty$-morphism
\begin{equation*}
  (L,Q^\pi)
	\stackrel{\tw_\pi}{\longrightsquigarrow{1}}
	(L,Q)
	\stackrel{F}{\longrightsquigarrow{1}}
  (L',Q')
	\stackrel{\tw_{- \widetilde{F}^1(\exp \pi)}}{\longrightsquigarrow{1}}
	(L',Q'^{\widetilde{F}^1(\exp\pi)}).
\end{equation*}
For the structure maps, we see that $(F^\pi)^1_0 = -\widetilde{F}^1(\exp \pi) 
+ \alpha + \widetilde{F}^1(\exp \pi) = \alpha$ and for $i>0$ we get 
\begin{equation*}
  (F^\pi)^1_i 
	= 
	(F \circ \tw_\pi)^1_i 
	=
	(\widetilde{F}^\pi)^1_i,
\end{equation*}
which implies in particular 
$\widetilde{F^\pi}= \widetilde{F}^\pi$.
\end{proof}

\begin{remark}
We collect a few immediate observations:
\begin{remarklist}
\item We directly see that $(F^0)^1 = F^1$, i.e. twisting by zero does 
      not change the morphism. 
\item The above definition for the twisted curved $L_\infty$-morphism 
      $(F^\pi)^1$ recovers the results for the strict case from 
			Proposition~\ref{prop:twistinglinftymorphisms}: If $F^1$ is strict, 
			i.e. $F^1_0=0$ and $F = \widetilde{F}$, then we get indeed $F^\pi 
			= \widetilde{F}^\pi= \widetilde{F^\pi}$. 
\item If $F^1_0 \neq 0 $ and if we twist with a Maurer-Cartan element, 
      then the image does no longer have to be a flat $L_\infty$-algebra, 
			since $\widetilde{F}^1(\exp\pi)$ does not have 
			to be a Maurer-Cartan element. 
\item The twist of a curved $L_\infty$-morphism is always a curved 
      $L_\infty$-morphism, i.e. we can not obtain a strict one by twisting.
\end{remarklist}
\end{remark}

Finally, note that the only thing we did not generalize yet are the 
results from Section~\ref{sec:HomEquivTwistedMorphisms}, where 
we showed that strict $L_\infty$-morphisms between flat 
$L_\infty$-algebras that are twisted with equivalent Maurer-Cartan 
elements are homotopic. Note that in Proposition~\ref{prop:TwistMorphHomEqu} 
and Proposition~\ref{prop:TwistedLinftyIsom} we considered only the case of 
a Maurer-Cartan element $\pi$ equivalent to zero, which was sufficient for 
the flat case since it implies the statement for all generic 
equivalent Maurer-Cartan elements $\pi$ and $\pi'$.
\begin{itemize}
	\item Let us consider at first the context of strict $L_\infty$-morphisms 
	      of curved $L_\infty$-algebras. Let $\pi \sim \pi'$ be two equivalent 
				curved Maurer-Cartan elements in $(L,Q)$. Then we know from 
				Lemma~\ref{lemma:twistCodiff} that 
				$(L,Q^\pi)$ and $(L,Q^{\pi'})$ are flat $L_\infty$-algebras, 
				and from Lemma~\ref{lemma:CorrespCurvedandFlatMC} that 
				$\pi'-\pi \sim 0$ in $(L,Q^\pi)$. Thus we can directly apply 
				Proposition~\ref{prop:TwistedLinftyIsom}.
	\item In the context of curved $L_\infty$-morphisms it is more 
	      difficult to generalize these results since we saw in 
				Proposition~\ref{prop:TwistingofCurvedMorph} that if we twist 
				a curved $L_\infty$-morphism with a Maurer-Cartan element, then 
				the codomain does not need to be a flat $L_\infty$-algebra. 
				More explicitly, let $F^1 \colon (L,Q) \rightsquigarrow (L',Q')$ 
				be a curved $L_\infty$-morphism and let $\pi,\pi'\in \mathcal{F}^1L^1$ 
				be two equivalent Maurer-Cartan elements. Then we end up with 
				the following diagram:
				\begin{equation}
				\label{eq:CurvedTwistDiagram}
          \begin{tikzpicture}
		       \node (LQpi) at (0,0) {{$(L,Q^\pi)$ }};
	         \node (LprimeFpi) at (6,0)
						{{$(L',Q'^{\widetilde{F}^1(\exp\pi)})$ }};
		       	
		     \node  (LQpiprime) at (0,-2.5)
		     {{$(L,Q^{\pi'})$}};
         \node (LprimeFpiprime) at (6,-2.5)
		     {{$(L',Q'^{\widetilde{F}^1(\exp\pi')})$}};
				\draw [->] (LQpi) -- (LQpiprime)  node[midway,left]
					{$\Phi_1$};
				\draw  [->,line join=round,
           decorate, decoration={
           zigzag,
           segment length=4,
           amplitude=.9,post=lineto,
           post length=2pt
         }] (LprimeFpi) -- (LprimeFpiprime)  node[midway,right]
				{$\tw_\alpha \circ \Phi'_1 \circ \tw_{-\alpha}$};
	       \draw [->,line join=round,
           decorate, decoration={
           zigzag,
           segment length=4,
           amplitude=.9,post=lineto,
           post length=2pt
         }]  (LQpiprime) -- (LprimeFpiprime)  node[midway,above] 
				{$(F^{\pi'})^1$};
		    \draw [->,line join=round,
          decorate, decoration={
          zigzag,
          segment length=4,
          amplitude=.9,post=lineto,
          post length=2pt
         }]  (LQpi) -- (LprimeFpi)  node[midway,above] {$(F^\pi)^1$};
	     \end{tikzpicture} 
	    \end{equation}
	where $\Phi_1$ and $\Phi'_1$ are strict $L_\infty$-isomorphisms 
	by Lemma~\ref{lemma:PhitLinftyMorph}. Note that in particular 
	\begin{equation*}
	  \Phi_1' \colon 
		(L',Q'^{\alpha +\widetilde{F}^1(\exp\pi)})
		\longrightarrow
		(L',Q'^{\alpha +\widetilde{F}^1(\exp\pi')})
	\end{equation*}
	is a well-defined strict $L_\infty$-isomorphism
	since $\alpha +\widetilde{F}^1(\exp\pi) = F_\MC(\pi)$ and 
	$\alpha +\widetilde{F}^1(\exp\pi') = F_\MC(\pi')$ are equivalent 
	Maurer-Cartan elements.
\end{itemize}
 
We can show that Diagram~\eqref{eq:CurvedTwistDiagram} commutes up to 
homotopy:
\begin{proposition}
 The curved $L_\infty$-morphisms $(F^\pi)^1$ and $\tw_\alpha \circ (
 \Phi'_1)^{-1} \circ \tw_{-\alpha} \circ F^{\pi'} \circ \Phi_1$ 
 are homotopic.
\end{proposition}
\begin{proof}
We know from Corollary~\ref{cor:CompwithTwist} that we have 
\begin{align*}
  (F^\pi)^1
	=
	\tw_\alpha \circ \widetilde{F^\pi},
	\quad \text{ and } \quad 
	(F^{\pi'})^1
	=
	\tw_\alpha \circ \widetilde{F^{\pi'}}
\end{align*}
which implies
\begin{align*}
  \tw_\alpha \circ 
	( \Phi'_1)^{-1} \circ \tw_{-\alpha} \circ F^{\pi'} \circ \Phi_1
	& = 
	 \tw_\alpha \circ 
	( \Phi'_1)^{-1} \circ \widetilde{F^{\pi'}} \circ \Phi_1.
\end{align*}
Moreover, Proposition~\ref{prop:RelationHomEquiv} implies 
\begin{align*}
  \left((F^\pi)^1 
	\sim 
	\tw_\alpha \circ 
	( \Phi'_1)^{-1} \circ \tw_{-\alpha} \circ F^{\pi'} \circ \Phi_1 \right) 
	\quad \quad \Longleftrightarrow \quad \quad 
	\left( \widetilde{F^\pi}
	\sim 
	( \Phi'_1)^{-1} \circ \widetilde{F^{\pi'}} \circ \Phi_1 \right).
\end{align*}
But the right hand side is exactly 
Proposition~\ref{prop:TwistedLinftyIsom} since 
we know $\widetilde{F^\pi} = \widetilde{F}^\pi$ and 
$\widetilde{F^{\pi'}} = \widetilde{F}^{\pi'}$.
\end{proof}

%
%
%

%
\section{$L_\infty$-modules}
\label{sec:LinftyModules}

After having introduced the notion of $L_\infty$-algebras, we want 
to understand the basics of their representation theory, i.e. 
the notion of $L_\infty$-modules, 
see e.g. \cite{dolgushev:2005b,dolgushev:2006a,esposito.dekleijn:2021a}. 
For example, they play an important role in the formality theorem for 
Hochschild chains~\cite{dolgushev:2006a,shoikhet:2003a}. 

\subsection{Definition and First Properties}

\begin{definition}[$L_\infty$-module]
  Let $(L,Q)$ be a (curved) $L_\infty$-algebra. An \emph{$L_\infty$-module} over $(L,Q)$ 
	is a graded vector space $M$ over $\mathbb{K}$ equipped with a codifferential 
	$\phi$ of degree $1$ on the 
	cofreely cogenerated comodule $\Sym(L[1]) \otimes M$ over $(\Sym(L[1]),Q)$. 
\end{definition}

On the total space of the comodule $\Sym(L[1]) \otimes M$ one has the coaction
\begin{equation}
  a \colon 
	\Sym(L[1]) \otimes M 
	\longrightarrow 
	\Sym(L[1]) \otimes (\Sym(L[1]) \otimes M)
\end{equation}
that is defined by 
\begin{align*}
	  a(\gamma_1 \vee \cdots \vee \gamma_n \otimes m)
		& = 
		(\Delta_\sh \otimes \id) (\gamma_1 \vee \cdots \vee \gamma_n \otimes m) \\
		& =
		\sum_{k=0}^{n} \sum_{\sigma \in Sh(k,n-k)} 
		\epsilon(\sigma) \gamma_{\sigma(1)}\vee \cdots \vee \gamma_{\sigma(k)} 
		\otimes \left( \gamma_{\sigma(k+1)}\vee \cdots \vee \gamma_{\sigma(n)} 
		\otimes m\right),
\end{align*}
where $\gamma_i \in L$ and $m\in M$ are homogeneous. For example, we have 
$a(1 \otimes m)=1 \otimes 1 \otimes m$ and $a(\gamma \otimes m) = 
\gamma \otimes 1\otimes m + 1 \otimes\gamma \otimes m$. 
By the coassociativity of $\Delta_\sh$ it follows directly  
$(\id \otimes a)a(X)= (\Delta_\sh \otimes \id)a(X)$ for all 
$X\in \Sym(L[1]) \otimes M$, thus the well-definedness of the coaction. 

\begin{remark}
  Note that in the flat setting one can restrict to modules over 
	$\cc{\Sym}(L[1])$ instead of $\Sym(L[1])$. In this case 
	the sum in the definition of 
	$a$ starts at $k=1$. For example, one has then in the flat case 
	$\ker a = M$, which is not true in the curved setting, see \cite[Section~2.2]{dolgushev:2006a}.
\end{remark}

By definition, an $L_\infty$-module structure is a codifferential 
$\phi$ of $\Sym(L[1]) \otimes M$, which means
\begin{equation}
  a \circ \phi(X)
	=
	(\id \otimes \phi)(a X) + (Q \otimes \id\otimes \id) (aX).
\end{equation}
In terms of homogeneous elements this takes the form
\begin{align*}
  \begin{split}
	  \phi(\gamma_1 &\vee \cdots \vee \gamma_n \otimes m)
		= 
		Q(\gamma_1 \vee \cdots \vee \gamma_n ) \otimes m \\
		& + 
		\sum_{k=0}^n\sum_{\sigma\in Sh(k,n-k)} 
		(-1)^{\sum_{i=1}^k \abs{\gamma_{\sigma(i)}}} \epsilon(\sigma)
		\gamma_{\sigma(1)} \vee \cdots \vee \gamma_{\sigma(k)} \vee 
		\phi_{n-k}^1(\gamma_{\sigma(k+1)} \vee \cdots \vee \gamma_{\sigma(n)}\otimes m)
	\end{split}
\end{align*}
with 
\begin{equation}
  \phi_n^1 = \phi_n \colon 
	\Sym^n (L[1]) \otimes M 
	\longrightarrow 
	M[1], 
	\quad n \geq 0.
\end{equation}
In other words,
\begin{equation}
  \label{eq:ComodStructure}
  \phi
	=
	Q \otimes \id + (\id \otimes \phi^1) \circ (\Delta_\sh \otimes \id).
\end{equation}
For example, this implies
\begin{equation*}
  \phi(\gamma \otimes m)
	=
	Q(\gamma)\otimes m + 1 \otimes \phi_1^1(\gamma \otimes m) 
	+ (-1)^{\abs{\gamma}} \gamma \otimes \phi_0^1(m).
\end{equation*}
The condition $\phi^2=0$ translates to quadratic relations, see e.g.
\cite[Formula~(2.25)]{dolgushev:2006a}. We compute: 
\begin{align}
\label{eq:phisquared}
 \begin{split}
  \phi \circ \phi
	& = 
	(Q\otimes \phi^1)\circ (\Delta_\sh \otimes \id)
	+ 
	(\id \otimes \phi^1)\circ (Q \otimes \id \otimes \id + 
	\id \otimes Q \otimes \id)\circ (\Delta_\sh\otimes\id) \\
	& \quad 
	+ 	Q\circ Q \otimes \id+(\id \otimes \phi^1)\circ (\Delta_\sh\otimes \id )\circ 
	(\id \otimes \phi^1)\circ(\Delta_\sh\otimes\id)\\
	& =
	(\id \otimes \phi^1\circ (Q\otimes \id+ 
	(\id \otimes\phi^1)\circ(\Delta_\sh\otimes\id))) \circ (\Delta_\sh \otimes \id).
	\end{split}
\end{align}
In particular, this gives
\begin{equation*}
  \phi_0(1 \otimes \phi_0(1\otimes m))
	+ \phi_1((Q_0(1)\otimes m)
	=
	0
\end{equation*} 
and 
\begin{align*}
  \phi_0(1 \otimes \phi_1(\gamma \otimes m))
	+  \phi_1(Q_1(\gamma)\otimes m)
	+ \phi_2(Q_0(1)\vee \gamma \otimes m)
  +(-1)^{\abs{\gamma}} \phi_1(\gamma \otimes \phi_0(1 \otimes m))
	=
	0
\end{align*}
for $\gamma \in L[1]^{\abs{\gamma}}$. In the flat setting, i.e. if $Q_0(1)=0$, the map 
$\phi_0$ is indeed a differential on $M$ and $\phi_1$ is closed with 
respect to the induced differential on $\Hom(L[1]\otimes M ,M)$.

\begin{example}[DGLA module]
  \label{ex:dglamodule}
  The simplest example is as expected a DG module $(M,b,\rho)$ over a DGLA 
	$(\liealg{g},\D,[\argument{,}\argument])$. The only non-vanishing 
	structure maps of $\phi$ are $\phi_0 = -b$ and $\phi_1(\gamma\otimes m) 
	= -(-1)^{\abs{\gamma}}\rho(\gamma)m$, where $\rho$ is the action of 
	$\liealg{g}$ on $M$.
\end{example}

\begin{example}
  Another basic example comes from $L_\infty$-morphisms. Let $F\colon (L,Q) 
	\rightarrow (L',Q')$ be an $L_\infty$-morphism, then it induces on $L'$ 
	the structure of an $L_\infty$-module over $L$ via
	\begin{equation*}
	  \phi_k( \gamma_1 \vee \cdots \vee \gamma_k \otimes m)
		=
		Q'^1(F(\gamma_1 \vee \cdots \vee \gamma_k)\vee m)
	\end{equation*}
	for $\gamma_i \in L, m\in L'$.
\end{example}

As in the case of $L_\infty$-algebras, $L_\infty$-module structures can 
be interpreted as Maurer-Cartan elements in a convolution-like algebra:

\begin{proposition}
  Let $M$ be a graded vector space over $\mathbb{K}$ and $(L,Q)$ an 
	$L_\infty$-algebra. Then the vector space $
	\liealg{h}_{M,L}=	\Hom(\Sym(L[1])\otimes M, M)$
	of graded linear maps can be equipped with the structure of a 
	DGLA with differential
	\begin{equation}
	  \del \phi
		=
		-(-1)^{\abs{\phi}} \phi \circ (Q\otimes \id),
	\end{equation}
	where $\abs{\argument}$ denotes the degree in $\Hom(\Sym(L[1])\otimes M, M)$,
	and bracket induced by the product
	\begin{equation}
	  \label{eq:ProductLinftyModDGLA}
	  \phi \bullet \psi
		=
		\phi \circ (\id \otimes \psi)\circ (\Delta_\sh \otimes \id).
	\end{equation}
	The Maurer-Cartan elements $\phi$ of this DGLA can be identified with 
	$L_\infty$-module structures $\widehat{\phi}$ on $M$.
\end{proposition}
\begin{proof}
The product \eqref{eq:ProductLinftyModDGLA} is associative and thus induces a 
Lie bracket:
\begin{align*}
  \phi \bullet (\psi \bullet \eta)
	& = 
	\phi \circ (\id \otimes (\psi \circ (\id \otimes \eta)\circ(\Delta_\sh\otimes\id)))
	\circ (\Delta_\sh \otimes\id) \\
	& =
	\phi \circ (\id \otimes \psi)\circ (\id \otimes \id \otimes \eta) 
	(\Delta_\sh^2 \otimes \id) 
	=
	(\phi \bullet \psi)\bullet \eta.
\end{align*}
The identity $\del^2=0$ is clear. The compatibility of $\del$ with the product 
follows from
\begin{align*}
  \del (\phi \bullet \psi)
	& =
	-(-1)^{\abs{\phi}+\abs{\psi}} 
	\phi \circ (\id \otimes \psi)\circ (\Delta_\sh \otimes \id)(Q \otimes \id) \\
	& =
	-(-1)^{\abs{\phi}+\abs{\psi}} 
	\phi \circ (\id \otimes \psi)\circ (Q\otimes \id \otimes \id + 
	\id \otimes Q \otimes \id)(\Delta_\sh \otimes \id) \\
	& =
	\del \phi \bullet \psi +
	(-1)^{\abs{\phi}} \phi \bullet \del \psi.
\end{align*}
Let now $\phi$ be a Maurer-Cartan element, i.e. 
\begin{align*}
  0
	=
	\del \phi + \phi \bullet \phi
	=
  \phi \circ (Q\otimes \id) + 
	\phi \circ (\id \otimes \phi)\circ (\Delta_\sh \otimes \id) .
\end{align*}
But by \eqref{eq:phisquared} this is equivalent to $\phi$ inducing 
a codifferential $\widehat{\phi}= 
Q \otimes \id + (\id \otimes \phi) \circ (\Delta_\sh \otimes \id)$.
\end{proof}

As usual, if $L$ and $M$ are equipped with filtrations, we require the maps 
in $\liealg{h}_{M,L}$ to be compatible with the filtrations. 

\begin{remark}
  \label{rem:FiltrationModConvDGLA} 
	For a flat $L_\infty$-algebra $(L,Q)$ the DGLA 
$\liealg{h}_{M,L}$ itself has again complete descending filtration
\begin{align}
  \begin{split}
  \liealg{h}_{M,L}
	=
	& \mathcal{F}^0\liealg{h}_{M,L}
	\supset
	\mathcal{F}^1\liealg{h}_{M,L}
	\supset \cdots \supset
	\mathcal{F}^k\liealg{h}_{M,L}
	\supset \cdots \\
	\mathcal{F}^k\liealg{h}_{M,L}
	& =
	\left\{ f \in \Hom(\Sym(L[1])\otimes M,M) \mid 
	f \at{\Sym^{<k}(L[1])\otimes M}=0\right\},
	\end{split}
\end{align}
analogously to \eqref{eq:FiltrationConvLieAlg}. In the curved setting 
we have to assume that $L$ and $M$ have complete descending filtrations 
and that $Q_0(1)\in \mathcal{F}^1L$. This induces a filtration on 
$\liealg{h}_{M,L}$ analogously to \eqref{eq:FiltrationConvLieAlg2}, taking 
the filtrations on $L$ and $M$ into account.
\end{remark}

This allows us to give another equivalent definition of $L_\infty$-modules, now 
in terms of $L_\infty$-morphisms:

\begin{lemma} 
\label{lemma:LinfModLinftMorph}
Let $\phi$ be a coderivation of degree one on the 
comodule $\Sym(L[1])\otimes M$ over $(\Sym(L[1]),Q)$ with 
Taylor coefficients $\phi_n$, $n \geq 0$. Then $\phi$ 
is a codifferential, i.e. defines an $L_\infty$-module structure on $M$, 
if and only if the maps $\Phi_k\colon \Sym (L[1])\to \End(M)$ defined by 
	\begin{align}\label{eq: Ten-Hom-adj}
	\Phi(X_1\vee\dots\vee X_k)(m):=\phi_k(X_1\vee\dots\vee X_k\tensor m)
    \end{align}	
	for $X_i \in L[1], m\in M$ 
are the Taylor coefficients of an $L_\infty$-morphism, where 
the $L_\infty$-structure on $\End(M)$ is the one induced by the Lie bracket 
from Example~\ref{ex:EndDGLA} and zero differential.
\end{lemma}

\begin{proof}
First we notice that if we consider the \emph{graded} tensor-hom adjunction, we get 
	\begin{align*}
	\Hom(\Sym(L[1])\otimes M, M)\simeq \Hom (\Sym(L[1]), \End(M))
	\end{align*}
and the isomorphism is given by equation \eqref{eq: Ten-Hom-adj}. 
One can check that the induced $L_\infty$-algebra structure on $\Hom (\Sym(L[1]), \End(M))$ coincides up to a sign with the convolution algebra 
structure of the two $L_\infty$-algebras.  	
\end{proof}

Next we want to recall the definition of morphisms between $L_\infty$-modules.

\begin{definition}[Morphism of $L_\infty$-modules]
  Let $(L,Q)$ be an $L_\infty$-algebra with two $L_\infty$-modules 
	$(M,\phi^M)$ and $(N,\phi^N)$ over $L$. Then a morphism between $L_\infty$-modules 
	is a morphism $\kappa$ between the comodules $\Sym(L[1]) \otimes M$ 
	and $\Sym(L[1]) \otimes N$, i.e. 
	\begin{equation}
	  (\id \otimes \kappa) \circ a^M 
		=
		a^N \circ \kappa,
	\end{equation}
	such that $\kappa \circ \phi^M = \phi^N \circ \kappa$.
\end{definition}
One can again show that $\kappa$ is uniquely determined by its structure 
maps
\begin{equation}
  \kappa_n \colon
	\Sym^n (L[1]) \otimes M 
	\longrightarrow N
\end{equation}
via 
\begin{equation}
  \kappa(\gamma_1 \vee \cdots \vee \gamma_n \otimes m)
	=
	\sum_{k=0}^n \sum_{\sigma\in Sh(k,n-k)} \epsilon(\sigma)  
	\gamma_{\sigma(1)} \vee \cdots \vee \gamma_{\sigma(k)} \otimes 
	\kappa_{n-k}(\gamma_{\sigma(k+1)} \vee \cdots \vee \gamma_{\sigma(n)}\otimes m).
\end{equation}
The compatibility with the coderivations in the flat case implies in particular 
$\kappa_0 \circ \phi^M_0 = \phi^N_0 \circ \kappa_0$, see 
\cite[Formula~(2.29)]{dolgushev:2006a} for the general relation.

\begin{definition}
  A \emph{quasi-isomorphism} $\kappa$ of $L_\infty$-modules over flat 
	$L_\infty$-algebras is a morphism 
	with the zeroth structure map $\kappa_0$ being a quasi-isomorphism of 
	complexes $(M,\phi^M_0)$ and $(N,\phi^N_0)$.
\end{definition}

Moreover, as for $L_\infty$-morphisms between $L_\infty$-algebras 
in Proposition~\ref{prop:Linftyiso} we see that an 
$L_\infty$-module morphism is an isomorphism if and only if 
the first structure map $\kappa_0 \colon M\to N$ is an isomorphism:

\begin{proposition}
\label{prop:LinftyModIso}
  Let $\kappa \colon (M,\phi^M)\to(N,\phi^N)$ be a morphism of 
	$L_\infty$-modules over $(L,Q)$. Then $\kappa$ is an 
	$L_\infty$-module isomorphism if and only if $\kappa_0 \colon 
	M\to N$ is an isomorphism.
\end{proposition}
\begin{proof}
The proof is completely analogue to the $L_\infty$-algebra case in 
Proposition~\ref{prop:Linftyiso}.
\end{proof}

As expected, morphisms of $L_\infty$-modules can be again interpreted 
as Maurer-Cartan elements of a convolution DGLA, now even a commutative one, 
i.e. morphisms of $L_\infty$-modules are just closed elements of degree one 
in a cochain complex:

\begin{proposition}
  \label{prop:LinftymodMorphasMC}
  Let $(M,\phi^M)$ and $(N,\phi^N)$ be two $L_\infty$-modules over $(L,Q)$. 
	The vector space $\Hom(\Sym(L[1])\otimes M, N)[-1]$ can be equipped with 
	the structure of an abelian DGLA $(\Hom(\Sym(L[1])\otimes M, N)[-1], \del,0)$ 
	with differential
  \begin{equation}
    \del X
	  =
	  \phi^{N,1} \circ(\id\otimes X)\circ (\Delta_\sh \otimes \id)
	  +(-1)^{\abs{X}} X \circ \phi^M
  \end{equation}
	and zero bracket. Maurer-Cartan elements, i.e. closed elements 
	$\kappa^1$ of degree one, 
  can be identified with morphisms of $L_\infty$-modules via
  \begin{equation}
	  \label{eq:MorphofLinftyMod}
    \kappa
		=
		(\id \otimes \kappa^1)\circ(\Delta_\sh\otimes \id) \colon 
	  \Sym(L[1])\otimes M
	  \longrightarrow
	  \Sym(L[1])\otimes N.
  \end{equation}
\end{proposition}
\begin{proof}
The fact that $\del^2=0$ follows since $\phi^M$ and $\phi^N$ are 
codifferentials. Explicitly, we have 
\begin{align*}
  \del^2 X
	& =
	\phi^{N,1}\circ (\id \otimes \phi^{N,1} \circ(\id\otimes X)\circ (\Delta_\sh \otimes \id))\circ (\Delta_\sh \otimes \id ) 
	+(-1)^{\abs{X}+\abs{X}+1}X \circ \phi^M\circ\phi^M \\
	& \quad 
	+(-1)^{\abs{X}+1} \phi^{N,1} \circ (\id\otimes X)\circ(\Delta_\sh\otimes\id)\circ 
	\phi^M 
	\\
	& \quad 
	+ (-1)^{\abs{X}} \phi^{N,1} \circ (\id \otimes X\circ\phi^M)\circ(\Delta_\sh\otimes 
	\id) \\
	& =
	\phi^{N,1}\circ (\id \otimes \phi^{N,1} \circ(\id\otimes X)\circ (\Delta_\sh \otimes \id))\circ (\Delta_\sh \otimes \id ) \\
	& \quad 
	+(-1)^{\abs{X}+1} \phi^{N,1} \circ (\id\otimes X)\circ(\id \otimes \phi^M + Q \otimes\id\otimes\id)\circ(\Delta_\sh\otimes\id)
	\\
	& \quad 
	+ (-1)^{\abs{X}} \phi^{N,1} \circ (\id \otimes X\circ\phi^M)\circ(\Delta_\sh\otimes 
	\id) \\
	& =
	\phi^{N,1}\circ (\id \otimes \phi^{N,1} )\circ (\id \otimes \id 
	\otimes X) \circ (\Delta_\sh \otimes\id\otimes \id)\circ (\Delta_\sh\otimes\id)
	+\phi^{N,1}\circ (Q\otimes X)\circ(\Delta_\sh\otimes\id) \\
	& =
	\phi^{N,1}\circ (\id \otimes \phi^{N,1} ) \circ (\Delta_\sh \otimes\id)\circ 
	(\id 	\otimes X)\circ(\Delta_\sh\otimes\id)
	+\phi^{N,1}\circ (Q\otimes X)\circ(\Delta_\sh\otimes\id)
	=
	0.
\end{align*}
Moreover, \eqref{eq:MorphofLinftyMod} yields a comodule morphism since
\begin{align*}
  (\id \otimes \kappa) \circ a
	& =
	(\id \otimes( (\id \otimes \kappa^1)\circ(\Delta_\sh \otimes \id)))\circ (\Delta_\sh \otimes\id)
	=
	(\id \otimes \id \otimes \kappa^1)\circ(\id \otimes \Delta_\sh \otimes \id)
	\circ (\Delta_\sh \otimes \id) \\
	& =
	(\Delta_\sh \otimes \id) \circ (\id \otimes \kappa^1)\circ(\Delta_\sh \otimes \id)
	=
	a \circ \kappa.
\end{align*} 
If $\kappa^1$ is of degree one in the shifted complex, then the induced comodule 
morphism is of degree zero and $\del \kappa^1=0$ is equivalent to 
$\phi^N\circ \kappa=\kappa\circ \phi^M$.
\end{proof}

\subsection{Homotopy Equivalence of Morphisms of $L_\infty$-Modules}

This allows us to define a homotopy equivalence relation between morphisms of 
$L_\infty$-modules. Analogously to the case of $L_\infty$-algebras, it is 
defined via the gauge equivalence in the convolution DGLA.
Since the convolution DGLA has only the 
differential as non-vanishing structure map, there is no filtration needed for the 
well-definedness of the gauge action.

\begin{definition}
  Let $(M,\phi^M)$ and $(N,\phi^N)$ be two $L_\infty$-modules over $(L,Q)$. 
	Two morphisms $\kappa_1,\kappa_2$ of $L_\infty$-modules are called 
	\emph{homotopic} if they are gauge equivalent Maurer-Cartan elements in 
	$(\Hom(\Sym(L[1])\otimes M, N)[-1], \del,0)$, i.e. if
	\begin{equation}
    \kappa_2
	  =
	  \kappa_1- \del h
  \end{equation}
  for some $h$ of degree zero. In other words, $[\kappa_1]=[\kappa_2]$ in 
	$\mathrm{H}^1(\Hom(\Sym(L[1])\otimes M,N)[-1],\del)$.
\end{definition}

The twisting procedures from Section~\ref{subsec:Twisting} can be 
transferred to $L_\infty$-modules, see \cite[Proposition~3]{dolgushev:2006a} for the 
flat case and \cite{esposito.dekleijn:2021a} for the curved setting.

\begin{proposition}
  \label{prop:twistlinftymodules}
  Let $(L,Q)$ be an $L_\infty$-algebra with $L_\infty$-module $(M,\phi)$ 
	and let $\pi \in \Mc^1(L)$. 
	\begin{propositionlist}
	  \item For any $X \in \Sym(L[1]) \otimes M$
		      \begin{equation}
					  a(\exp(\pi \vee)X)
						=
						\exp(X\vee) \otimes \exp(\pi\vee)(aX).
					\end{equation}
					
		\item The map
		      \begin{equation}
					  \phi^\pi = \exp(-\pi \vee)\phi \exp(\pi \vee)
					\end{equation}
					is a coderivation of $\Sym(L[1])\otimes M$ along $Q^\pi$ squaring to zero.
					
		\item If $\kappa\colon M \rightarrow N$ is an $L_\infty$-morphism of 
		      $L_\infty$-modules over $L$, then 
					\begin{equation}
					  \kappa ^\pi 
						=
						\exp(-\pi \vee) \kappa \exp(\pi \vee)
					\end{equation}
					is an $L_\infty$-morphism of the twisted modules over the twisted 
					$L_\infty$-algebra.
	\end{propositionlist}
\end{proposition}
\begin{proof}
The first claim is clear. The second and the third claim follow directly. 
\end{proof}

The twisted structure maps are as expected given by
\begin{equation}
  \label{eq:twistenlmoduecoderivation}
  \phi^\pi_n(\gamma_1\vee \cdots \vee \gamma_n\otimes m)
	=
	\sum_{k=0}^\infty\frac{1}{k!}
	\phi_{n+k}(\pi\vee \cdots\vee \pi \vee \gamma_1 \vee 
	\cdots \vee \gamma_n \otimes m)
\end{equation}
and
\begin{equation}
  \label{eq:twistedlinftymodule}
  \kappa^\pi_n(\gamma_1\vee \cdots \vee \gamma_n\otimes m)
	=
	\sum_{k=0}^\infty \frac{1}{k!}
	\kappa_{n+k}(\pi\vee \cdots\vee \pi \vee \gamma_1 \vee 
	\cdots \vee \gamma_n \otimes m).
\end{equation}

We want to investigate the relation between equivalent Maurer-Cartan elements, 
i.e. equivalent $L_\infty$-module structures. To this end, assume that we 
have a complete descending filtration $\mathcal{F}^\bullet$ on the convolution DGLA $\liealg{h}_{M,L}=\Hom(\Sym(L[1])\otimes M, M)$, compare 
Remark~\ref{rem:FiltrationModConvDGLA}.

\begin{proposition}
  \label{prop:EquModulStructuresIsomorphic}
  Let $\phi_0,\phi_1 \in \mathcal{F}^1\liealg{h}_{M,L}^1$ be two equivalent 
	Maurer-Cartan elements with equivalence described by 
	\begin{equation}
	  \phi_1
		=
		e^{[h,\argument]}\phi_0 - \frac{e^{[h,\argument]}-\id}{[h,\argument]}\del h, 
		\quad \quad 
		h \in \mathcal{F}^1\liealg{h}_{M,L}^0.
	\end{equation}
	Then $A_h = (\id \otimes e^h)\circ(\Delta_\sh \otimes \id)$ is an $L_\infty$-isomorphism 
	between the $L_\infty$-modules $(\Sym(L[1])\otimes M, \widehat{\phi_0})$ 
	and $(\Sym(L[1])\otimes M, \widehat{\phi_1})$.
\end{proposition}
\begin{proof}
$A_h$ is a comodule morphism by 
Proposition~\ref{prop:LinftymodMorphasMC}.
Concerning the compatibility with the codifferentials we compute
\begin{align*}
  A_h \widehat{\phi_0}
	& =
	(\id \otimes e^h)\circ(\Delta_\sh \otimes \id)\circ (Q\otimes \id) 
	+ (\id \otimes e^h)\circ(\Delta_\sh \otimes \id)\circ(\id \otimes\phi_0)\circ 
	(\Delta_\sh \otimes \id)  \\
	& =
	(Q\otimes \id) \circ A_h + 
	(\id \otimes (-\del e^h + e^h\bullet \phi_0))\circ (\Delta_\sh \otimes \id) \\
	& =
	(Q\otimes \id) \circ A_h + 
	\left(\id \otimes \left(- \frac{e^{[h,\argument]}-\id}{[h,\argument]} 
	\del h \bullet e^h
	+ e^{[h,\argument]}\phi_0  \bullet e^h \right)\right) 	\circ(\Delta_\sh \otimes \id) \\
	& = 
	(Q\otimes \id) \circ A_h + \left(\id \otimes (\phi_1 \bullet e^h)\right)\circ
	(\Delta_\sh \otimes \id)
	=
	\widehat{\phi_1} A_h
\end{align*}
and the statement is shown.
\end{proof}

Note that an $L_\infty$-morphism $F\colon (L,Q)\rightarrow (L',Q')$ induces 
a map 
\begin{equation} 
\label{eq:PullBackLinftyMod}
  F^* \colon
	\liealg{h}_{M,L'} \ni
	\phi 
	\longmapsto 
	F^*\phi 
	=
	\phi \circ (F \otimes \id) \in
	\liealg{h}_{M,L}
\end{equation}
via the pull-back. It is compatible with the DGLA structure:

\begin{lemma}
\label{lem:PullbackLinftyMod}
  The map  $F^* \colon \liealg{h}_{M,L'} \rightarrow \liealg{h}_{M,L}$ is a 
	DGLA morphism.
\end{lemma}
\begin{proof}
Concerning the differential we get 
\begin{align*}
  F^*\del' \phi
  =
	-(-1)^{\abs{\phi}} \phi \circ (Q'\otimes \id)\circ (F\otimes \id)
	=
	-(-1)^{\abs{\phi}} \phi \circ (F\otimes \id)\circ (Q\otimes \id)
	=
	\del F^*\phi.
\end{align*}
Similarly, we get for the product $\bullet$
\begin{align*}
  F^*(\phi \bullet \psi)
	& =
	F^*(\phi \circ (\id \otimes \psi)\circ (\Delta_\sh \otimes \id)) 
	= 
	(\phi \circ (\id \otimes \psi)\circ (F\otimes F \otimes \id) 	\circ 
	(\Delta_\sh \otimes \id)\\
	& = F^*\phi \bullet F^*\psi
\end{align*}
and the statement is shown.
\end{proof}

\begin{remark}
  Using Lemma~\ref{lemma:LinfModLinftMorph} we can identify 
	the DGLA $\liealg{h}_{M,L}=	\Hom(\Sym(L[1])\otimes M, M)$ with 
	$\Hom (\Sym(L[1]), \End(M)) $ equipped with the convolution 
	$L_\infty$-structure. Under this identification the above 
	pull-back is indeed just the usual pull-back. All the constructions below 
	work in both points of view, however, we formulate them in terms of 
	$\liealg{h}_{M,L}$.
\end{remark}

We want to investigate the relation between pull-backs with homotopic 
$L_\infty$-morphisms, where 
we assume for simplicity that our $L_\infty$-algebras are flat.  
Let therefore $F(t)$ and $\lambda(t)$ encode the homotopy equivalence between 
two $L_\infty$-morphisms from $L$ to $L'$ in the convolution $L_\infty$-algebra 
$(\Hom(\cc{\Sym}(L[1]),	L'),\widehat{Q})$ with 
\begin{equation*}
  \frac{\D}{\D t} F^1(t)
  =
	\widehat{Q}(\lambda^1(t) \vee \exp(F^1(t))),
\end{equation*}
compare Definition~\ref{def:homotopicMorph}.
This can be rewritten in the following way: 

\begin{proposition}
  \label{prop:CoalgHomotopyGamma}
	The map $\Gamma =	\lambda^1(t) \star F(t)= 
	\vee \circ (\lambda^1(t) \otimes F(t)) \circ 
	\Delta_\sh$ satisfies
	\begin{equation}
	  \frac{\D}{\D t} F(t)
		=
		Q' \circ \Gamma + \Gamma \circ Q
	\end{equation}
	and
	\begin{equation}
	  \label{eq:GammaCoderivation}
	  \Delta_\sh \circ \Gamma
		=
		(\Gamma \otimes F + F \otimes \Gamma) \circ \Delta_\sh.
	\end{equation}
\end{proposition}
\begin{proof}
By Theorem~\ref{thm:CofreeCocomConilpotentCoalg} we know that $\Gamma$ is a coderivation 
along $F$ and \eqref{eq:GammaCoderivation} is clear. 
Moreover, both $\frac{\D}{\D t}F$ and 
$(Q' \circ \Gamma + \Gamma \circ Q)$ are coderivations along $F$ and thus it 
suffices to show
\begin{equation*}
  \frac{\D }{\D t} F^1(t)
	=
	Q'^1 \circ \Gamma + \lambda^1 \circ Q.
\end{equation*}
But we know
\begin{align*}
  \frac{\D}{\D t} F^1(t)
  & =
	\widehat{Q}^1(\lambda^1(t) \vee \exp(F^1(t)) 
	= 
	\sum_{i=1}^\infty  \frac{1}{(i-1)!} Q'^1_i \circ (\lambda^1(t) \star 
	F^1(t)\star\cdots \star F^1(t))
	+ \lambda^1(t) \circ Q \\
	& =
	Q'^1 \circ \Gamma + \lambda^1 \circ Q
\end{align*}
and the proposition is shown.
\end{proof}

This allows us to show that homotopic $L_\infty$-morphisms give homotopic pull-back morphisms.

\begin{theorem}
  Let $F_0$ and $F_1$ be two homotopic $L_\infty$-morphisms 
	between the flat $L_\infty$-algebras $(L,Q)$ to $(L',Q')$ and assume that 
	$\liealg{h}_{M,L'}$ and $\liealg{h}_{M,L}$ are equipped with complete filtrations 
	as described in Remark~\ref{rem:FiltrationModConvDGLA}. 
	Then the induced DGLA morphisms 
	\begin{equation}
	  F_0^*, F_1^* \colon 
		\liealg{h}_{M,L'} 
		\longrightarrow
		\liealg{h}_{M,L}
	\end{equation}
	are homotopic.
\end{theorem}
\begin{proof}
Let $F(t)$ and $\lambda^1(t)$ encode the homotopy equivalence between 
$F_0$ and $F_1$. We show that 
\begin{equation*}
  \frac{\D}{\D t} F_t^*
	=
	\widehat{Q}^1(\Gamma_t^* \vee \exp F_t^*).
\end{equation*}
Here $\Gamma_t^* \phi = (-1)^{\abs{\phi}} \phi \circ (\Gamma(t) \otimes \id)$, 
$F_t^*\phi = \phi \circ (F(t)\otimes\id)$, where $\abs{\phi}$ denotes now 
the degree in $\liealg{h}_{M,L'}[1]$. Moreover,
$\widehat{Q}$ is the codifferential on the convolution DGLA
$\Sym(\Hom(\cc{\Sym}(\liealg{h}_{M,L'}[1]),\liealg{h}_{M,L})[1])$. It is given 
by 
\begin{equation*}
  \widehat{Q}_1^1 X
	=
	- \del X - (-1)^{\abs{X}} X \circ Q_{\liealg{h}_{M,L'}}
\end{equation*}
and the bracket is induced by the bracket on $\liealg{h}_{M,L}$.
Using Proposition~\ref{prop:CoalgHomotopyGamma}, we have 
for $\phi \in \liealg{h}_{M,L'}$
\begin{align*}
  \frac{\D}{\D t} F_t^*\phi
	=
	\frac{\D}{\D t}  \phi \circ (F(t) \otimes \id) 
	=
	\phi \circ
	((Q' \circ \Gamma + \Gamma \circ Q) \otimes\id).
\end{align*}
Moreover, we get
\begin{align*}
  (\widehat{Q}^1_1\Gamma_t^*)^1_1 \phi
	=
	- \del \Gamma_t^* \phi + \Gamma_t^*\circ (-\del')\phi  
	=
	\phi \circ ((\Gamma \circ Q + Q' \circ\Gamma )\otimes \id)  .
\end{align*}
The higher orders of $\widehat{Q}^1(\Gamma_t^* \vee \exp F_t^*)$ are given by 
\begin{align*}
  \Gamma_t^* \circ Q_{\liealg{h}_{M,L'},2}^1
	+ Q_{\liealg{h}_{M,L},2}^1 \circ (\Gamma_t^* \vee F_t^*) \circ \Delta_\sh.
\end{align*}
We can compute
\begin{align*}
  \Gamma_t^*(\phi \bullet \psi)
	& = 
	(-1)^{\abs{\phi} +\abs{\psi}+1} \phi \circ (\id \otimes \psi)\circ 
	(\Delta_\sh \otimes \id)\circ (\Gamma \otimes \id) \\
	& =
	(-1)^{\abs{\phi} +\abs{\psi}+1} \phi \circ (\id \otimes \psi)\circ 
	((\Gamma \otimes F + F \otimes \Gamma) \otimes \id)(\Delta_\sh \otimes \id) \\
	& =
	\Gamma_t^*\phi \bullet F_t^*\psi
	- (-1)^{\abs{\phi}} F_t^*\phi \bullet \Gamma_t^* \psi
\end{align*}
which yields
\begin{align*}
   \Gamma_t^* \circ Q_{\liealg{h}_{M,L'},2}^1 (\phi \vee \psi)
	& =
	-(-1)^{\abs{\phi}} 
	\Gamma_t^*(\phi \bullet \psi - (-1)^{(\abs{\phi}-1)(\abs{\psi}-1)} 
	\psi \bullet \phi) \\
	& =
	-(-1)^{\abs{\phi}} \left(\Gamma_t^*\phi \bullet F_t^*\psi
	- (-1)^{\abs{\phi}} F_t^*\phi \bullet \Gamma_t^* \psi\right) \\
	& \quad
	- (-1)^{(\abs{\phi} -1)\abs{\psi}}\left(
	\Gamma_t^*\psi \bullet F_t^*\phi
	- (-1)^{\abs{\psi}} F_t^*\psi \bullet \Gamma_t^* \phi
	\right) \\
	& =
	-Q_{\liealg{h}_{M,L},2}^1 (\Gamma_t^*\phi \vee F_t^*\psi)
	-(-1)^{\abs{\phi}\abs{\psi}}Q_{\liealg{h}_{M,L},2}^1 
	(\Gamma_t^*\psi \vee F_t^*\phi)  \\
	& =
	- Q_{\liealg{h}_{M,L},2}^1 \circ (\Gamma_t^* \vee F_t^*) \circ \Delta_\sh
	(\phi \vee \psi).
\end{align*}
Thus the higher terms cancel each other. We only have to check that 
$F_t^*$ and $\Gamma_t^*$ are in the completion of the polynomials in $t$ 
with respect to the filtration. But this is clear since $F(t)$ and 
$\Gamma(t)$ have these properties.
\end{proof}

With Proposition~\ref{prop:EquModulStructuresIsomorphic} this immediately 
yields:

\begin{corollary}
  The pull-backs of an $L_\infty$-module structure 
	$\phi \in \mathcal{F}^1\mathfrak{h}^1_{M,L'}$ on $M$ over $L'$ 
	via two homotopic $L_\infty$-morphisms $F_0,F_1\colon (L,Q) 
	\rightarrow (L',Q')$ yield two isomorphic $L_\infty$-module 
	structures on $M$ over $L$.
\end{corollary}

\begin{remark}
This statement is similar to the following theorem from algebraic topology, 
see e.g. \cite[Theorem~2.1]{cohen:script}: Let $p \colon E \rightarrow B$ be a fiber 
bundle and let $f_0, f_1 \colon X \rightarrow B$ be two homotopic maps. Then the pull-back bundles 
are isomorphic. It would be interesting to see if there are other statements that can be transferred 
to our algebraic setting.
\end{remark}

Analogously to our computations in Section~\ref{sec:HomEquivTwistedMorphisms} 
for the case of flat $L_\infty$-algebras and 
Section~\ref{sec:HomTheoryofCurvedLinfty} for curved 
$L_\infty$-algebras, we want to show now that $L_\infty$-module 
morphisms that are twisted with equivalent Maurer-Cartan elements are 
homotopic.

At first, one can check that one has a version of 
Proposition~\ref{prop:TwistMorphHomEqu} for DGLAs $(\liealg{g},\D,[\argument{,}\argument])$ 
and DGLA-modules $(M,b,\rho)$ and $(N,b',\rho')$ with complete filtrations. Let $\pi \in \mathcal{F}^1\liealg{g}^1$ be equivalent to zero via $\pi = \exp([A(1),\argument]) \acts 0$ with 
$g\in \liealg{g}^0$. Then
\begin{equation}
\label{eq:DGLAModTwistEquMC}
  \kappa(t)
	=
  e^{-\rho'(A(t))} \circ 
	\kappa^{\pi(t),1}\circ\left(e^{[A(t),\argument]}\otimes e^{\rho(A(t))} \right)
\end{equation}
is a path between the $L_\infty$-module morphism $\kappa$ and the twisted 
morphism $e^{-\rho'(A(1))} \circ \kappa^{\pi(1),1}\circ\left(e^{[A(1),\argument]}\otimes 
e^{\rho(A(1))} \right)$. Moreover, it satisfies
\begin{equation*}
  \frac{\D}{\D t} \kappa(t)
	=
  \del \left(
	 e^{-\rho'(A(t))} \circ 
	\kappa^{\pi(t),1}(\lambda(t) \vee \argument)
	\circ\left(e^{[A(t),\argument]}\otimes e^{\rho(A(t))} \right)\right).
\end{equation*}

This can be generalized to general $L_\infty$-modules over 
$L_\infty$-algebras. Let $(L,Q)$ be a (flat) $L_\infty$-algebra 
and $\pi \in \mathcal{F}^1L^1$ be a Maurer-Cartan element equivalent 
to zero via $\pi(t),\lambda(t)$. Let us assume for simplicity that 
$\lambda(t)=\lambda$ is constant, which is possible by 
Remark~\ref{rem:QuillenandGaugeEquivalence}.
Let moreover $\kappa \colon (M,\phi)\to (N,\phi')$ be an $L_\infty$-module 
morphism of $L_\infty$-modules over $(L,Q)$. Then we know from 
Proposition~\ref{prop:twistlinftymodules} that $\kappa^\pi \colon 
(M,\phi^\pi)\to (N,\phi'^\pi)$ is an $L_\infty$-module morphism of 
$L_\infty$-modules over $(L,Q^\pi)$. 

By Proposition~\ref{lemma:PhitLinftyMorph} we have an $L_\infty$-isomorphism 
$\Phi_t \colon (L,Q) \to (L,Q^{\pi(t)})$ satisfying
\begin{equation*}
  \frac{\D}{\D t} \Phi_t
	= 
  \left(
	  [Q^{\pi(t)}, \lambda\vee \argument] - 
		Q^{\pi(t)}(\lambda)\vee \argument
 \right) \circ \Phi_t  
\end{equation*}
By Lemma~\ref{lem:PullbackLinftyMod} we can use $\Phi_t$ to pull-back 
the $L_\infty$-module structures $\phi^{\pi(t)}$ and $\phi'^{\pi(t)}$ 
on $M$ resp. $N$. Analogously, one can pull-back $\kappa^{\pi(t)}$ and 
one obtains an $L_\infty$-module morphism
\begin{equation*}
  \Phi_t^* \kappa^{\pi(t)} 
  \colon 
  (M,\Phi_t^* \phi^{\pi(t)}) 
  \longrightarrow
	(N, \Phi_t^*\phi'^{\pi(t)})
\end{equation*}
over $(L,Q)$, where $(\Phi_t^*\kappa^{\pi(t)})^1 = 
(\kappa^{\pi(t)})^1 \circ (\Phi_t \otimes \id)$, analogously to 
the case \eqref{eq:PullBackLinftyMod} for the codifferential.

\begin{proposition}
  In the above setting there exists an $L_\infty$-module isomorphism 
	$\Psi_{M,t} \colon (M,\phi) \to (M,\Phi_t^* \phi^{\pi(t)})$ over 
	$(L,Q)$ with structure maps 
	$(\Psi_t)_i^1 \colon \Sym^i(L[1])\otimes M \to M$ for $i\geq 0$ 
	defined by
	\begin{equation}
		\label{eq:PsitTwistedMods}
		\frac{\D}{\D t} (\Psi_t)^1_i 
		=
		\left((\Phi_t^*\phi^{\pi(t)}) \circ (\lambda \vee \argument) 
		\right)^1
		\circ (\Psi_t)_i,
		\quad
		\quad
		(\Psi_0)^1_i = \pr_M.
	\end{equation}
\end{proposition}
\begin{proof}
The well-definedness of $\Psi_t$ follows analogously to the 
well-definedness of $\Phi_t$ in Lemma~\ref{lem:MorphPhitTwistedLinftyAlgs}. 
It remains to show that it is indeed a $L_\infty$-module morphism. 
To this end, we compute for $X \in \Sym^i(L[1])$, $m\in M$
\begin{align*}
  \frac{\D}{\D t} & (\Phi_t^*\phi^{\pi(t)})^1 \circ 
	\Psi_t (X\otimes m)
	=
	 \frac{\D}{\D t} \phi^1 \circ (\exp(\pi(t))\vee \Phi_t \otimes \id_M) 
	\circ (X^{(1)} \otimes \Psi_t^1 (X^{(2)} \otimes m))  \\
	& =
	(\phi^{\pi(t)})^1 \circ ( Q^{\pi(t)}(\lambda)\vee 
	\Phi_t \otimes \id_M) 
	\circ (X^{(1)} \otimes \Psi_t^1 (X^{(2)} \otimes m)) \\
	&  \quad
	+(\phi^{\pi(t)})^1 \circ ( \left(
	  [Q^{\pi(t)}, \lambda\vee \argument] - 
		Q^{\pi(t)}(\lambda)\vee \argument
 \right) \circ \Phi_t  \otimes \id_M)  \circ (X^{(1)} 
\otimes \Psi_t^1 (X^{(2)} \otimes m))   \\
  & \quad
  +
	(\phi^{\pi(t)})^1 \circ ( \Phi_t \otimes \id_M) 
	\circ \left(X^{(1)} \otimes\left((\Phi_t^*\phi^{\pi(t)}) \circ 
	(\lambda\vee \argument) 	\right)^1
		\circ ( X^{(2)} \otimes \Psi_t^1 (X^{(3)} \otimes m))\right) \\
		& = 
	  (\phi^{\pi(t)})^1 \circ \left( \left(
	  [Q^{\pi(t)}, \lambda\vee \argument] \Phi_tX^{(1)}\right)
		\otimes \Psi_t^1 (X^{(2)} \otimes m)\right)   \\
  & \quad
  +
	(\phi^{\pi(t)})^1 \circ  
	\left(\Phi_t X^{(1)} \otimes\left((\phi^{\pi(t)})^1 \circ 
	\Phi_t(\lambda \vee  X^{(2)} )\otimes \Psi_t^1 (X^{(3)} \otimes m)
	\right) \right).
\end{align*}
Using $(\Phi_t^*\phi^{\pi(t)})^2 (\lambda \vee X^{(1)} \otimes 
\Psi_t^1(X^{(2)}\otimes m))=0$ and the fact that 
$\Phi_t \circ(\lambda \vee\argument) = (\lambda \vee \argument)\circ 
\Phi_t$, we see that this coincides with
\begin{align*}
  \frac{\D}{\D t} (\Phi_t^*\phi^{\pi(t)})^1 \circ 
	\Psi_t (X\otimes m)
	& = 
	(\phi^{\pi(t)})^1 \circ ( \Phi_t \otimes \id ) 
	\circ (\lambda \vee \argument) \circ
	\big( Q(X^{(1)}) \otimes \Psi_t^1(X^{(2)}\otimes m) \\
	& \quad
	+ X^{(1)} \otimes (\phi^{\pi(t)})^1 (\Phi_t X^{(2)} \otimes
	\Psi_t^1(X^{(3)}\otimes m))\big) \\
	& =
	\left((\Phi_t^*\phi^{\pi(t)}) \circ (\lambda \vee \argument) 
		\right)^1
		\circ (\Phi_t^*\phi^{\pi(t)}) \circ 
	\Psi_t (X\otimes m)	.
\end{align*}
Therefore, $(\Phi_t^*\phi^{\pi(t)})^1 \circ 
\Psi_t (X\otimes m)$ coincides with $\Psi^1_t \circ \phi (X \otimes m)$ 
since 
\begin{align*}
  \frac{\D}{\D t} \Psi^1_t \circ \phi (X \otimes m)
	= 
	\left((\Phi_t^*\phi^{\pi(t)}) \circ (\lambda \vee \argument) 
	\right)^1 \circ (\Psi_t \circ \phi (X\otimes m)),
\end{align*}
i.e. both expressions satisfy the same differential equation, and 
both coincide at $t=0$.
The fact that $\Psi_t$ is even an $L_\infty$-module isomorphism follows 
since $(\Psi_t)_0\colon M \to M$ is an isomorphism compare 
Proposition~\ref{prop:LinftyModIso}.
\end{proof}

This allows us to generalize the case of DGLA modules from 
\eqref{eq:DGLAModTwistEquMC} and we can show:

\begin{proposition}
  The $L_\infty$-module morphisms $\kappa$ and 
	$\Psi_{N,1}^{-1} \circ (\Phi^*_1\kappa^{\pi}) \circ \Psi_{M,1}$ 
	from $(M,\phi)$ to $(N,\phi')$ over $(L,Q)$ are homotopic.
\end{proposition}
\begin{proof}
We set $\kappa(t)= \Psi_{N,t}^{-1}\circ 
(\Phi_t^*\kappa^{\pi(t)})\circ \Psi_{M,t}$ and compute
\begin{align*}
  \frac{\D}{\D t}& \Psi_{M,t}
	= 
	\left(\id \otimes (\Phi^*_t\phi^{\pi(t)})^1 \right)\circ
	( \id \otimes (\lambda \vee \argument)\otimes\id) \circ 
	(\id \otimes \Psi_{M,t}) \circ 
	(\Delta_\sh\otimes \id)  \\
	& =
	\left(\id \otimes (\Phi^*_t\phi^{\pi(t)})^1 \right)\circ
	(\Delta_\sh\otimes \id)  \circ \lambda \vee \Psi_{M,t}
	+ \left((\lambda\vee \argument) \otimes (\Phi^*_t\phi^{\pi(t)})^1 \right)
	\circ (\Delta_\sh\otimes \id)\circ \Psi_{M,t} \\
	& =
	(\Phi^*_t\phi^{\pi(t)}) \circ \lambda \vee \Psi_{M,t} 
	- Q \circ \lambda \vee \Psi_{M,t} 
	+ (\lambda\vee \argument) \circ (\Phi^*_t\phi^{\pi(t)})
	\circ \Psi_{M,t} 
	- \lambda \vee Q \circ \Psi_{M,t}  \\
	& =
	(\Phi^*_t\phi^{\pi(t)})   \circ \lambda \vee \Psi_{M,t}
	+ (\lambda\vee \argument) \circ (\Phi^*_t\phi^{\pi(t)})
	\circ \Psi_{M,t} 
	- [Q,\lambda \vee \argument] \circ \Psi_{M,t}
\end{align*}
and similarly
\begin{align*}
  \frac{\D}{\D t} \Psi_{N,t}^{-1}
	& =
	- \Psi_{N,t}^{-1} \circ( \Phi^*_t\phi'^{\pi(t)})   
	\circ (\lambda \vee \argument)
	- \Psi_{N,t}^{-1} \circ (\lambda\vee \argument) \circ 
	(\Phi^*_t\phi'^{\pi(t)})
	+	\Psi_{N,t}^{-1} \circ [Q,\lambda \vee \argument] .
\end{align*}
In addition, we have
\begin{align*}
  \frac{\D}{\D t}& \Phi_t^*\kappa^{\pi(t)}
	=
	\frac{\D}{\D t} (\id \otimes \kappa^1)\circ
	(\id \otimes \exp(\pi(t))\vee \Phi_t \otimes \id)\circ 
	(\Delta_\sh\otimes\id) \\
	& =
	 (\id \otimes (\kappa^{\pi(t)})^1)\circ
	(\id \otimes Q^{\pi(t)}(\lambda) \vee \Phi_t \otimes \id)\circ 
	(\Delta_\sh\otimes\id) \\
	& \quad
	+
	(\id \otimes (\kappa^{\pi(t)})^1)\circ
	\left(\id \otimes \left(
	  [Q^{\pi(t)}, \lambda\vee \argument] - 
		Q^{\pi(t)}(\lambda)\vee \argument
 \right) \circ \Phi_t   \otimes \id\right)\circ 
	(\Delta_\sh\otimes\id)\\
	& =
	(\id \otimes (\kappa^{\pi(t)})^1)\circ
	\left(\id \otimes  [Q^{\pi(t)}, \lambda\vee \argument] \circ 
	\Phi_t   \otimes \id\right)\circ (\Delta_\sh\otimes\id).
\end{align*}
With $\Phi_t \circ (\lambda \vee \argument) = (\lambda\vee \argument)\circ 
\Phi_t$ and $Q^{\pi(t)}\circ \Phi_t = \Phi_t \circ Q$, this gives
\begin{align*}
  \frac{\D}{\D t}& \Phi_t^*\kappa^{\pi(t)}
	=
	(\id \otimes (\kappa^{\pi(t)})^1)\circ
	\left(\id \otimes \Phi_t \circ [Q, \lambda\vee \argument]  
	 \otimes \id\right)\circ (\Delta_\sh\otimes\id) \\
	& =
	(\Phi_t^*\kappa^{\pi(t)}) \circ [Q,\lambda\vee \argument] 
	- 
	[Q,\lambda\vee \argument] \circ (\Phi_t^*\kappa^{\pi(t)}).
\end{align*}
Summarizing, we have shown
\begin{align*}
  \frac{\D}{\D t} \kappa(t)
	& =
	\phi'\circ \left( - \Psi_{N,t}^{-1}\circ (\lambda\vee\argument)\circ 
	(\Phi_t^*\kappa^{\pi(t)})\circ \Psi_{M,t} 
	+
	\Psi_{N,t}^{-1}\circ 
	(\Phi_t^*\kappa^{\pi(t)})\circ(\lambda\vee\argument) \circ  \Psi_{M,t} 
	\right) \\
	& \quad
	+
	\left( - \Psi_{N,t}^{-1}\circ (\lambda\vee\argument)\circ 
	(\Phi_t^*\kappa^{\pi(t)})\circ \Psi_{M,t} 
	+
	\Psi_{N,t}^{-1}\circ 
	(\Phi_t^*\kappa^{\pi(t)})\circ(\lambda\vee\argument) \circ  \Psi_{M,t} 
	\right) \circ \phi
\end{align*}
which implies
\begin{align*}
  \frac{\D}{\D t} \kappa^1(t)
	& =
	\del \left( - (\Psi_{N,t}^{-1})^1\circ (\lambda\vee\argument)\circ 
	(\Phi_t^*\kappa^{\pi(t)})\circ \Psi_{M,t} 
	+
	(\Psi_{N,t}^{-1})^1\circ 
	(\Phi_t^*\kappa^{\pi(t)})\circ(\lambda\vee\argument) \circ  \Psi_{M,t} 
	\right)  
\end{align*}
as desired.
\end{proof}

\begin{remark}[Application to Deformation Quantization II]
  These results imply that Dolgushev's globalizations 
	\cite{dolgushev:2006a} of Shoikhet's formality for Hochschild 
	chains \cite{shoikhet:2003a} with respect to different covariant 
	derivatives are homotopic, completely analogously to the 
	cochain case in \cite{kraft.schnitzer:2021a:pre}. 
\end{remark}

%
\bibliographystyle{nchairx}

\end{document}